\documentclass[12pt]{book}

\usepackage[english]{babel}
\usepackage{csquotes}

\usepackage{amsmath, amsfonts, amsthm, amssymb, mathtools}

\usepackage{emptypage}

\usepackage{tikz-cd}

\usepackage[hidelinks]{hyperref}
\hypersetup{
  pdftitle={Factor-Closure of Pro-Nilsystems via Local Rigidity of Nilsystems},
  pdfauthor={Pauwel Van Den Eeckhaut}
}
\usepackage[capitalise]{cleveref}
\usepackage[shortlabels]{enumitem}
\setlist{nosep}

\usepackage{MA_Titlepage}

\usepackage[
  a4paper,
  inner=3.6cm,
  outer=3.0cm,
  top=3.3cm,
  bottom=3.3cm,
  headheight=14pt
]{geometry}
\usepackage[T1]{fontenc}
\usepackage{libertinus}
\usepackage[libertinus]{newtxmath}

\usepackage{microtype}
\usepackage{xcolor}

\usepackage[url=false, doi=false, natbib=true]{biblatex}
\newtheoremstyle{thesisplain}{\parskip}{\parskip}{\itshape}{}{\bfseries}{.}{.5em}{}
\newtheoremstyle{thesisdefinition}{\parskip}{\parskip}{\itshape}{}{\normalfont}{.}{.5em}{}

\theoremstyle{thesisplain}
\newtheorem{theorem}{Theorem}[chapter]
\newtheorem*{theorem*}{Theorem}
\newtheorem{corollary}[theorem]{Corollary}
\newtheorem{proposition}[theorem]{Proposition}
\newtheorem*{proposition*}{Proposition}
\newtheorem{lemma}[theorem]{Lemma}

\theoremstyle{thesisdefinition}
\theoremstyle{definition} 
\newtheorem{definition}[theorem]{Definition}
\newtheorem{example}[theorem]{Example}
\newtheorem{remark}[theorem]{Remark}
\newtheorem*{assumption}{Standing assumption}

\authornew{Pauwel Van Den Eeckhaut}
\geburtsdatum{August 13, 2003}
\geburtsort{Aalst, Belgium} 
\date{August 11, 2026}

\betreuer{Advisor: Prof. Dr. Asgar Jamneshan}
\zweitgutachter{Second Advisor: Prof. Dr. Christoph Thiele}

\institut{Mathematisches Institut}
\title{Factor-Closure of Pro-Nilsystems via Local Rigidity of Nilsystems}
\ausarbeitungstyp{Master's Thesis  Mathematics}

\newcommand\C{\ensuremath{\mathbb{C}}}
\newcommand\R{\ensuremath{\mathbb{R}}}
\newcommand\Q{\ensuremath{\mathbb{Q}}}
\newcommand\Z{\ensuremath{\mathbb{Z}}}
\newcommand\N{\ensuremath{\mathbb{N}}}

\newcommand\T{\ensuremath{\mathbb{T}}}
\newcommand\D{\ensuremath{\mathbb{D}}}

\newcommand\tp{\ensuremath{\mathrm{top}}}

\newcommand{\aee}{\text{-a.e.}}

\newcommand{\wt}{\widetilde}
\newcommand{\wh}{\widehat}

\newcommand\mcx{\ensuremath{\mathcal{X}}}
\newcommand\mcy{\ensuremath{\mathcal{Y}}}
\newcommand\mcz{\ensuremath{\mathcal{Z}}}
\newcommand\mcw{\ensuremath{\mathcal{W}}}
\newcommand\mca{\ensuremath{\mathcal{A}}}
\newcommand\mcb{\ensuremath{\mathcal{B}}}
\newcommand\mcc{\ensuremath{\mathcal{C}}}
\newcommand\mcd{\ensuremath{\mathcal{D}}}
\newcommand\mcf{\ensuremath{\mathcal{F}}}

\newcommand\si{\ensuremath{^{(i)}}}

\newcommand{\dd}{\, \mathrm{d}}

\newcommand{\ind}[1]{\mathbf{1}_{#1}}

\DeclareMathOperator{\Id}{Id}

\DeclareMathOperator{\pr}{Pr}
\DeclareMathOperator{\Poly}{Poly}

\DeclareMathOperator{\supp}{supp}
\DeclareMathOperator{\Span}{span}
\DeclareMathOperator{\diam}{diam}

\DeclareMathOperator{\E}{\mathbb{E}}

\let\phi\varphi
\let\epsilon\varepsilon

\newcommand\restrict[2]{{
  \left.\kern-\nulldelimiterspace 
  #1 
  \right|_{#2} 
  }}

\begin{document}

\maketitle

\frontmatter
\pagenumbering{gobble}


\cleardoublepage
\chapter*{Acknowledgments}

I would like to thank my advisor, Asgar Jamneshan, for introducing me to the topic of this thesis and for his mathematical guidance as the work developed. 
Working together on the paper underlying the main results of this thesis has taught me a great deal, especially about writing and presenting mathematics.
I am also grateful for his careful reading of this thesis and for his detailed comments. 

I want to thank Christoph Thiele for his interest in this work, for his thoughtful questions, and in particular for suggesting that I work out the material in \cref{chap:weak_structure_thrm}, which considerably strengthened the thesis. 

Finally, I would like to thank my family and friends, and in particular my girlfriend, for their patience and support throughout this project.  

\vspace{1.5cm}
\begin{flushright}
\emph{Bonn, August 2026}\\
Pauwel Van Den Eeckhaut
\end{flushright}

\tableofcontents

\mainmatter

\chapter{Introduction}

Pro-nilsystems arise naturally in ergodic theory through the study of multiple ergodic averages of the form
\[
\frac{1}{N} \sum_{n=1}^N T^n f_1 \cdot T^{2n} f_2 \cdots T^{sn} f_s,  
\]
where $(X, \mu, T)$ is a measure-preserving system and $f_1, \dots, f_s \in L^\infty(X, \mu)$. 
Such averages were originally introduced by \citet{furstenberg1977ergodic} in his ergodic-theoretic proof of Szemerédi's theorem, which states that every subset of the integers of positive upper density contains arbitrarily long arithmetic progressions.  
Furstenberg showed, through what is now known as the Furstenberg correspondence principle, that Szemerédi's theorem is equivalent to a multiple recurrence statement for measure-preserving systems. 
He then proved this statement by considering the above averages in the case where each $f_j$ is the indicator function of a fixed set. 
The question of convergence in $L^2(X, \mu)$ of these averages for general $s \geq 1$ remained open for nearly three decades. 
It was eventually answered affirmatively by \citet{host2005nonconventional}, and independently by \citet{ziegler2007universal}, after earlier work on special cases by \citet{furstenberg1977ergodic} and by Conze and Lesigne \citep{conze1984theoremes,conze1987theoreme,conze1988theoreme}.

The proof by Host and Kra established far more than convergence. 
They developed a theory of cubical structures and associated seminorms, now called the Gowers--Host--Kra seminorms.
For an ergodic measure-preserving system $(X, \mu, T)$ and $f \in L^\infty(X, \mu)$, these seminorms can be defined recursively by 
\[
\|f \|_{U^1(X)} := \left| \int_X f \dd\mu \right|, 
\quad 
\|f \|_{U^{s+1}(X)}^{2^{s+1}} := \lim_{N \to \infty} \frac{1}{N} \sum_{n=1}^N \|\overline{f} \cdot T^n f \|_{U^s(X)}^{2^s}. 
\]
The fact that these limits exist and define seminorms is part of the theory; see \citep[Chapter 15]{eisner2025journey}. 
An equivalent definition using averages over so-called dynamical parallelepipeds is recalled in \cref{chap:weak_structure_thrm}.  

Using these seminorms, Host and Kra constructed for every ergodic system $(X, \mu, T)$ a sequence of factors
\[
\cdots \longrightarrow Z_s \longrightarrow Z_{s-1} \longrightarrow \cdots \longrightarrow Z_1,  
\]
called the structure factors. 
The factor $Z_s$ is characterized by the property that, for $f \in L^\infty(X, \mu)$, 
\[
\E_\mu(f \mid Z_s) = 0 \quad \Longleftrightarrow \quad  \| f \|_{U^{s+1}(X)} = 0. 
\]
The first factor $Z_1$ agrees with the classical Kronecker factor, that is, the maximal factor isomorphic to a compact abelian group rotation. 
The higher structure factors may then be viewed as non-abelian generalizations of it. 
An ergodic system is said to be of order $s$ if it agrees with its structure factor $Z_s$, equivalently if the seminorm $\| \cdot \|_{U^{s+1}(X)}$ is a norm on $L^\infty(X, \mu)$. 
In particular, the structure factor $Z_s$ is itself of order $s$. 

The structure factors are characteristic for multiple ergodic averages, in the sense that the $s$-term averages above are controlled by the factor $Z_{s-1}$.
More precisely, 
\[
\lim_{N \to \infty} 
\left\| 
    \frac{1}{N} \sum_{n=1}^N \left(
    \prod_{j=1}^s T^{jn} f_j - \prod_{j=1}^s T^{jn} \E_\mu(f_j \mid Z_{s-1})
    \right)
\right\|_{L^2(\mu)} = 0,  
\]
so that each $f_j$ may be replaced by its conditional expectation onto $Z_{s-1}$ without changing the asymptotic behavior of the average. 
This reduces the question of convergence from arbitrary systems to the structure factors.  
Independently, \citet{ziegler2007universal} constructed universal characteristic factors by a different method. 
The equivalence of these two descriptions was later shown by \citet{leibman2005host}. 

The power of this theory lies in the fact that the structure factors admit a concrete algebraic description. 
An $s$-step nilsystem is a measure-preserving system $(X, \mu, T)$ where $X = G/\Gamma$ is an $s$-step nilmanifold, $\mu$ is its Haar probability measure, and $T \colon X \to X$ is the translation $T(x) = \tau \cdot x$ for some fixed $\tau \in G$. 
An $s$-step pro-nilsystem is an inverse limit, in the measure-theoretic sense, of $s$-step nilsystems.
This is all explained in detail in \cref{chap:niltheory}.  
The Host--Kra structure theorem characterizes systems of order $s$ precisely. 

\begin{theorem*}[Host--Kra structure theorem]
    Let $s \geq 1$. An ergodic measure-preserving system is of order $s$ if and only if it is isomorphic to an $s$-step pro-nilsystem.  
\end{theorem*}

The two implications in this theorem are of a rather different character.  
The implication that an ergodic $s$-step pro-nilsystem is of order $s$ is relatively straightforward to prove: one verifies that ergodic $s$-step nilsystems are of order $s$, and then passes to the inverse limit; see \citep[Chapter 12, Corollary 19]{host2018nilpotent}. 
The converse implication is the structural part of the theorem: one has to recover the concrete form of a pro-nilsystem from the abstract order $s$ condition. 
This is the direction that requires the main machinery of the Host--Kra theory. 
As a consequence of the theorem, the structure factor $Z_s$ of an ergodic system can equivalently be described as its maximal factor which is an $s$-step pro-nilsystem; see \citep[Chapter 16]{host2018nilpotent}. 

In this way, pro-nilsystems arise as the algebraic models for the characteristic factors controlling multiple ergodic averages. 
This motivates the study of pro-nilsystems in their own right. 
It is therefore a natural question whether the class of $s$-step pro-nilsystems is closed under taking factors. 

For nilsystems, the corresponding question was resolved by \citet{parry1973dynamical}, who proved that factors of ergodic nilsystems are again nilsystems.
We record this result below in \cref{thrm:factor_ergodic_nilsystem_is_nilsystem}.
For pro-nilsystems, the Host--Kra structure theorem also gives an affirmative answer. 
Indeed, it is known that any factor of a system of order $s$ is again of order $s$; see \citep[Chapter 9, Proposition 17]{host2018nilpotent}. 
Combining this with the structure theorem, we obtain the following result.  

\begin{theorem*}[Factor-closure of ergodic pro-nilsystems]
    Every factor of an ergodic $s$-step pro-nilsystem is itself an $s$-step pro-nilsystem. 
\end{theorem*}

This argument, however, is somewhat unsatisfactory: although the statement only concerns the very concrete class of pro-nilsystems, the proof passes through the full Host--Kra structure theorem and thus relies on all the complicated machinery used to prove it. 
It is reasonable to expect that a direct proof should exist. 
This has been an open problem in the subject for some time. 
As \citeauthor{tao2015weak} writes in \citep{tao2015weak}:
\begin{quote}
    Theorem 4 is, in principle, purely a fact about nilsystems, and should have an independent proof, but this is not known; the only known proofs go through the full machinery needed to prove Theorem 2.
\end{quote}
Here Tao's Theorem 4 corresponds to the factor-closure theorem, while his Theorem 2 corresponds to the Host--Kra structure theorem. 
Similarly, \citeauthor{host2018nilpotent} write in \citep[p. 221]{host2018nilpotent}:
\begin{quote}
    More generally, one can show that any factor of an inverse limit of nilsystems is also an inverse limit of nilsystems. However, unfortunately the only proof that we know for this result relies on the ergodic Structure Theorem stated in Chapter 16, and at this point we are not able to prove this in a more elementary manner.
\end{quote}

The main result of this thesis is an independent proof of the factor-closure theorem; see \cref{thrm:main}. 
Our proof avoids the machinery behind the Host--Kra structure theorem entirely and instead relies only on intrinsic properties of nilsystems. 

The key new ingredient is a local rigidity result for self-joinings of nilsystems. 
Recall that a self-joining of $(X, \mu, T)$ is a $T \times T$-invariant probability measure on $X \times X$ whose coordinate projections are both equal to $\mu$. 
We denote by $J_e(X, X)$ the space of ergodic self-joinings of a nilsystem $X$, equipped with the weak-* topology; see \cref{chap:rigidity_result} for the precise definition. 

\begin{proposition*}[Local rigidity of self-joinings of nilsystems]
    Let $(X, \mu, T)$ be an ergodic nilsystem, and let $\mu_\Delta$ be the diagonal self-joining.  
    Then there exists a neighborhood $U$ of $\mu_\Delta$ in $J_e(X, X)$ such that every $\lambda \in U$ is the graph joining of an automorphism of $(X, \mu, T)$. 
\end{proposition*}

Thus, although ergodic nilsystems may admit non-graph self-joinings, these cannot occur arbitrarily close to the diagonal joining. 

The starting point in the proof of the local rigidity result is the fact that an ergodic self-joining of a nilsystem is the Haar measure of a subnilmanifold of the product nilmanifold; see \cref{prop:ergodic_joinings_nilsystems_are_nilsystems}. 
This allows us to solve the problem by studying the geometry of subnilmanifolds. 
We show that a nilmanifold cannot contain arbitrarily small non-trivial subnilmanifolds, together with an averaged form of this statement; see \cref{lem:no_small_subnilmanifolds} and \cref{lem:no_small_subnilmanifolds_averaged}. 
Now, if a self-joining is sufficiently close to the diagonal joining, then the two coordinates of a point in its support are on average close to each other, so that the averaged no small subnilmanifolds lemma forces the coordinate fibers of the support to be trivial.  
This demonstrates that the joining must be a graph joining. 

Let us briefly explain how the rigidity result is used to prove the factor-closure theorem; see \cref{sec:proof_strategy} for a detailed outline. 
Let $X$ be a pro-nilsystem, and suppose that $\pi \colon X \to Y$ is a factor map. 
The proof that $Y$ is a pro-nilsystem proceeds in two steps. 

In the first step, we reduce to the special case where the factor map $\pi$ is an extension by a compact abelian group. 
To this end, in \cref{sec:pronilsystem_tower} we construct a tower of pro-nilsystem factors 
\[
X = Z_s \longrightarrow Z_{s-1} \longrightarrow \cdots \longrightarrow Z_1 \longrightarrow Z_0 = \{*\}, 
\] 
where each map is an extension by a compact abelian group. 
This tower is obtained by passing the standard tower of factors of each finite-stage nilsystem to the inverse limit. 
In particular, this construction does not make use of Host--Kra theory. 
After defining the intermediate factors $W_i$ generated by $Y$ and $Z_i$, we argue by downward induction on $i$. 
The induction step reduces precisely to the compact abelian group extension case. 

In the second step, we prove this special case.
Suppose that $X = \varprojlim X_n$ is an ergodic pro-nilsystem which is an extension of $Y$ by a compact abelian group $K$.
Then $K$ acts on $X$ by automorphisms, called the vertical rotations $V_u$ for $u \in K$.  
The main obstruction is that the given inverse limit for $X$ need not be compatible with the group extension structure.
More precisely, the vertical rotations on $X$ need not descend to the finite stages $X_n$. 
Applying the rigidity proposition to the self-joinings 
\[
(p_n, p_n \circ V_u)_* \mu
\]
of $X_n$, where $p_n \colon X \to X_n$ denotes the factor map, we show that $V_u$ does descend to an automorphism of $X_n$ for all $u$ in a finite index subgroup of $K$. 
Enlarging each $X_n$ by a suitable self-joining then gives a new inverse limit for $X$ in which all vertical rotations do descend to every finite stage. 
An averaging argument over $K$ then allows us to realize $Y$ as an inverse limit of nilsystem factors. 

The factor-closure theorem also admits a topological dynamical analogue. 
A topological $s$-step pro-nilsystem is an inverse limit, in the topological dynamical sense, of $s$-step nilsystems. 

\begin{theorem*}[Topological factor-closure]
    Every factor of a transitive topological $s$-step pro-nilsystem is itself a topological $s$-step pro-nilsystem.  
\end{theorem*}

This theorem is a consequence of the topological structure theorem of \citet{host2010nilsequences}. 
We instead deduce it directly from the measure-theoretic factor-closure theorem discussed above, by combining it with the correspondence between ergodic pro-nilsystems and their canonical topological models, and an elementary criterion for upgrading measure-theoretic isomorphisms of distal uniquely ergodic systems to topological ones; see \cref{sec:topological_factor_closure}. 
Thus we obtain a proof which is intrinsic to nilsystems and does not invoke any structure theorem.   

The final part of this thesis concerns a further consequence of the factor-closure theorem, regarding the relation between the Host--Kra structure theorem and the combinatorial inverse theorem for the Gowers norms. 
To explain this connection, we isolate the following weaker form of the Host--Kra structure theorem. 

\begin{theorem*}[Weak structure theorem]
    Every ergodic measure-preserving system of order $s$ is a factor of an ergodic $s$-step pro-nilsystem. 
\end{theorem*}

The weak structure theorem, when combined with the factor-closure theorem, implies the structural direction of the Host--Kra structure theorem. 
Indeed, if an ergodic system is of order $s$, then the weak structure theorem shows that it is a factor of an ergodic $s$-step pro-nilsystem, and by factor-closure, it is therefore an $s$-step pro-nilsystem.  

As observed by \citet{tao2015weak}, the weak structure theorem can be deduced directly from the inverse theorem for the Gowers norms, in the form proved by \citet{green2012inverse}. 
This inverse theorem is purely combinatorial. 
Roughly speaking, it says that a $1$-bounded function on the set $\{1, \dots, N\}$, whose Gowers $U^{s+1}$-norm is large, must correlate with an $s$-step nilsequence of bounded complexity. 
The precise statement, together with the required background on filtered nilmanifolds, polynomial sequences, and nilsequences, is collected in \cref{chap:appendix}. 
Tao's argument proceeds by approximating the dynamical dual functions in a system of order $s$ by nilsequences using the inverse theorem, and by assembling the resulting nilsequences into a single pro-nilsystem admitting the original system as a factor. 
Since some details in the argument in \citep{tao2015weak} are omitted, we give the complete proof in \cref{chap:weak_structure_thrm}. 

What remained missing in Tao's plan to prove the full Host--Kra structure theorem was precisely an independent proof of the factor-closure theorem. 
The results of this thesis provide this missing part. 
Combining Chapters \ref{chap:factor_closure} and \ref{chap:weak_structure_thrm}, we therefore obtain a proof of the structural direction of the Host--Kra structure theorem from the inverse theorem for the Gowers norms. 
This completes Tao's proposed route and clarifies the logical relationship between the combinatorial and ergodic inverse theorems. 

This thesis is based on the joint paper \citep{eeckhaut2026localrigidityselfjoiningsfactors} with my advisor Asgar Jamneshan, which already included the local rigidity proposition, the factor-closure theorem, and its topological analogue. 
Compared with \citep{eeckhaut2026localrigidityselfjoiningsfactors}, the present thesis provides a self-contained discussion of the necessary background on nilmanifolds, nilsystems, and pro-nilsystems, and it includes the full details of the deduction of the weak structure theorem from the inverse theorem for the Gowers norms as proposed by Tao. 

This thesis is organized as follows.
In \cref{chap:background} we fix our conventions and recall the necessary background on measure-preserving and topological dynamical systems, inverse limits, and group extensions.  
In \cref{chap:niltheory}, we introduce nilmanifolds, nilsystems, and pro-nilsystems, and record their properties used later.  
\cref{chap:rigidity_result} contains the proof of the local rigidity result for self-joinings of nilsystems. 
In \cref{chap:factor_closure}, we use the rigidity result to prove the factor-closure theorem for ergodic pro-nilsystems, and then deduce its topological analogue. 
In \cref{chap:weak_structure_thrm}, using Tao's argument, we give a detailed proof of the weak structure theorem from the inverse theorem for the Gowers norms, and we explain how this combines with the factor-closure theorem to recover the Host--Kra structure theorem. 
\cref{chap:appendix} collects the notions and results from the inverse theory for the Gowers norms used in \cref{chap:weak_structure_thrm}, and \cref{chap:factor_criterion_appendix} contains a criterion for proving that a given system is a factor of another, which is needed in the proof of the weak structure theorem. 

\chapter{Preliminaries on dynamical systems}
\label{chap:background}

The purpose of this chapter is to fix the conventions and notation for dynamical systems used throughout this thesis, and to record some facts that will be needed later. 
We recall only the material which is directly relevant to later arguments. 
For general background on ergodic theory and topological dynamics, we refer to \citep{auslander1988minimal, eisner2025journey,glasner2003ergodic,einsiedler2010ergodic}. 

\section{Measure-preserving systems}
\label{sec:measure_preserving_systems}

\subsection{Measure-theoretic conventions}

If $(X, \mcx, \mu)$ is a probability space, we denote by $\overline{\mcx}^\mu$ the completion of the $\sigma$-algebra $\mcx$ with respect to $\mu$.
The measure $\mu$ then extends uniquely to $\overline{\mcx}^\mu$, and we denote this extension also by $\mu$. 
The elements of $\overline{\mcx}^\mu$ are called the measurable subsets of $X$. 

If $(X, \mcx, \mu)$ and $(Y, \mcy, \nu)$ are probability spaces, a map $\phi \colon X \to Y$ is measurable if $\phi^{-1}(A) \in \overline{\mcx}^\mu$ for every $A \in \mcy$. 
In this case, the pushforward or image of $\mu$ under $\phi$ is 
\[
\phi_* \mu (A) := \mu(\phi^{-1}(A)) \quad \text{for } A \in \mcy. 
\]
The map $\phi$ is called measure-preserving if $\phi_*\mu = \nu$.
If $\phi \colon X \to X$ is measure-preserving, we also say that $\mu$ is invariant under $\phi$. 
An isomorphism of probability spaces is a measure-preserving measurable map 
\[
\phi \colon (X, \mcx, \mu) \to (Y, \mcy, \nu), 
\]
which admits a measurable inverse modulo null sets.
More precisely, this means that there is a measurable map $\psi \colon Y \to X$ such that $\psi \circ \phi = \Id_X$ $\mu$-a.e. and $\phi \circ \psi = \Id_Y$ $\nu$-a.e.
In this case, the measurable inverse is uniquely determined up to $\nu$-null sets and we will usually denote it by $\phi^{-1}$. 

A probability space is called standard if it is isomorphic to a Polish space, equipped with its Borel $\sigma$-algebra and a Borel probability measure. 

\begin{assumption}
    All probability spaces in this thesis are assumed to be standard.  
\end{assumption}

\subsection{Definition}

\begin{definition}
    A measure-preserving system is a quadruple $(X, \mcx, \mu, T)$ consisting of a probability space $(X, \mcx, \mu)$ and an isomorphism $T \colon X \to X$ from this probability space to itself.  
\end{definition}

We often abbreviate the notation and write $(X, \mu, T)$ or simply $X$ for a measure-preserving system. 
The corresponding $\sigma$-algebra will in this case be denoted by the corresponding script letter $\mcx$.  
As $T$ is an isomorphism of probability spaces, we can fix a measurable inverse once and for all, and denote it by $T^{-1}$. 

A measurable subset $A \subseteq X$ is called $T$-invariant if 
\[
\mu(T^{-1} A \Delta A) = 0, 
\]
where $\Delta$ denotes the symmetric difference of sets. 
The system $(X, \mu, T)$ is ergodic if every invariant measurable subset either has measure $0$ or $1$. 

For $1 \leq p \leq \infty$, we write $L^p(X, \mu)$ for the usual $L^p$-space. 
If $\mca \subseteq \overline{\mcx}^\mu$ is a sub-$\sigma$-algebra, we write $L^p(X, \mca, \mu)$ for the closed subspace of $L^p(X, \mu)$ consisting of functions that admit an $\mca$-measurable representative. 
We denote the Koopman operator associated to $T$ also by $T$, so $Tf := f \circ T$ for every function $f$ on $X$.  
A function $f \in L^p(X, \mu)$ is called $T$-invariant if $Tf = f$ holds $\mu$-almost everywhere. 
Ergodicity is then equivalent to the assertion that the $T$-invariant functions in $L^1(X,\mu)$ are exactly the almost everywhere constants. 

For $f \in L^1(X,\mu)$ and a sub-$\sigma$-algebra $\mca \subseteq \overline{\mcx}^\mu$, we denote the conditional expectation by $\E_\mu (f \mid \mca)$. 
For functions in $L^2(X, \mu)$, this agrees with the orthogonal projection onto the closed subspace $L^2(X, \mca, \mu)$. 

\subsection{Factor maps}

\begin{definition}
    Let $(X, \mu, T)$ and $(Y, \nu, S)$ be measure-preserving systems. 
    A measurable measure-preserving map $\pi \colon X \to Y$  is a factor map if 
    \[
    \pi \circ T = S \circ \pi \quad \mu\aee 
    \] 
    In this case $(Y, \nu, S)$ is called a factor of $(X, \mu, T)$.  
\end{definition}

A factor map which is also an isomorphism of probability spaces is called an isomorphism of measure-preserving systems. 
In this case, the measurable inverse is also a factor map. 
An isomorphism from a measure-preserving system to itself is called an automorphism. 

Factors admit a useful equivalent description in terms of invariant sub-$\sigma$-algebras. 
Let $(X, \mu, T)$ be a measure-preserving system. 
If $\mca$ and $\mcb$ are sub-$\sigma$-algebras of $\overline{\mcx}^\mu$, we write $\mca \subseteq_\mu \mcb$ if for every $A \in \mca$, there exists $B \in \mcb$ such that $\mu(A \Delta B) = 0$. 
Moreover, we write $\mca =_\mu \mcb$ if both $\mca \subseteq_\mu \mcb$ and $\mcb \subseteq_\mu \mca$ hold. 

A sub-$\sigma$-algebra $\mca \subseteq \overline{\mcx}^\mu$ is called $T$-invariant if 
\[
T^{-1}(\mca) =_\mu \mca.  
\]
Every factor map $\pi \colon X \to Y$ gives rise to the invariant sub-$\sigma$-algebra $\pi^{-1}(\mcy)$.
Conversely, since all probability spaces are assumed to be standard, every invariant sub-$\sigma$-algebra of $\overline{\mcx}^\mu$ arises in this way, up to $=_\mu$, from a factor map, and this factor is unique up to isomorphism. 
More precisely, if $\pi \colon X \to Y$ and $\pi' \colon X \to Y'$ are factor maps with $\pi^{-1}(\mcy) =_\mu (\pi')^{-1}(\mcy')$, then there exists an isomorphism $\phi \colon Y \to Y'$ such that $\phi \circ \pi = \pi'$ $\mu$-almost everywhere. 

This correspondence also naturally orders factors by inclusion of the corresponding sub-$\sigma$-algebras. 
If $\pi \colon X \to Y$ and $\pi' \colon X \to Y'$ are factor maps, we say that $Y$ lies below $Y'$ or that $Y'$ lies above $Y$ if 
\[
\pi^{-1}(\mcy) \subseteq_\mu (\pi')^{-1}(\mcy').
\]
Equivalently, $Y$ lies below $Y'$ if there exists a factor map $\phi \colon Y' \to Y$ such that $\phi \circ \pi' = \pi$ $\mu$-almost everywhere. 

For two sub-$\sigma$-algebras $\mca$ and $\mcb$, we write $\mca \vee \mcb$ for the $\sigma$-algebra generated by $\mca \cup \mcb$.
More generally, if $(\mca_i)_{i \in I}$ is a family of sub-$\sigma$-algebras, we write $\bigvee_{i \in I} \mca_i$ for the sub-$\sigma$-algebra generated by $\bigcup_{i \in I} \mca_i$. 

\subsection{Joinings}
\label{sec:joinings}

\begin{definition}
    Let $(X, \mu, T)$ and $(X', \mu', T')$ be measure-preserving systems.  
    A joining of these two systems is a probability measure $\lambda$ on $(X \times X', \mcx \otimes \mcx')$ which is invariant under $T \times T'$ and such that the coordinate projections map $\lambda$ to $\mu$ and $\mu'$ respectively. 
\end{definition}

If $\lambda$ is a joining of $(X, \mu, T)$ and $(X', \mu', T')$, then  
\[
(X \times X', \lambda, T \times T')
\]
is a measure-preserving system, and the coordinate projections are factor maps. 
The joining $\lambda$ is called ergodic if this system is ergodic.  
We write $J(X, X')$ for the set of joinings, and $J_e(X, X')$ for the set of ergodic joinings. 
The set $J(X,X')$ is never empty since the product measure $\mu \times \mu'$ is always a joining. 

If $\phi \colon (X, \mu, T) \to (X', \mu', T')$ is a factor map, then the probability measure 
\[
(\Id_X, \phi)_* \mu
\]
on $X \times X'$ is a joining, called the graph joining of $\phi$. 

A joining of a system with itself is called a self-joining. 
In particular, the diagonal joining
\[
\mu_\Delta := (\Id_X, \Id_X)_* \mu
\]
is always a self-joining of $(X, \mu, T)$. 
Later, graph self-joinings associated to automorphisms will play an important role. 

\section{Topological dynamical systems}

\begin{definition}
    A topological dynamical system is a pair $(X, T)$, where $X$ is a compact metrizable space and $T \colon X \to X$ is a homeomorphism. 
\end{definition}

As before, we often denote the system simply by $X$ when the transformation is understood. 

\begin{definition}
    Let $(X, T)$ and $(Y, S)$ be topological dynamical systems. 
    A continuous surjection $\pi \colon X \to Y$ is a factor map if 
    \[
    \pi \circ T = S \circ \pi. 
    \]
    In this case $(Y, S)$ is called a factor of $(X, T)$. 
\end{definition}

An isomorphism of topological dynamical systems is a bijective factor map. 
Thus, an isomorphism is in particular a homeomorphism, and its inverse is again a factor map.  
An isomorphism from a topological dynamical system to itself is called an automorphism. 

We use the term factor map both in the setting of measure-preserving and topological dynamical systems.   
Whenever this may be ambiguous, we specify the intended meaning by saying measure-theoretic or topological factor map.  

Let $(X, T)$ be a topological dynamical system. 
The orbit of $x \in X$ is the set $\{T^n x : n \in \Z\} \subseteq X$. 
The system is called transitive if there exists a point whose orbit is dense, and minimal if the orbit of every point is dense.  
A subset $A \subseteq X$ is called $T$-invariant if $T^{-1}A = A$. 
The system $(X, T)$ is minimal if and only if the only closed $T$-invariant subsets are $\emptyset$ and $X$. 

Every topological dynamical system $(X, T)$ admits at least one $T$-invariant Borel probability measure by the Krylov--Bogolyubov theorem. 
If $\mu$ is an invariant measure, then $(X, \mu, T)$ is a measure-preserving system. 
The system $(X, T)$ is called uniquely ergodic if there exists precisely one invariant probability measure. 
The support of an invariant probability measure $\mu$ on $(X, T)$, denoted by $\supp(\mu)$, is the smallest closed subset $A$ of $X$ satisfying $\mu(A) = 1$.  
Equivalently, $\supp(\mu)$ is the set of points for which every open neighborhood has strictly positive measure. 

Two points $x, y \in X$ are called distal if 
\[
\inf_{n \in \Z} d(T^n x, T^n y) > 0,
\]
where $d$ is a compatible metric on $X$. 
Otherwise, $x$ and $y$ are said to be proximal. 
The system $(X, T)$ is distal if any pair of distinct points is distal.  
This condition does not depend on the specific choice of metric $d$. 

We will need the following lemma, which relates measure-theoretic and topological isomorphisms for distal uniquely ergodic systems. 
Topological factor maps between uniquely ergodic systems that are also measure-theoretic isomorphisms were studied by \citet{downarowicz2016isomorphic}, and the lemma can in particular be derived from \citep[Proposition 2.5]{downarowicz2016isomorphic}. 
We instead include a short self-contained proof.

\begin{lemma} 
    \label[lemma]{lem:isomorphism_upgrade_crit} 
    Let $(X, T)$ be a distal uniquely ergodic topological dynamical system with unique invariant measure $\mu$. 
    Let \[ \pi \colon (X, T) \to (Y, S) \] be a topological factor map, and set $\nu = \pi_* \mu$. If $\pi$ is a measure-theoretic isomorphism from $(X, \mu, T)$ to $(Y, \nu, S)$, then $\pi$ is a topological isomorphism.
\end{lemma}

\begin{proof} 
    It suffices to show that $\pi$ is injective. 
    Suppose it is not.
    Then there exist distinct points $x, x' \in X$ such that $\pi(x) = \pi(x')$. 
    Let
    \[ 
    Z := \overline{\{(T^n x, T^n x') : n \in \Z\}}. 
    \] 
    Then $Z$ is closed in $X \times X$, $T \times T$-invariant, and $\pi(z_1) = \pi(z_2)$ for all $(z_1, z_2) \in Z$. 
    Also $Z \cap \Delta_X = \emptyset$ where $\Delta_X \subseteq X \times X$ denotes the diagonal. 
    Indeed, if $Z$ intersected $\Delta_X$ non-trivially, then $x$ and $x'$ would be proximal, contradicting distality as $x$ and $x'$ are distinct.

    Let $\lambda$ be a $T \times T$-invariant probability measure supported on $Z$.
    By unique ergodicity, the two coordinate projections both map $\lambda$ to $\mu$, so $\lambda$ is a self-joining of $(X, \mu, T)$. 
    Since $\pi$ is a measure-theoretic isomorphism, we can take its measurable inverse $\pi^{-1}$.
    Then for $\lambda$-almost every $(z_1, z_2) \in Z$, we have that 
    \[
    z_1 = \pi^{-1}(\pi(z_1)) = \pi^{-1}(\pi(z_2)) = z_2. 
    \] 
    Therefore, $\lambda$ is supported on $\Delta_X$. This contradicts the fact that $Z \cap \Delta_X = \emptyset$. 
\end{proof}

\section{Inverse limits of dynamical systems} 
\label{sec:inverse_limits}

We consider inverse limits of dynamical systems indexed by $\N$. 
This causes no loss of generality compared to the countable inverse limits appearing in the works of Host and Kra \citep{host2005nonconventional, host2018nilpotent}. 
Indeed, every countable directed set admits a cofinal sequence, and restricting to such a sequence does not change the inverse limit; see \citep[Chapter 3, Section 6]{host2018nilpotent}.

\subsection{Inverse limits of measure-preserving systems}

An inverse system of measure-preserving systems is a sequence $(X_n, \mu_n, T_n)_{n \in \N}$ of measure-preserving systems together with factor maps 
\[
\alpha_n \colon (X_{n+1}, \mu_{n+1}, T_{n+1}) \to (X_n, \mu_n, T_n) \quad \text{for } n \in \N.
\] 
Such an inverse system will be denoted by $((X_n, \mu_n, T_n), \alpha_n)_{n \in \N}$. 
The maps $\alpha_n$ are called the connecting factor maps. 

\begin{definition}
    Let $((X_n, \mu_n, T_n), \alpha_n)_{n \in \N}$ be an inverse system.
    A measure-preserving system $(X, \mu, T)$, equipped with factor maps 
    \[
    p_n \colon (X, \mu, T) \to (X_n, \mu_n, T_n) \quad \text{for } n \in \N,
    \]
    is the inverse limit of the inverse system if the following two conditions hold.  
    \begin{enumerate}[(i)]
        \item $\alpha_n \circ p_{n+1} = p_n$ $\mu$-a.e. for every $n \in \N$,  
        \item $\bigvee_{n\in\N} p_n^{-1}(\mathcal{X}_n) =_\mu \mathcal{X}$.  
    \end{enumerate}
    In this case, we write 
    \[
    (X, \mu, T) = \varprojlim (X_n, \mu_n, T_n). 
    \]
\end{definition}

Even when not explicitly specified, the connecting factor maps remain part of the definition of the inverse limit.  
The inverse limit always exists and is unique up to isomorphism. 
Existence can be shown by explicitly constructing the inverse limit on the product $\prod_n X_n$ and defining the measure using the Kolmogorov consistency theorem; see \citep{parthasarathy2005probability}.
To see uniqueness one can then prove that any inverse limit is isomorphic to this concrete construction. 

Let $(X, \mu, T)$ be a measure-preserving system equipped with factor maps $p_n \colon X \to X_n$ for $n \in \N$ satisfying (i). 
Then property (ii) is equivalent to either of the following two properties. 
\begin{enumerate}
    \item[(ii')] For every $1 \leq p < \infty$, the subspace
    \[
    \bigcup_{n\in\N}\{f \circ p_n : f \in L^p(X_n, \mu_n)\}
    \]
    is dense in $L^p(X, \mu)$. 
    \item[(ii'')] For every $1 \leq p < \infty$, and every $f \in L^p(X, \mu)$, one has 
    \[
    \E_\mu(f \mid p_n^{-1}(\mathcal{X}_n)) \to f \quad \text{in } L^p(X, \mu).
    \]
\end{enumerate}

Indeed, (ii'') follows from (ii) by the increasing martingale theorem; see \citep[Theorem 5.5]{einsiedler2010ergodic}. 
It is clear that (ii'') implies (ii'), and an approximation argument using indicator functions shows that (ii') implies (ii). 

We will need the following lemma, which is a consequence of the definition of the inverse limit, combined with the characterization of factors in terms of invariant sub-$\sigma$-algebras discussed above. 

\begin{lemma}
    \label[lemma]{lem:inverse_limit_criterion}
    Let $(X, \mu, T)$ be a measure-preserving system, and let $p_n \colon (X, \mu, T) \to (Y_n, \nu_n, S_n)$ for $n \in \N$ be factor maps such that  
    \[
    p_1^{-1}(\mathcal{Y}_1) \subseteq_\mu p_2^{-1}(\mathcal{Y}_2) \subseteq_\mu \cdots.  
    \]
    Let $\pi \colon (X, \mu, T) \to (Y, \nu, S)$ be the factor corresponding to the invariant sub-$\sigma$-algebra $\bigvee_{n\in\N} p_n^{-1}(\mathcal{Y}_n)$. 
    Then
    \[
    (Y, \nu, S) = \varprojlim (Y_n, \nu_n, S_n)
    \] 
    with respect to the connecting factor maps induced by the above inclusions. 
\end{lemma}

\begin{proof}
    For $n \in \N$, let $\alpha_n \colon Y_{n + 1} \to Y_n$ be the connecting factor map induced by the inclusion $p_n^{-1}(\mcy_n) \subseteq_\mu p_{n+1}^{-1}(\mcy_{n+1})$.
    Similarly, let $q_n \colon Y \to Y_n$ be the factor induced by $p_n^{-1}(\mcy_n) \subseteq_\mu \pi^{-1}(\mcy)$.  
    By uniqueness of the induced factor maps, $\alpha_n \circ q_{n+1} = q_n$ $\nu$-almost everywhere. 
    Moreover, $\mcy =_\nu \bigvee_{n \in \N} q_n^{-1}(\mcy_n)$ by construction. 
\end{proof}

\begin{lemma} 
    \label[lemma]{lem:ergodicity_of_inverse_limit}
    Let $(X, \mu, T) = \varprojlim (X_n, \mu_n, T_n)$ be an inverse limit of measure-preserving systems. 
    If each system $(X_n, \mu_n, T_n)$ is ergodic, then $(X, \mu, T)$ is ergodic. 
\end{lemma}

\begin{proof}
    Suppose that $f \in L^1(X, \mu)$ satisfies $T f = f$. 
    For $n \in \N$, let  
    \[
    f_n := \E_\mu(f \mid p_n^{-1}(\mcx_n)).
    \]
    Then each $f_n$ is $T$-invariant. 
    Since $f_n$ is $p_n^{-1}(\mcx_n)$-measurable, there exists $g_n \in L^1(X_n, \mu_n)$ such that $f_n = g_n \circ p_n$. 
    Then $T_n g_n = g_n$, so that $g_n$ is constant by ergodicity of $(X_n, \mu_n, T_n)$.
    Hence $f_n$ is constant. 
    By condition (ii'') above, $f_n \to f$ in $L^1(X, \mu)$. 
    It follows that $f$ is a constant, so $(X, \mu, T)$ is ergodic. 
\end{proof}

\subsection{Inverse limits of topological dynamical systems}

An inverse system of topological dynamical systems is a sequence $(X_n, T_n)_{n \in \N}$ of topological dynamical systems, together with factor maps
\[
\alpha_n \colon (X_{n+1}, T_{n+1}) \to (X_n, T_n) \quad \text{for } n \in \N.
\]
Such an inverse system will be denoted by $((X_n, T_n), \alpha_n)_{n \in \N}$. 
Again, the maps $\alpha_n$ are called connecting factor maps. 

\begin{definition}
    Let $((X_n, T_n), \alpha_n)_{n \in \N}$ be an inverse system of topological dynamical systems. 
    Let
    \[
    X := \left\{
        (x_n)_{n \in \N} \in \prod_{n\in\N} X_n : \alpha_n(x_{n+1}) = x_n \text{ for all } n \in \N
        \right\},
    \] 
    equipped with the subspace topology from $\prod_{n \in \N} X_n$, and define
    \[
    T \colon X \to X, \quad (x_n)_{n \in \N} \mapsto (T_n x_n)_{n \in \N}.
    \]
    The system $(X, T)$ is called the inverse limit of the inverse system, and we write
    \[
    (X, T) = \varprojlim(X_n, T_n).
    \]
\end{definition}

The subset $X$ is closed in $\prod_{n \in \N} X_n$, hence compact and metrizable. 
The map $T$ is a well-defined homeomorphism since each $T_n$ is a homeomorphism and the maps $\alpha_n$ are factor maps. 
Thus, $(X, T)$ is indeed a topological dynamical system. 

The coordinate projections $p_n \colon X \to X_n, (x_m)_{m \in \N} \mapsto x_n$ are factor maps, and they satisfy $\alpha_{n} \circ p_{n+1} = p_n$. 
Moreover, they separate points of $X$.
Thus, by Stone--Weierstrass, it follows that  
\[
\bigcup_{n \in \N} \{f \circ p_n : f \in C(X_n)\}
\]
is uniformly dense in $C(X)$. 

Conversely, suppose that $(Y, S)$ is a topological dynamical system that is equipped with factor maps $q_n \colon (Y, S) \to (X_n, T_n)$ such that $\alpha_n \circ q_{n+1} = q_n$ for all $n \in \N$. 
Then  
\[
q \colon (Y, S) \to (X, T), \quad y \mapsto (q_n(y))_{n \in \N}
\]
is a factor map. 
The factor map $q$ is an isomorphism if and only if the maps $q_n$ separate points. 

\begin{lemma}
    \label[lemma]{lem:dynamical_prop_inverse_limit}
    Let $(X, T) = \varprojlim(X_n, T_n)$ be an inverse limit of topological dynamical systems. 
    \begin{enumerate}[(i)]
        \item
        If each $(X_n, T_n)$ is minimal, then $(X, T)$ is minimal.  
        \item
        If each $(X_n, T_n)$ is distal, then $(X, T)$ is distal. 
        \item 
        If each $(X_n, T_n)$ is uniquely ergodic, then $(X, T)$ is uniquely ergodic. 
        Moreover, if $\mu_n$ is the unique invariant probability measure on $X_n$, and $\mu$ is the unique invariant probability measure on $X$, then 
        \[
        (X, \mu, T) = \varprojlim (X_n, \mu_n, T_n) 
        \]
        in the measure-theoretic sense.
    \end{enumerate}
\end{lemma}

\begin{proof}
    Let $p_n \colon X \to X_n$ denote the coordinate factor maps. 

    For (i), let $F \subseteq X$ be non-empty, closed, and $T$-invariant. 
    For each $n$, $p_n(F)$ is non-empty, closed and $T_n$-invariant.  
    Hence by minimality, $p_n(F) = X_n$. 
    Pick any $x \in X$. 
    Then for each $n$, choose $x_n \in F$ such that $p_n(x_n) = p_n(x)$.  
    After passing to a convergent subsequence, $x_n$ converges to a point $y \in F$.
    For each fixed $m$, the compatibility of the coordinate factor maps gives $p_m(x_n) = p_m(x)$ for all $n \geq m$. 
    Therefore $p_n(y) = p_n(x)$ for each $n$.
    Thus $y = x$, so $F = X$. 

    For (ii), let $x \neq y$ in $X$.
    Then $p_n(x) \neq p_n(y)$ for some $n \in \N$. 
    If $x$ and $y$ were proximal, then $p_n(x)$ and $p_n(y)$ would be proximal in $X_n$, contradicting distality of $(X_n, T_n)$. 

    For (iii), let $\nu$ be a $T$-invariant probability measure on $X$. 
    Then $(p_n)_* \nu$ is a $T_n$-invariant probability measure on $X_n$, so $(p_n)_* \nu = \mu_n$ for each $n$. 
    Hence the integral of $f \circ p_n$ is independent of $\nu$ for every $f \in C(X_n)$. 
    Since such functions are uniformly dense in $C(X)$, the measure $\nu$ is uniquely determined.
    The final claim follows from the same density result combined with characterization (ii') of the measure-theoretic inverse limit above. 
    Indeed, uniform density in $C(X)$ implies density in $L^p(X, \mu)$ for every $1 \leq p < \infty$. 
\end{proof}

\section{Compact abelian group extensions}

\begin{assumption}
    Throughout this thesis, all topological groups are assumed to be metrizable. 
\end{assumption}

For the rest of this section, any compact abelian group $K$ is written additively, with identity element $0$, and the Haar measure denoted by $m_K$. 

\subsection{Group extensions of measure-preserving systems}
\label{sec:group_extensions_measurable}

Let $(Y, \nu, S)$ be a measure-preserving system, let $K$ be a compact abelian group, and let $\rho \colon Y \to K$ be a measurable map. 
Define the map
\[
T_\rho \colon Y \times K \to Y \times K, \quad 
(y, g) \mapsto (Sy, \rho(y) + g). 
\]
Then $T_\rho$ is measurable, and it preserves the product measure $\nu \times m_K$. 
Moreover, $T_\rho$ admits a measurable inverse, namely $T_\rho^{-1} (y, g) = (S^{-1} y, g - \rho(S^{-1} y))$. 
Thus, 
\[
Y \rtimes_\rho K := (Y \times K, \nu \times m_K, T_\rho)
\]
is a measure-preserving system. 
The projection $p \colon Y \rtimes_\rho K \to Y, (y, g) \mapsto y$ is a factor map, called the skew-product extension of $Y$ by $K$ associated to $\rho$.

For $u \in K$, define  
\[
V_u \colon Y \rtimes_\rho K \to Y \rtimes_\rho K, \quad 
(y, g) \mapsto (y, u + g),  
\]
which is called the vertical rotation by $u$. 
The map $V_u$ is an automorphism, which satisfies
\[
V_u \circ V_v = V_{u + v}, \quad V_u^{-1} = V_{-u}, \quad p \circ V_u = p 
\]
for $u, v \in K$. 

An application of Fubini's theorem shows that the conditional expectation with respect to the factor map $p$ is given by averaging over the vertical rotations. 
More precisely, for all $f \in L^1(Y \times K, \nu \times m_K)$, 
\[
\E_{\nu \times m_K}(f \mid p^{-1}(\mcy))(y, g) 
= \int_K f(y, u + g) \dd m_K(u)
\]
for $\nu \times m_K$-almost every $(y, g) \in Y \times K$. 

We will also need the strong continuity of the associated Koopman representation. 
For $u \in K$, let 
\[
U_u f := f \circ V_u, \quad \text{for } f \in L^2(Y \times K, \nu \times m_K). 
\]
Then $u \mapsto U_u$ is a strongly continuous unitary representation of $K$ on $L^2(Y \times K, \nu \times m_K)$. 
Indeed, if $u_n \to u$ in $K$, then it is clear that $U_{u_n} f \to U_u f$ for all functions $f$ of the form $f(y,g) = h(y) r(g)$, where $h \in L^2(Y, \nu)$ and $r \in C(K)$. 
Since such functions are dense in $L^2(Y \times K, \nu \times m_K)$, this proves strong continuity. 

More generally, a factor map is a compact abelian group extension if it agrees with a skew-product extension up to isomorphism.  

\begin{definition}
    A factor map $\pi \colon (X, \mu, T) \to (Y, \nu, S)$ is called an extension of $Y$ by $K$ if there exist a measurable map $\rho \colon Y \to K$ and an isomorphism 
    \[
    \Phi \colon (X, \mu, T) \to Y \rtimes_\rho K, 
    \] 
    such that $p \circ \Phi = \pi$ $\mu$-almost everywhere. 
\end{definition}

The vertical rotations can be transferred via $\Phi$ to automorphisms of $(X, \mu, T)$, which we again denote by $V_u$ and call vertical rotations. 
They satisfy $\pi \circ V_u = \pi$, $V_u \circ V_v = V_{u + v}$, and $V_u^{-1} = V_{-u}$ almost everywhere. 
Furthermore, for $f \in L^1(X, \mu)$, the conditional expectation now takes the form   
\[
\E_\mu (f \mid \pi^{-1}(\mcy))(x)
= \int_K f(V_u x) \dd m_K(u)
\]
for $\mu$-almost every $x \in X$. 
The Koopman representation of $K$ on $L^2(X, \mu)$, defined by $U_u(f) = f \circ V_u$ for $f \in L^2(X,\mu)$, is still strongly continuous. 

The following lemma is the main technical result about group extensions which we will use later. 
This is proved in \citep[Chapter 5, Lemma 18]{host2018nilpotent}. 
A related more classical version appears in \citep[Section 7]{furstenberg2011mean}.

\begin{lemma} 
    \label[lemma]{lem:induced_group_ext}
    Let $(X, \mu, T)$, $(X', \mu', T')$, $(Y, \nu, S)$, and $(Y', \nu', S')$ be ergodic measure-preserving systems, related by the following commutative diagram of factor maps. 
    \[\begin{tikzcd}
        X & {X'} \\
        Y & {Y'}
        \arrow["q", from=1-1, to=1-2]
        \arrow["\pi"', from=1-1, to=2-1]
        \arrow["{\pi'}", from=1-2, to=2-2]
        \arrow["r", from=2-1, to=2-2]
    \end{tikzcd}\]
    Suppose that $\pi' \colon X' \to Y'$ is an extension by a compact abelian group $K$, and that 
    \[
    \mcx =_\mu \pi^{-1}(\mcy) \vee q^{-1}(\mcx'). 
    \]
    Then there exists a closed subgroup $H \leq K$ such that $\pi \colon X \to Y$ is an extension by $H$. 
\end{lemma}

\subsection{Group extensions of topological dynamical systems}

We also need the analogous theory for topological dynamical systems.
Here, the convenient definition is intrinsic rather than via skew-products. 

Let $(X, T)$ be a topological dynamical system and let $K$ be a compact abelian group. 
A continuous action of $K$ on $X$ is a continuous map 
\[
K \times X \to X, \quad (u, x) \mapsto V_u x, 
\]
such that $V_0 = \Id_X$ and $V_{u + v} = V_u \circ V_v$ for all $u, v \in K$. 
Then each map $V_u \colon X \to X$ is a homeomorphism with inverse $V_{-u}$. 
The action is said to be by automorphisms if $V_u \circ T = T \circ V_u$ for each $u \in K$, and it is free if $V_u x = x$ for some $x \in X$ implies $u = 0$. 

\begin{definition}
    A factor map 
    \[
    \pi \colon (X, T) \to (Y, S) 
    \]
    is an extension of $(Y, S)$ by $K$ if there exists a free continuous action of $K$ on $X$ by automorphisms such that the fibers of $\pi$ are exactly the $K$-orbits, meaning that 
    \[
    \pi(x) = \pi(x') \quad \Longleftrightarrow \quad x' = V_u x \text{ for some } u \in K. 
    \]
\end{definition}

The automorphisms $V_u \colon X \to X$ are again called vertical rotations. 
Since the fibers are $K$-orbits, we have that $\pi \circ V_u = \pi$ for all $u \in K$. 
The factor $Y$ is homeomorphic to the orbit space $X/K$, and the transformation on $X/K$ induced by $T$ agrees with $S$ under this identification. 

The following lemma connects the notions of group extensions in the measure-theoretic and topological settings.  
This is \citep[Chapter 5, Proposition 1]{host2018nilpotent}. 

\begin{lemma} 
    \label[lemma]{lem:topological_group_extension_is_measurable}
    Let 
    \[
    \pi \colon (X, T) \to (Y, S)
    \] 
    be a topological extension by the compact abelian group $K$. 
    Suppose that $\nu$ is an $S$-invariant Borel probability measure on $Y$. 
    Then there exists a unique Borel probability measure $\mu$ on $X$ which is invariant under all vertical rotations $V_u$ for $u \in K$, and satisfies $\pi_* \mu = \nu$. 
    Moreover, $\mu$ is $T$-invariant and 
    \[
    \pi \colon (X, \mu, T) \to (Y, \nu, S)
    \] 
    is a measure-theoretic group extension by $K$. 
\end{lemma}

\chapter{Nilmanifolds, nilsystems and pro-nilsystems}
\label{chap:niltheory}

Nilsystems are the fundamental objects of interest in this thesis. 
They may be viewed as non-abelian generalizations of rotations on compact abelian Lie groups, where the group is replaced by a homogeneous space of a nilpotent Lie group. 
This chapter introduces the basic theory of nilmanifolds, nilsystems, and pro-nilsystems. 
We begin with the geometric setting of nilmanifolds, then define nilsystems as translations on these spaces, and conclude by passing to the inverse limit leading to the notion of a pro-nilsystem. 

\section{Nilmanifolds}

We first consider nilmanifolds. 
Our conventions are based on those used by Leibman in \citep{leibman2005pointwiseZd, leibman2005pointwise,leibman2006rational}. 
We state the basic definitions, give some examples, and then discuss subnilmanifolds, rational subgroups, and the natural tower of quotients associated to a nilmanifold.  

\subsection{Definition} 
\label{sec:nilmanifolds_definition}

Let $G$ be a group, and let $e_G$ denote its identity element. 
For $g, h \in G$, we denote their commutator by
\[
[g, h] := g h g^{-1} h^{-1}.
\]
If $H, K \leq G$ are subgroups, then $[H, K]$ denotes the subgroup of $G$ generated by the commutators $[h, k]$ where $h 
\in H$ and $k \in K$. 

The lower central series of $G$ is the sequence  
\[
G = G_1 \supseteq G_2 \supseteq G_3 \supseteq \dots. 
\]
defined inductively by
\[
G_1 := G, \quad G_{i+1} := [G, G_i] \quad \text{for } i \geq 1.
\]
The subgroups $G_i$ are normal in $G$, and  
\[
[G_i, G_j] \subseteq G_{i+j} \quad \text{for all } i, j \geq 1.
\]
For $s \geq 1$, the group $G$ is  called $s$-step nilpotent if $G_{s+1} = \{e_G\}$. 
We do not require $s$ to be minimal with this property. 
In particular, an $s$-step nilpotent group is also $r$-step nilpotent for every $r \geq s$. 

We are mainly interested in Lie groups. 
We will use the following standard facts without further comment; see \citep[Proposition 7.15, Theorem 20.12, Theorem 21.26]{lee2003smooth} for details.
\begin{itemize}
    \item The identity component $G^o$ of a Lie group $G$ is an open normal subgroup.  
    \item A closed subgroup of a Lie group is a Lie group.
    \item The quotient of a Lie group by a closed normal subgroup is a Lie group.
\end{itemize}

A closed subgroup $\Gamma$ of a Lie group $G$ is called cocompact if the quotient $G/\Gamma$ is compact.

\begin{definition}
    Let $s \geq 1$. 
    An $s$-step nilmanifold is a quotient 
    \[
    X = G/\Gamma, 
    \]
    where $G$ is an $s$-step nilpotent Lie group and $\Gamma \leq G$ is a discrete cocompact subgroup.
    We call $X$ a nilmanifold if it is an $s$-step nilmanifold for some $s \geq 1$.
\end{definition}

If $X = G/\Gamma$ is a nilmanifold, we write 
\[
p \colon G \to X, \quad g \mapsto g \Gamma
\]
for the quotient map.
This map is open. 
The point $e_X := p(e_G) = e_G \Gamma$ is called the base point of $X$.
By definition, $X$ is a compact homogeneous space of $G$.
The group $G$ thus acts on $X$ by translations 
\[
g \cdot h\Gamma := gh \Gamma \quad \text{for } g \in G, \ h \Gamma \in X.
\]
With its natural smooth structure, $X$ is a compact smooth manifold, and this action is smooth. 

There is a unique Borel probability measure $\mu$ on $X$ that is invariant under all translations by elements of $G$.
We call $\mu$ the Haar measure on $X$. 
This measure has full support: $\mu(U) > 0$ for every non-empty open set $U \subseteq X$. 
For background on Haar measure, see \citep[Section 2.6]{folland2016course}.

\subsection{Examples}

We consider some examples of nilmanifolds. 

\begin{example}
    \label[example]{eg:one_step_nilmnflds}
    The torus
    \[
    \T = \R/\Z  
    \]
    is a $1$-step nilmanifold. 
    More generally, every compact abelian Lie group $K$ can be viewed as a $1$-step nilmanifold, since 
    \[
    K \cong K/\{0\}.  
    \]
    Conversely, if $X = G/\Gamma$ is a $1$-step nilmanifold, then $G$ is abelian so $\Gamma$ is normal in $G$. 
    Therefore, the quotient $G/\Gamma$ is a compact abelian Lie group. 
    Thus the $1$-step nilmanifolds are precisely the compact abelian Lie groups. 
    Under this identification, the Haar measure on the nilmanifold agrees with the usual group Haar measure. 
\end{example}

\begin{example}
    \label[example]{eg:heisenberg_nilmnfld}
    Let $G = \R^3$ with multiplication law 
    \[
    (x, y, z) \cdot (x', y', z') = (x + x', y + y', z + z' + xy').  
    \]
    With the usual smooth structure on $\R^3$, this is a Lie group called the Heisenberg group. 
    Its identity element is $0 := (0, 0, 0)$. 
    The commutator is given by
    \[
    [(x, y, z), (x', y', z')] = (0, 0, xy' - x'y). 
    \]
    Therefore
    \[
    G_2 = \{0\} \times \{0\} \times \R, 
    \quad G_3 = \{0\}, 
    \]
    so $G$ is $2$-step nilpotent. 
    The subgroup $\Gamma := \Z^3 \leq G$ is discrete, and the quotient 
    \[
    X = G/\Gamma 
    \]
    is compact, since every coset has a representative in the compact cube $[0,1]^3 \subseteq G$. 
    Thus, $X$ is a $2$-step nilmanifold, called the Heisenberg nilmanifold. 
    
    The Heisenberg group can equivalently be realized as the group  
    \[
        \left\{
            \begin{pmatrix}
                1 & x & z \\
                0 & 1 & y \\
                0 & 0 & 1
            \end{pmatrix} :
            x, y, z \in \R
        \right\},  
    \]
    under matrix multiplication, which is where the above multiplication law comes from. 
\end{example}

\begin{example}
    \label[example]{eg:eg:up_triangle_nilmnfld}
    Let $U_n(\R)$ be the group of real upper triangular $n \times n$-matrices with all diagonal entries equal to $1$, and let $U_n(\Z)$ be the subgroup with integer entries.  
    For $s \geq 1$, let
    \[
    G = U_{s+1}(\R), \quad \Gamma = U_{s+1}(\Z).  
    \]
    Then $G$ is an $s$-step nilpotent Lie group and $\Gamma$ is a discrete cocompact subgroup; see \citep[Section 10.6]{einsiedler2010ergodic} for details. 
    Thus, 
    \[
    X = G/\Gamma 
    \]
    is an $s$-step nilmanifold. 
    For $s = 2$, this recovers the matrix form of the Heisenberg nilmanifold.  
\end{example}

\begin{example}
    \label[example]{eg:product_nilmnfld}
    Products of nilmanifolds are naturally again nilmanifolds. 
    Indeed, if $X = G/\Gamma$ and $X' = G'/\Gamma'$ are nilmanifolds of step $s$ and $s'$ respectively, then  
    \[
    X \times X' \cong (G \times G')/(\Gamma \times \Gamma')
    \]
    is an $r$-step nilmanifold, where $r = \max\{s, s'\}$. 
    The Haar measure on $X \times X'$ agrees with the product of the Haar measures on $X$ and $X'$.  
\end{example}

\subsection{Standing assumption} 
\label{sec:nilmanifolds_standing_assumption}

Let $X = G/\Gamma$ be an $s$-step nilmanifold, and let $G^o$ be the identity component in $G$.  
In several later arguments, it will be convenient to assume that $G/G^o$ is finitely generated.
This is a harmless assumption for our purposes.  

Let $p \colon G \to X = G/\Gamma$ be the quotient map. 
Then $p(G^o) \subseteq X$ is open. 
Since $G$ acts transitively on $X$, the translates $g \cdot p(G^o) = p(g G^o)$ for $g \in G$ form an open cover of $X$. 
By compactness, there exist $g_1, \dots, g_k \in G$ such that 
\[
X = \bigcup_{j = 1}^k p(g_j G^o).
\]
Let $H := \langle G^o, g_1, \dots, g_k \rangle \leq G$, 
and let $\Lambda := H \cap \Gamma$. 
Then $H \leq G$ is an open, hence closed, subgroup and therefore an $s$-step nilpotent Lie group. 
By construction, $H$ acts transitively on $X$. 
Hence by standard homogeneous space arguments, see \citep[Theorem 21.18]{lee2003smooth}, the map 
\[
H/\Lambda \to X, \quad h \Lambda \mapsto h \cdot e_X = h\Gamma
\]
is a diffeomorphism, identifying $H/\Lambda$ and $X$ as homogeneous spaces of $H$.

Note that $\Lambda \leq H$ is discrete, and it is cocompact by the above identification of the quotient. 
Hence $H/\Lambda$ is an $s$-step nilmanifold.
Moreover, $H/H^o$ is finitely generated since it is generated by images of $g_1, \dots, g_k$. 

\begin{assumption}
    From now on, whenever a nilmanifold is written as $X = G/\Gamma$, we assume, unless otherwise stated, that $G/G^o$ is finitely generated.  
\end{assumption}

All examples in the previous section satisfy this assumption. 

\subsection{Subnilmanifolds and rational subgroups}
\label{sec:subnilmanifolds_rational_subgroups}

We next discuss subnilmanifolds, the natural subspaces of nilmanifolds.  

\begin{definition}
    Let $X = G/\Gamma$ be a nilmanifold. 
    A subset $Y \subseteq X$ is called a subnilmanifold of $X$ if $Y$ is closed and there exists a closed subgroup $H \leq G$ and a point $x \in X$ such that 
    \[
    Y = H \cdot x.
    \]
\end{definition}

Suppose that $Y = H \cdot x$ is a subnilmanifold, with $x = g\Gamma$.
Then $H$ acts transitively on $Y$ by translations, and the stabilizer subgroup of $x$ in $H$ is 
\[
\Lambda := H \cap g \Gamma g^{-1}.
\]
Hence by standard homogeneous space arguments, see for example \citep[Theorem 2.2.2]{becker1996descriptive}, the orbit map induces a homeomorphism
\[
H/\Lambda \to Y, \quad h\Lambda \mapsto h \cdot x.
\]
Since $H$ is a closed subgroup of $G$, it is a nilpotent Lie group. 
Moreover, $\Lambda$ is discrete, and it is cocompact because the above homeomorphism identifies $H/\Lambda$ with the compact space $Y$.  
Therefore, $Y$ is itself naturally a nilmanifold of step equal to the step of $X$. 

The Haar measure on $Y$ is defined as the image of the Haar measure on $H/\Lambda$ under the above identification. 
In particular, this is the unique Borel probability measure on $Y$ that is invariant under all translations by elements of $H$. 

The standing assumption from the previous section is preserved under passing to subnilmanifolds. 
Indeed, let $X = G/\Gamma$ be a nilmanifold with $G/G^o$ finitely generated, and let $Y = H \cdot x$ be a subnilmanifold, and let $\Lambda$ be defined as above. 
Then $\Gamma$ and consequently also $\Lambda$ are finitely generated; see \citep[Theorem 2.1, Theorem 2.7]{raghunathan1972discrete}. 
Let $q \colon H \to H/H^o$ be the quotient map. 
Then 
\[
(H/H^o)/q(\Lambda)
\]
is discrete and compact, and therefore finite. 
It follows that $H/H^o$ contains the finitely generated group $q(\Lambda)$ as a finite index subgroup. 
Thus $H/H^o$ is finitely generated. 

We will need the following simple observation.
Subnilmanifolds are preserved under maps arising from homomorphisms on the level of the groups. 

\begin{lemma}
    \label[lemma]{lem:algebraic_image_subnilmanifold}
    Let $X = G/\Gamma$ and $X' = G'/\Gamma'$ be nilmanifolds, and let $\phi \colon G \to G'$ be a continuous group homomorphism with $\phi(\Gamma) \subseteq \Gamma'$. 
    Let 
    \[
    \pi \colon X \to X', \quad g\Gamma \mapsto \phi(g) \Gamma' 
    \]
    be the induced map. 
    If $Y \subseteq X$ is a subnilmanifold with Haar measure $\nu$, then $\pi(Y)$ is a subnilmanifold of $X'$, with Haar measure $\pi_* \nu$. 
\end{lemma}

\begin{proof}
    Write $Y = H \cdot x$ with $H \leq G$ closed and $x \in X$. 
    Then $\pi(Y) = \pi(H \cdot x) = \phi(H) \cdot \pi(x)$. 
    Since $Y$ is compact, $\pi(Y)$ is compact, hence closed.
    Therefore
    \[
    \pi(Y) = \overline{\phi(H)} \cdot \pi(x)
    \]
    is a subnilmanifold. 
    The measure $\pi_* \nu$ is invariant under $\phi(H)$, hence under $\overline{\phi(H)}$, so it is the Haar measure on $\pi(Y)$.  
\end{proof}

Let $X = G/\Gamma$ be a nilmanifold, and let $H \leq G$ be a closed subgroup. 
Then 
\[
H \cdot e_X = p(H)
\]
need not be closed in $X$. 
Equivalently, $H \cdot e_X$ is not necessarily a subnilmanifold. 
This motivates the following definition. 

\begin{definition}
    Let $X = G/\Gamma$ be a nilmanifold.
    A closed subgroup $H \leq G$ is called rational if $H \cdot e_X$ is closed in $X$. 
\end{definition}

We will use the following equivalent characterizations; this is \citep[Chapter 10, Lemma 14]{host2018nilpotent}. 

\begin{lemma} 
    \label[lemma]{lem:rational_subgroups_equivalent_chatacterizations}
    Let $X = G/\Gamma$ be a nilmanifold, and let $H \leq G$ be a closed subgroup. 
    Then the following are equivalent. 
    \begin{enumerate}[(i)]
        \item $H$ is a rational subgroup,  
        \item $H \Gamma$ is closed in $G$, 
        \item $H \cap \Gamma$ is cocompact in $H$, 
        \item there exists a compact subset $K \subseteq H$ such that $H = K (H \cap \Gamma)$. 
    \end{enumerate}
\end{lemma}

\begin{proof}
    The equivalence of (i) and (ii) holds because $p^{-1}(H \cdot e_X) = H \Gamma$.

    To see that (i) and (iii) are equivalent, note that the map 
    \[
    H/(H \cap \Gamma) \to H \cdot e_X, \quad h(H \cap \Gamma) \mapsto h \cdot e_X
    \]
    is a continuous bijection. 
    If $H \cdot e_X$ is closed, then by \citep[Theorem 2.2.2]{becker1996descriptive} this map is a homeomorphism, so $H/(H \cap \Gamma)$ is compact.
    Conversely, if $H \cap \Gamma$ is cocompact, then the above map is again a homeomorphism since it has compact domain, and therefore $H \cdot e_X$ is closed. 

    Finally, (iii) and (iv) are equivalent by the usual characterization of cocompactness.  
\end{proof}

For us, the most important rational subgroups are the terms of the lower central series.  

\begin{proposition} 
    \label[proposition]{prop:lower_central_series_is_rational}
    Let $X = G/\Gamma$ be a nilmanifold.
    Then for $i \geq 1$, $G_i$ is a rational subgroup.
\end{proposition}

In particular, each $G_i$ is a closed subgroup of $G$. 
This is the main point where the standing assumption that $G/G^o$ is finitely generated is used.  
This result is originally due to \citet{malcev1951class} for connected simply connected $G$. 
A proof in the present setting can be found in \citep[\S 2.11]{leibman2005pointwise}.

\subsection{The tower of a nilmanifold}
\label{sec:nilmanifold_tower}

Let $X = G/\Gamma$ be an $s$-step nilmanifold. 
The lower central series of $G$ gives rise to a natural tower of quotients
\[
X = Z_s \longrightarrow Z_{s-1} \longrightarrow \cdots \longrightarrow Z_1 \longrightarrow Z_0 = \{*\},
\]
where $Z_i$ is an $i$-step nilmanifold for $i = 1, \dots, s$.
This is a powerful tool, as it allows for arguments by induction on the step.  
We now carry out this construction carefully. 
This section is based on \citep[Section 14.3]{eisner2025journey} and \citep[Chapter 11, Section 1.3]{host2018nilpotent}.

For $i = 0, \dots, s$, let $q_i \colon G \to G/G_{i+1}$ be the quotient homomorphism, and define  
\[
Z_i := (G/G_{i+1})/q_i(\Gamma).
\]

\begin{lemma}
    For each $i = 1, \dots, s$, the space $Z_i$ is an $i$-step nilmanifold.  
\end{lemma}

\begin{proof}
    Since $G_{i+1}$ is closed and normal in $G$, the quotient $G/G_{i+1}$ is an $i$-step nilpotent Lie group.     
    Moreover, $q_i(\Gamma) \leq G/G_{i+1}$ is discrete and cocompact.  
    Indeed, by rationality of $G_{i+1}$, the subgroup $G_{i+1} \Gamma$ is closed in $G$, so $q_i(\Gamma)$ is closed in $G/G_{i+1}$. 
    Since $q_i(\Gamma)$ is a countable Lie group, it is discrete.
    Finally, the map 
    \[
    (G/G_{i+1})/q_i(\Gamma) \to G/(G_{i+1} \Gamma),
    \quad q_i(g) q_i(\Gamma) \mapsto g G_{i+1}\Gamma
    \]
    is a homeomorphism. 
    Since $G/(G_{i+1} \Gamma)$ is a quotient of $X = G/\Gamma$, it is compact.
    Thus, $q_i(\Gamma)$ is cocompact, and $Z_i$ is an $i$-step nilmanifold. 
\end{proof}

For $i = 0, \dots, s-1$, let
\[
\pi_i \colon Z_{i+1} \to Z_i
\]
be the map induced by the quotient group homomorphism $G/G_{i+2} \to G/G_{i+1}$. 
Explicitly, for $q_{i+1}(g) q_{i+1}(\Gamma) \in Z_{i+1}$,  
\[
\pi_i(q_{i+1}(g) q_{i+1}(\Gamma)) = q_i(g) q_i(\Gamma).
\] 
This is a well-defined continuous surjection. 

Since $G_{s+1} = \{e_G\}$, we have that $Z_s \cong X$.  
Moreover, $G/G_1$ is trivial, so $Z_0 = \{*\}$. 
This gives the tower
\[
X = Z_s \xrightarrow{\pi_{s-1}} Z_{s-1} \xrightarrow{\pi_{s-2}} \cdots \xrightarrow{\pi_1} Z_1 \xrightarrow{\pi_0} Z_0 = \{*\}.
\]

The maps in the tower can be realized as quotients by compact abelian Lie groups, called the structure groups. 
For $i = 1, \dots, s$, define 
\[
K_i := q_i(G_i)/(q_i(G_i) \cap q_i(\Gamma)).
\]

\begin{lemma}
    For $i = 1, \dots, s$, $K_i$ is a compact abelian Lie group. 
\end{lemma}

\begin{proof}
    Since $q_i(G_i)$ is the last non-trivial term in the lower central series of $G/G_{i+1}$, it is central, hence abelian. 
    Thus $K_i$ is an abelian group. 
    It is a Lie group, because $q_i(\Gamma)$ is closed, and it is compact because it is the continuous image of the compact nilmanifold $G_i/(G_i \cap \Gamma)$ under the map 
    \[
    G_i/(G_i \cap \Gamma) \to K_i, \quad g (G_i \cap \Gamma) \mapsto q_i(g) (q_i(G_i) \cap q_i(\Gamma)).
    \]
\end{proof}

We call $K_i$ the $i$-th structure group of the tower.  
The group $K_i$ acts on $Z_i$ by translations. 
If $a \in G_i$, then the class of $q_i(a)$ in $K_i$ acts by 
\[
q_i(g) q_i(\Gamma) \mapsto q_i(a) q_i(g) q_i(\Gamma).
\]
This is well-defined because $q_i(G_i)$ is central in $G/G_{i+1}$. 
The action is continuous and it commutes with all translations by elements of $G/G_{i+1}$ on $Z_i$. 

\begin{lemma}
    \label[lemma]{lem:structure_group_action}
    For each $i = 1, \dots, s$, the fibers of 
    \[
    \pi_{i-1} \colon Z_i \to Z_{i-1}
    \] 
    are exactly the $K_i$-orbits, and the action of $K_i$ on $Z_i$ is free.
\end{lemma}

\begin{proof}
    Since every element of $G_i$ maps to the identity in $G/G_i$, the map $\pi_{i-1}$ is constant on the $K_i$-orbits.  
    Conversely, suppose that two points $q_i(g) q_i(\Gamma), q_i(g') q_i(\Gamma) \in Z_i$ have the same image under $\pi_{i-1}$, that is,   
    \[
    q_{i-1}(g) q_{i-1}(\Gamma) = q_{i-1}(g') q_{i-1}(\Gamma).
    \]
    Then there exists $\gamma \in \Gamma$ such that $a := g^{-1} g' \gamma^{-1} \in G_i$. 
    Hence, 
    \[
    q_i(g') q_i(\Gamma) = q_i(g) q_i(a) q_i(\Gamma) = q_i(a) q_i(g) q_i(\Gamma),
    \]
    so the two points lie in the same $K_i$-orbit.  
    Finally, if an element represented by $q_i(a)$ in $K_i$ fixes a point of $Z_i$, then centrality of $q_i(a)$ gives $q_i(a) \in q_i(\Gamma)$, so this element is trivial in $K_i$. 
    Therefore, the $K_i$-action is free. 
\end{proof}

Thus each map 
\[
\pi_{i-1} \colon Z_i \to Z_{i-1}
\]
is the quotient map of the free action of the structure group $K_i$ on $Z_i$. 
In particular, each fiber is homeomorphic to $K_i$. 

For $i = 1$,  
\[
K_1 = (G/G_2)/q_1(\Gamma) = Z_1, 
\]
and the action of $K_1$ on $Z_1$ is simply the action of the group $K_1 = Z_1$ on itself by translations, which indeed has the trivial quotient $Z_0 = \{*\}$.   

In summary, we have constructed a tower 
\[
X = Z_s \xrightarrow{\pi_{s-1}} Z_{s-1} \xrightarrow{\pi_{s-2}} \cdots \xrightarrow{\pi_1} Z_1 \xrightarrow{\pi_0} Z_0 = \{*\},
\]
where for $i = 1, \dots, s$, $Z_i$ is an $i$-step nilmanifold and $\pi_{i-1} \colon Z_i \to Z_{i-1}$ is the quotient by the compact abelian Lie group $K_i$.  

\section{Nilsystems}

We now add dynamics to the nilmanifolds introduced in the previous section. 
This is done by fixing an element in the acting nilpotent group and considering the translation by this element. 
The resulting dynamical system is called a nilsystem. 
The material in this section is based on \citep[Chapter 11]{host2018nilpotent} and \citep[Chapter 14]{eisner2025journey}.

\subsection{Definition} 
\label{sec:nilsystems_definition}

\begin{definition}
    Let $X = G/\Gamma$ be an $s$-step nilmanifold, and let $\tau \in G$. 
    Consider the translation
    \[
    T \colon X \to X, \quad x \mapsto \tau \cdot x.  
    \]
    The nilmanifold $X$, together with the translation $T$, is called an $s$-step nilsystem.
\end{definition}
 
Since $T$ is a translation, it is a homeomorphism of $X$, with inverse $x \mapsto \tau^{-1} \cdot x$. 
Thus $(X, T)$ is a topological dynamical system. 
The Haar measure $\mu$ on $X$ is invariant under the action of $G$, so in particular under $T$. 
Therefore, $(X, \mu, T)$ is also a measure-preserving system. 

Thus a nilsystem simultaneously has the structure of a measure-preserving and a topological dynamical system. 
It will usually be clear from the context what the relevant structure is: topological notions such as minimality and distality refer to $(X, T)$, whereas measure-theoretic notions such as ergodicity refer to $(X, \mu, T)$. 

We denote nilsystems flexibly, depending on which part of the structure we want to emphasize.  
When we write $(X, T)$, the Haar measure $\mu$ is understood, and when we write $(X, \mu, T)$, the nilmanifold structure on $X$ and the realization of $T$ as a translation remain part of the data. 
When the specific nilmanifold is relevant, we will write $(X = G/\Gamma, T)$ or $(X = G/\Gamma, \mu, T)$.  

Unless stated otherwise, we keep the standing assumption from \cref{sec:nilmanifolds_standing_assumption}. 
Thus, whenever $(X = G/\Gamma, \mu, T)$ is a nilsystem, we assume that $G/G^o$ is finitely generated. 
In \cref{sec:presentations_of_nilsystems}, we explain why this does not restrict the resulting dynamical systems.  

\subsection{Examples}

We record some basic examples. 

\begin{example}
    \label[example]{eg:lie_group_rotations}
    Let $K$ be a metrizable compact abelian group, let $m_K$ be its Haar measure, and let $a \in K$. 
    Then the rotation 
    \[
    R \colon K \to K, \quad u \mapsto a + u  
    \]
    defines a measure-preserving system $(K, m_K, R)$, known as a compact abelian group rotation. 
    By \cref{eg:one_step_nilmnflds}, the $1$-step nilsystems are precisely the group rotations $(K, m_K, R)$, where $K$ is a compact abelian Lie group. 

    For example, if $a \in \R$, the rotation 
    \[
    R \colon \T \to \T, \quad u \mapsto a + u \bmod 1 
    \]
    defines a $1$-step nilsystem on $\T$. 
    It is ergodic if and only if $a\notin \Q$. 
    More generally, a rotation $(K, m_K, R)$ is ergodic if and only if  
    \[
    \{ n a: n \in \Z \}
    \] 
    is dense in $K$; see \citep[Theorem 4.14]{einsiedler2010ergodic}. 
\end{example}

\begin{example}
    \label[example]{eg:heisenberg_nilsystem}
    Let $X = G/\Gamma$ be the Heisenberg nilmanifold from \cref{eg:heisenberg_nilmnfld}, and let $\mu$ be its Haar measure. 
    Fix $\tau = (a, b, c) \in G$, and define the translation 
    \[
    T \colon X \to X, \quad g\Gamma \mapsto \tau g \Gamma.  
    \]
    Then $(X, \mu, T)$ is a $2$-step nilsystem, called the Heisenberg nilsystem. 
    This system is ergodic if and only if $1, a, b$ are linearly independent over $\Q$; see \citep[Theorem 10.1]{einsiedler2010ergodic}. 
\end{example}

\begin{example}
    \label[example]{eg:product_nilsystem}
    Let $(X = G/\Gamma, \mu, T)$ and $(X' = G'/\Gamma', \mu', T')$ be nilsystems of step $s$ and $s'$ respectively, where $T$ is the translation by $\tau \in G$ and $T'$ is the translation by $\tau' \in G'$. 
    By \cref{eg:product_nilmnfld},  
    \[
    X \times X' \cong (G \times G')/(\Gamma \times \Gamma') 
    \]
    is a nilmanifold of step $r := \max\{s, s'\}$, and its Haar measure is $\mu \times \mu'$. 
    Under this identification, the map $T \times T'$ agrees with the translation by $(\tau, \tau')$. 
    Hence $(X \times X', \mu \times \mu', T \times T')$ is an $r$-step nilsystem. 
\end{example}

\subsection{Presentations of nilsystems} 
\label{sec:presentations_of_nilsystems}

A nilsystem is defined concretely by a nilmanifold $X = G/\Gamma$ and an element $\tau \in G$, whose translation defines the dynamics. 
It is possible, however, for different such concrete choices to give rise to the same dynamical system. 
To make this precise, we introduce the following terminology. 

\begin{definition}
    Two nilsystems $(X, \mu, T)$ and $(X', \mu', T')$ are called presentations of the same nilsystem if there exists a homeomorphism
    \[
    \pi \colon X \to X' 
    \]
    such that $\pi \circ T = T' \circ \pi$, and $\pi_* \mu = \mu'$. 
    Thus $\pi$ is required to be both a measure-theoretic and a topological isomorphism.  
\end{definition}

Several useful reductions can be made by changing the presentation. 
The standing assumption from \cref{sec:nilmanifolds_standing_assumption} is one example. 

\begin{lemma}
    Let $(X = G/\Gamma, \mu, T)$ be a nilsystem, without the assumption that $G/G^o$ is finitely generated. 
    After changing the presentation, we may assume that $G/G^o$ is finitely generated. 
\end{lemma}

\begin{proof}
    Suppose that $T$ is the translation by $\tau \in G$.
    By the argument in \cref{sec:nilmanifolds_standing_assumption}, there exist $g_1, \dots, g_k \in G$ such that the subgroup generated by $G^o$ and $g_1, \dots, g_k$ acts transitively on $X$. 
    Set
    \[
    H := \langle G^o, g_1, \dots, g_k, \tau \rangle, \quad \Lambda := H \cap \Gamma. 
    \]
    Then $H$ still acts transitively on $X$, and  
    \[
    H/\Lambda \to G/\Gamma, \quad h\Lambda \mapsto h \cdot e_X = h\Gamma 
    \]
    is a homeomorphism, which intertwines the translation by $\tau$ on $H/\Lambda$ with the translation by $\tau$ on $G/\Gamma$, and which sends Haar measure to Haar measure.   
    Thus $H/\Lambda$, equipped with the translation by $\tau$, presents the nilsystem as $(X, \mu, T)$. 
    Since $H/H^o$ is generated by the images of $g_1, \dots, g_k, \tau$, this presentation satisfies the standing assumption.  
\end{proof}

For ergodic nilsystems one can make a stronger reduction. 

\begin{definition}
    Let $(X = G/\Gamma, \mu, T)$ be a nilsystem, where $T$ is the translation by $\tau \in G$. 
    We say that the nilsystem is presented in reduced form if 
    \[
    G = \langle G^o, \tau \rangle,
    \]
    and $\Gamma$ contains no non-trivial normal subgroup of $G$. 
\end{definition}

\begin{lemma}
    \label[lemma]{lem:reduced_form}
    Every ergodic nilsystem admits a presentation in reduced form. 
\end{lemma}

\begin{proof}
    Let $(X = G/\Gamma, \mu, T)$ be an ergodic nilsystem, where $T$ is the translation by $\tau \in G$. 
    Let 
    \[
    H := \langle G^o, \tau \rangle.
    \] 
    Since $H$ is an open subgroup of $G$, $H \cdot e_X$ is a non-empty open subset of $X$. 
    The $H$-orbits form a partition of $X$, hence $H \cdot e_X$ is also closed. 
    Since $\tau \in H$, $H \cdot e_X$ is $T$-invariant. 
    It has strictly positive Haar measure, so ergodicity gives $\mu(H \cdot e_X) = 1$. 
    As $\mu$ has full support and $H \cdot e_X$ is closed, it follows that $X = H \cdot e_X$. 
    Thus $H$ acts transitively on $X$. 
    Arguing as before, we can replace $G$ by $H$ and $\Gamma$ by $H \cap \Gamma$ to obtain a presentation of the same nilsystem which satisfies $G = \langle G^o, \tau \rangle$. 

    Now let $N$ be the largest normal subgroup of $G$ contained in $\Gamma$, equivalently 
    \[
    N = \bigcap_{g \in G} g \Gamma g^{-1}. 
    \]
    Since $N \subseteq \Gamma$, there is a natural homeomorphism
    \[
    G/\Gamma \cong (G/N) / (\Gamma/N).
    \]
    Under this identification, the translation by $\tau$ corresponds to the translation by $\tau N \in G/N$, and Haar measure is preserved. 
    Hence, this gives a new presentation of the nilsystem as the translation by $\tau N$ on $(G/N)/(\Gamma/N)$.
    In this new presentation, $\Gamma/N$ contains no non-trivial normal subgroup of $G/N$.  
    Moreover,   
    \[
    G/N = \langle (G/N)^o, \tau N \rangle. 
    \] 
    Thus this new presentation is reduced. 
\end{proof}

\subsection{Subnilsystems}

In \cref{sec:subnilmanifolds_rational_subgroups}, we introduced subnilmanifolds. 
We now adapt this notion to nilsystems.  

\begin{definition}
    Let $(X = G/\Gamma, T)$ be an $s$-step nilsystem, where $T$ is the translation by $\tau \in G$. 
    A subnilsystem of $(X, T)$ is a system of the form 
    \[
    (Y, \restrict{T}{Y}),
    \]
    where $Y = H \cdot x$ is a subnilmanifold of $X$ and $\tau \in H$. 
\end{definition}

A subnilsystem of an $s$-step nilsystem is naturally again an $s$-step nilsystem. 
Indeed, recall that a subnilmanifold $Y = H \cdot x$ with $x = g \Gamma$ is  identified with the $s$-step nilmanifold
\[
H/(H \cap g \Gamma g^{-1}).
\]
If $\tau \in H$, then under this identification, $\restrict{T}{Y}$ corresponds to the translation by $\tau$ on this nilmanifold. 
Hence $(Y, \restrict{T}{Y})$ is topologically isomorphic to an $s$-step nilsystem. 
Moreover, recall that the Haar measure $\nu$ on $Y$ is the image of the Haar measure on $H/(H \cap g \Gamma g^{-1})$. 
Therefore, $(Y, \nu, \restrict{T}{Y})$ is also measure-theoretically isomorphic to an $s$-step nilsystem.

By \cref{sec:subnilmanifolds_rational_subgroups}, passing to subnilmanifolds preserves the assumption that $G/G^o$ is finitely generated. 
Consequently, this is also preserved by passing to subnilsystems. 

\subsection{The tower of factors of a nilsystem} 
\label{sec:tower_of_factors_nilsystem}

We now explain how the tower of a nilmanifold constructed in \cref{sec:nilmanifold_tower} gives rise to a tower of factors for any nilsystem. 

Let $(X = G/\Gamma, \mu, T)$ be an $s$-step nilsystem, where $T$ is the translation by $\tau \in G$. 
For $i = 0, \dots, s$, let $q_i \colon G \to G/G_{i+1}$ be the quotient homomorphism. 
The construction from \cref{sec:nilmanifold_tower} gives a tower of quotient maps 
\[
X = Z_s \xrightarrow{\pi_{s-1}} Z_{s-1} \xrightarrow{\pi_{s-2}} \cdots \xrightarrow{\pi_1} Z_1 \xrightarrow{\pi_0} Z_0 = \{*\} 
\]
where, for $i = 1, \dots, s$, 
\[
Z_i = (G/G_{i+1})/q_i(\Gamma) 
\]
is an $i$-step nilmanifold. 
Let $\mu_i$ denote the Haar measure on $Z_i$, and let $T_i$ be the translation by $q_i(\tau) \in G/G_{i+1}$ on $Z_i$.  
Then $(Z_i, \mu_i, T_i)$ is an $i$-step nilsystem. 
Moreover, we view $Z_0 = \{*\}$ as the trivial system. 

The quotient maps $\pi_i$ are factor maps. 
Indeed, for $1 \leq i \leq s$, recall that $\pi_{i-1}$ was defined as 
\[
\pi_{i-1} \colon Z_i \to Z_{i-1}, \quad q_i(g) q_i(\Gamma) \mapsto q_{i-1}(g) q_{i-1}(\Gamma).
\]
Hence for $q_i(a) \in G/G_{i+1}$, 
\[
\pi_{i-1}(q_i(a) \cdot q_i(g) q_i(\Gamma)) = q_{i-1}(a) q_{i-1}(g) q_{i-1}(\Gamma) = q_{i-1}(a) \cdot \pi_{i-1}(q_i(g) q_i(\Gamma)) 
\]
for each $q_i(g) q_i(\Gamma) \in Z_i$.
Plugging in $q_i(\tau)$, we see that 
\[
\pi_{i-1} \circ T_i = T_{i-1} \circ \pi_{i-1}.
\]
Moreover, the above calculation shows that $(\pi_{i-1})_* \mu_i$ is invariant under all translations by elements of $G/G_i$, hence it must be the Haar measure $\mu_{i-1}$. 
Therefore
\[
\pi_{i-1} \colon (Z_i, \mu_i, T_i) \to (Z_{i-1}, \mu_{i-1}, T_{i-1})
\]
is both a measure-theoretic and topological factor map.  

Next, we consider the structure groups in this context. 
Recall from \cref{sec:nilmanifold_tower} that, for $i = 1, \dots, s$, the structure group 
\[
K_i := q_i(G_i)/(q_i(G_i) \cap q_i(\Gamma))
\]
is a compact abelian Lie group, acting on $Z_i$ by translations by elements of the central subgroup $q_i(G_i)$ in $G/G_{i+1}$. 
Therefore, this action commutes with $T_i$. 
Moreover, the action is continuous and free, and its orbits are exactly the fibers of $\pi_{i-1}$.  
It follows that 
\[
\pi_{i-1} \colon Z_i \to Z_{i-1}
\] 
is an extension by $K_i$ in the topological dynamical sense. 
By \cref{lem:topological_group_extension_is_measurable}, it is also an extension by $K_i$ in the measure-theoretic sense. 

Thus we get a tower of factor maps
\[
X = Z_s \xrightarrow{\pi_{s-1}} Z_{s-1} \xrightarrow{\pi_{s-2}} \cdots \xrightarrow{\pi_1} Z_1 \xrightarrow{\pi_0} Z_0 = \{*\},
\]
where for $i = 1, \dots, s$, $(Z_i, \mu_i, T_i)$ is an $i$-step nilsystem, and $\pi_{i-1} \colon Z_i \to Z_{i-1}$ is an extension by the structure group $K_i$ in both the topological and measure-theoretic sense.  

For later use, we record the following equivalent description of the factors and structure groups. 
Recall from \cref{sec:nilmanifold_tower} that for $i = 0, \dots, s$, the map 
\[
Z_i \to G/(G_{i+1} \Gamma), \quad 
q_i(g) q_i(\Gamma) \mapsto g G_{i+1} \Gamma
\]
is a homeomorphism. 
Under this identification, $T_i$ is the translation by $\tau$ on $G/(G_{i+1} \Gamma)$, and the factor map $\pi_{i-1} \colon Z_i \to Z_{i-1}$ takes the simple form 
\[
g G_{i+1} \Gamma \mapsto g G_i \Gamma. 
\]
The structure group can then be written as 
\[
K_i \cong G_i/(G_i \cap G_{i+1} \Gamma), 
\]
so that an element $u = g(G_i \cap G_{i+1} \Gamma) \in K_i$ acts by 
\[
u \cdot h (G_{i+1} \Gamma) = gh G_{i+1} \Gamma.  
\] 
We will use these equivalent descriptions wherever convenient without further comment. 

\begin{remark}
    In general, the structure groups $K_i$ need not be connected. 
    If the nilsystem is ergodic, however, one may choose a reduced presentation. 
    In this case, $G_i$ is connected for $i \geq 2$; see \citep[Chapter 10, Lemma 5]{host2018nilpotent}.
    Consequently, the structure groups $K_i$ for $i \geq 2$ are connected, and are therefore finite-dimensional tori. 
    The first structure group $K_1 = Z_1$ may still be disconnected. 
    We will not make use of this fact.
\end{remark}

\subsection{Properties of nilsystems}

We collect some basic properties of nilsystems. 
The overarching principle is that dynamical constructions for nilsystems remain controlled by the algebraic structure of the underlying nilmanifold. 

First, nilsystems are distal. 

\begin{lemma} 
    \label[lemma]{lem:nilsystems_are_distal}
    Every nilsystem $(X, T)$ is distal. 
\end{lemma}

\begin{proof}
    Let 
    \[
    (X, T) = (Z_s, T_s) \to (Z_{s-1}, T_{s-1}) \to \dots \to (Z_1, T_1)   
    \]
    be the tower of factors from \cref{sec:tower_of_factors_nilsystem}. 
    The trivial system is distal, and compact abelian group extensions of distal systems are distal; see \citep[Chapter 5, Proposition 9]{auslander1988minimal}.
    Since each map $Z_i \to Z_{i-1}$ is an extension by a compact abelian group, an induction argument along the tower shows that $(Z_s, T_s) = (X, T)$ is distal.
\end{proof}

The following theorem is the reason that orbit closures in nilsystems are highly structured: they are again nilmanifolds. 

\begin{theorem} 
    \label{thrm:orbit_closures_nilsystems_are_nilsystems}
    Let $(X = G/\Gamma, T)$ be a nilsystem, where $T$ is the translation by $\tau \in G$, and let $x \in X$. 
    Then the closed orbit 
    \[
    Y := \overline{\{T^n x : n \in \Z\}}, 
    \]
    equipped with the restriction of $T$, defines a subnilsystem of $(X, T)$. 
    More precisely, there exists a closed subgroup $H \leq G$ with $\tau \in H$, such that $Y = H \cdot x$. 
    Moreover, the subnilsystem $(Y, \restrict{T}{Y})$ is minimal and uniquely ergodic.
\end{theorem}

In this generality, the theorem is due to \citet[Theorem 2.21]{leibman2005pointwise}, with the case where $G$ is connected and simply connected proved earlier by \citet[Proposition 2]{lesigne1991nil}.
Proofs can also be found in \citep[Theorem 14.35]{eisner2025journey}, and \citep[Chapter 11, Theorem 9]{host2018nilpotent}. 

As a consequence, the standard topological and measure-theoretic dynamical properties coincide for nilsystems. 
We will make use of this equivalence repeatedly. 

\begin{corollary} 
    \label[corollary]{cor:equivalence_ergodicity_minimality_etc_nilsystems}
    Let $(X, \mu, T)$ be a nilsystem.
    Then the following are equivalent. 
    \begin{enumerate}[(i)]
        \item $(X, \mu, T)$ is ergodic, 
        \item $(X, T)$ is transitive,  
        \item $(X, T)$ is minimal, 
        \item $(X, T)$ is uniquely ergodic.  
    \end{enumerate}      
\end{corollary}

\begin{proof}
    Suppose first that $(X, \mu, T)$ is ergodic. 
    By Birkhoff's pointwise ergodic theorem, see for example \citep[Theorem 7.40]{eisner2025journey}, and separability of $C(X)$, we may choose a point $x \in X$ such that 
    \[
    \lim_{N \to \infty} \frac{1}{N} \sum_{n=0}^{N-1} f(T^n x) = \int_X f \dd\mu,
    \]
    for every $f \in C(X)$. 
    Let $Y := \overline{\{T^n x : n \in \Z\}}$ be the closed orbit of $x$.
    By \cref{thrm:orbit_closures_nilsystems_are_nilsystems}, $(Y, \restrict{T}{Y})$ is a uniquely ergodic subnilsystem. 
    Let $\nu$ be its Haar measure. 
    Then for every $f \in C(X)$, 
    \[
    \int_Y f \dd\nu = \lim_{N \to \infty} \frac{1}{N} \sum_{n=0}^{N-1} f(T^n x) = \int_X f \dd\mu.
    \]
    It follows that $\nu = \mu$, so $Y = X$. 
    Thus $(X, T)$ is transitive, so (i) implies (ii). 

    Since nilsystems are distal, every transitive nilsystem is minimal by the standard fact that transitive distal systems are minimal; see \citep[Chapter 5, Corollary 7]{auslander1988minimal}.
    Thus (ii) implies (iii). 

    If $(X, T)$ is minimal, then the orbit closure of a point is all of $X$. 
    \cref{thrm:orbit_closures_nilsystems_are_nilsystems} therefore implies that $(X, T)$ is uniquely ergodic. 
    Hence (iii) implies (iv). 

    Finally (iv) implies (i), because the unique invariant measure of a uniquely ergodic system is ergodic. 
\end{proof}

This theorem allows us to describe the ergodic invariant measures on a nilsystem. 

\begin{corollary} 
    \label[corollary]{cor:ergodic_measure_nilsystem_is_haar_subnilsystem}
    Let $(X, T)$ be a nilsystem, and let $\nu$ be an ergodic $T$-invariant Borel probability measure on $X$.
    Then $\nu$ is the Haar measure of a subnilsystem $(Y, \restrict{T}{Y})$ of $(X, T)$.
\end{corollary}

\begin{proof}
    By ergodicity, we can choose a point $x \in X$ whose orbit is dense in the support of $\nu$. 
    Let 
    \[
    Y := \overline{\{T^n x : n \in \Z\}} = \supp(\nu). 
    \]
    By \cref{thrm:orbit_closures_nilsystems_are_nilsystems}, $(Y, \restrict{T}{Y})$ is a uniquely ergodic subnilsystem of $(X, T)$. 
    Its unique invariant probability measure is its Haar measure. 
    Since $\nu$ is invariant and supported on $Y$, it must agree with this Haar measure. 
\end{proof}

By applying this to ergodic joinings of nilsystems, we get the following result.  

\begin{proposition} 
    \label[proposition]{prop:ergodic_joinings_nilsystems_are_nilsystems}
    Let $\lambda$ be an ergodic joining of the $s$-step nilsystems $(X = G/\Gamma, \mu, T)$ and $(X' = G'/\Gamma', \mu', T')$. 
    Then there exists a subnilsystem 
    \[
    (Y, \restrict{(T \times T')}{Y})
    \]
    of the product nilsystem $(X \times X', T \times T')$ such that $\lambda$ is the Haar measure on $Y$.
    In particular, the system defined by $\lambda$ is isomorphic to an $s$-step nilsystem. 
\end{proposition}

More concretely, if $T$ and $T'$ are the translations by $\tau \in G$ and $\tau' \in G'$, then there is a closed subgroup $H \leq G \times G'$ with $(\tau, \tau') \in H$, and a point $z \in X \times X'$, such that $Y = H \cdot z$ is a subnilmanifold and $\lambda$ is its Haar measure. 

\begin{proof}
    This follows immediately by applying \cref{cor:ergodic_measure_nilsystem_is_haar_subnilsystem} to the product nilsystem $(X \times X', T \times T')$ and the ergodic invariant measure $\lambda$.
\end{proof}

\subsection{Factor maps between nilsystems}

We finish this section with two classical facts about factors of nilsystems. 
First, factor maps between ergodic nilsystems in reduced form are of a particular algebraic form. 
Second, every factor of an ergodic nilsystem is again a nilsystem. 
Both results are due to \citet{parry1971metric,parry1973dynamical}. 
The present section is based on the formulation in \citep[Chapter 13]{host2018nilpotent}. 

Let $(X = G/\Gamma, \mu, T)$ and $(Y = H/\Lambda, \nu, S)$ be ergodic nilsystems, where $T$ and $S$ are the translations by $\tau \in G$ and $\sigma \in H$ respectively. 
An algebraic factor map from $X$ to $Y$ is a map of the form 
\[
\pi \colon X \to Y, \quad g\Gamma \mapsto a \phi(g) \Lambda, 
\]
where $\phi \colon G \to H$ is a continuous surjective group homomorphism with $\phi(\Gamma) \subseteq \Lambda$, and where $a \in H$ satisfies $\phi(\tau) = a^{-1} \sigma a$.  

The condition $\phi(\Gamma) \subseteq \Lambda$ ensures that this expression is well-defined. 
Since $\phi$ is continuous and surjective, the induced map $\pi$ is also a continuous surjection. 
Moreover, $\pi$ intertwines the transformations $T$ and $S$. 
Indeed, for every $g\Gamma \in X$,   
\[
\pi(T(g \Gamma)) = \pi(\tau g \Gamma) = a\phi(\tau g) \Lambda = \sigma a \phi(g) \Lambda = S(\pi(g\Gamma)). 
\]
Thus $\pi$ is a topological factor map. 
It is also measure-preserving: as $\mu$ is $T$-invariant, the pushforward $\pi_* \mu$ is $S$-invariant, so by unique ergodicity of ergodic nilsystems it follows that $\pi_* \mu = \nu$. 

The point is that all factor maps between ergodic nilsystems in reduced form are algebraic up to null sets. 

\begin{proposition} 
    \label[proposition]{prop:factor_maps_ergodic_nilsystems_are_algebraic}
    Let $(X, \mu, T)$ and $(Y, \nu, S)$ be ergodic nilsystems, both presented in reduced form. 
    Then every factor map 
    \[
    \pi \colon (X, \mu, T) \to (Y, \nu, S) 
    \]
    agrees $\mu$-almost everywhere with an algebraic factor map. 
\end{proposition}

This result is originally due to \citet{parry1971metric}. 
A proof in the form used here can be found in \citep[Chapter 13, Theorem 5]{host2018nilpotent}.

As an immediate consequence, a measure-theoretic factor map between ergodic nilsystems always admits a topological representative.  

\begin{corollary} 
    \label[corollary]{cor:factors_between_nilsystem_are_topological}
    Let $(X, \mu, T)$ and $(Y, \nu, S)$ be ergodic nilsystems. 
    Then every measure-theoretic factor map $\pi \colon (X, \mu, T) \to (Y, \nu, S)$ agrees $\mu$-almost everywhere with a topological factor map.  
\end{corollary}

\begin{proof}
    By \cref{lem:reduced_form}, we may assume that both nilsystems are presented in reduced form. 
    Hence by \cref{prop:factor_maps_ergodic_nilsystems_are_algebraic}, $\pi$ agrees almost everywhere with an algebraic factor map.
    Since algebraic factor maps are topological factor maps, the claim follows. 
\end{proof}

Consequently, for ergodic nilsystems there is essentially no distinction between topological and measure-theoretic factor maps: every topological factor map is measure-preserving by unique ergodicity, and every measure-theoretic factor map has a topological representative. 

We will need the following factor-closure result for nilsystems. 
The previous corollary considers factor maps whose target is already known to be a nilsystem. 
The following theorem says that this is automatically true.  

\begin{theorem} 
    \label{thrm:factor_ergodic_nilsystem_is_nilsystem}
    Let $(X, \mu, T)$ be an ergodic $s$-step nilsystem, and let 
    \[
    \pi \colon (X, \mu, T) \to (Y, \nu, S) 
    \] 
    be a measure-theoretic factor map. 
    Then $(Y, \nu, S)$ is isomorphic to an $s$-step nilsystem. 
\end{theorem}

This result was proved by \citet[Theorem 5.1]{parry1973dynamical}. 
We use the formulation in \citep[Chapter 13, Theorem 11]{host2018nilpotent}.

\section{Pro-nilsystems}
\label{sec:pronilsystems}

We now pass from nilsystems to inverse limits of nilsystems. 
This results in genuinely new behavior, since the class of nilsystems is not closed under taking inverse limits. 
This can already be seen in the $1$-step case; see \cref{eg:one_step_pronilsystems} below. 
Since nilsystems carry both a topological and a measure-theoretic structure, this gives rise to two a priori different notions of pro-nilsystems.  
We define both and then show that, at least in the ergodic setting, they are compatible. 
All inverse limits are understood in the sense of \cref{sec:inverse_limits}.

\begin{definition}
    Let $s \geq 1$. 
    A measure-preserving system $(X, \mu, T)$ is called a measure-theoretic $s$-step pro-nilsystem if there exists an inverse system $((X_n, \mu_n, T_n), \alpha_n)_{n \in \N}$ of $s$-step nilsystems such that 
    \[
    (X, \mu, T) \cong \varprojlim (X_n, \mu_n, T_n) 
    \]
    as measure-preserving systems. 

    Similarly, a topological dynamical system $(X, T)$ is called a topological $s$-step pro-nilsystem if there exists an inverse system $((X_n, T_n), \alpha_n)_{n \in \N}$ of $s$-step nilsystems such that 
    \[
    (X, T) \cong \varprojlim (X_n, T_n) 
    \]
    as topological dynamical systems.
\end{definition}

We omit the terms measure-theoretic and topological if the intended setting is clear. 
Note that different choices of inverse limits of nilsystems can give rise to the same pro-nilsystem. 
For a fixed inverse limit, we refer to the systems $X_n$ as the finite-stage nilsystems. 
Again, we do not always explicitly specify the connecting factor maps. 

\begin{example}
    \label[example]{eg:one_step_pronilsystems}
    Let $K$ be a metrizable compact abelian group, let $m_K$ be its Haar measure and consider the rotation 
    \[
    R \colon K \to K, \quad u \mapsto a + u 
    \]
    for some fixed $a \in K$.  
    If $K$ is a Lie group, then $(K, m_K, R)$ is a $1$-step nilsystem; see \cref{eg:lie_group_rotations}.  
    In general, $(K, m_K, R)$ is a $1$-step pro-nilsystem.  
    We can show this by the following Pontryagin duality argument; see also \citep[Chapter 18, Section 1.1]{host2018nilpotent} and \citep[Remark 2.35]{hofmann2023structure}. 
    Since $K$ is compact and metrizable, its Pontryagin dual $\wh K$ is countable and discrete. 
    Choose an increasing sequence $(A_n)_{n \in \N}$ of finitely generated subgroups of $\wh K$ such that $\wh K = \bigcup_{n \in \N} A_n$. 
    Let 
    \[
    K_n := \wh A_n.  
    \]
    As the dual of a finitely generated discrete abelian group, each $K_n$ is a compact abelian Lie group. 
    By duality, the inclusion maps $A_n \hookrightarrow A_{n+1}$ induce continuous surjective group homomorphisms $\alpha_n \colon K_{n+1} \to K_n$ for each $n$, and the condition on the union implies that with respect to these connecting maps, 
    \[
    K = \varprojlim K_n.   
    \]  
    The element $a \in K$ projects to elements $a_n \in K_n$ for each $n$, and the rotation by $a$ is the inverse limit of the rotations by $a_n$. 
    Hence $(K, m_K, R)$ is a $1$-step pro-nilsystem.  
\end{example}

For topological pro-nilsystems, minimality, transitivity, and unique ergodicity are all equivalent notions. 
Indeed, as nilsystems are distal by \cref{lem:nilsystems_are_distal}, and distality passes to the inverse limit by \cref{lem:dynamical_prop_inverse_limit}, pro-nilsystems are distal. 
Hence minimality and transitivity are equivalent; see \citep[Chapter 5, Corollary 7]{auslander1988minimal}. 
Moreover, the equivalence of minimality and unique ergodicity holds because this is true for nilsystems by \cref{cor:equivalence_ergodicity_minimality_etc_nilsystems}, and these notions both pass to factors and inverse limits. 

The next lemma explains the relation between measure-theoretic and topological pro-nilsystems. 
The argument is based on \citep[Chapter 13, Section 3]{host2018nilpotent}. 

\begin{lemma}
    \label[lemma]{lem:topological_model_inverse_limit}
    Let $s \geq 1$. 
    \begin{enumerate}[(i)]
        \item Every ergodic measure-theoretic $s$-step pro-nilsystem is isomorphic to a uniquely ergodic topological $s$-step pro-nilsystem, equipped with its unique invariant probability measure. 
        \item Conversely, if $(X, T)$ is a uniquely ergodic topological $s$-step pro-nilsystem, and $\mu$ is its unique invariant probability measure, then $(X, \mu, T)$ is an ergodic measure-theoretic $s$-step pro-nilsystem. 
    \end{enumerate} 
\end{lemma}

\begin{proof}
    For (i), let $(X, \mu, T)$ be an ergodic measure-theoretic $s$-step pro-nilsystem. 
    Choose an inverse limit 
    \[
    (X, \mu, T) \cong \varprojlim(X_n, \mu_n, T_n)
    \]
    where each $(X_n, \mu_n, T_n)$ is an $s$-step nilsystem. 
    As a factor of the ergodic system $X$, each $X_n$ is ergodic. 
    By \cref{cor:factors_between_nilsystem_are_topological}, after modifying on null sets, we may assume that each connecting factor map $\alpha_n \colon X_{n+1} \to X_n$ is a topological factor map.  
    Thus we can take the topological inverse limit 
    \[
    (X^\tp, T^\tp) := \varprojlim(X_n, T_n). 
    \]
    Each nilsystem $X_n$ is ergodic, hence uniquely ergodic; see \cref{cor:equivalence_ergodicity_minimality_etc_nilsystems}.  
    Therefore, the inverse limit $(X^\tp, T^\tp)$ is uniquely ergodic. 
    Let $\mu^\tp$ be its unique invariant probability measure. 
    By \cref{lem:dynamical_prop_inverse_limit}, $(X^\tp, \mu^\tp, T^\tp)$ is the measure-theoretic inverse limit of $(X_n, \mu_n, T_n)$. 
    Thus, by uniqueness of the measure-theoretic inverse limit, 
    \[
    (X^\tp, \mu^\tp, T^\tp) \cong (X, \mu, T), 
    \]
    which proves (i). 

    For (ii), suppose that $(X, T)$ is a uniquely ergodic $s$-step pro-nilsystem, and let $\mu$ be its unique invariant probability measure. 
    Choose an inverse limit 
    \[
    (X, T) \cong\varprojlim (X_n, T_n) 
    \] 
    where each $(X_n, T_n)$ is an $s$-step nilsystem.
    As a factor of $X$, each nilsystem $X_n$ is uniquely ergodic, and its unique invariant probability measure is the Haar measure $\mu_n$. 
    By \cref{lem:dynamical_prop_inverse_limit}, 
    \[
    (X, \mu, T) \cong \varprojlim (X_n, \mu_n, T_n),  
    \]
    hence $(X, \mu, T)$ is an ergodic measure-theoretic $s$-step pro-nilsystem. 
\end{proof}

In view of this lemma, every ergodic $s$-step pro-nilsystem admits a uniquely ergodic topological $s$-step pro-nilsystem as a topological model. 
A priori, this model could depend on the specific choice of inverse limit.
The following proposition shows that it is in fact canonical. 

\begin{proposition} 
    \label[proposition]{prop:factors_between_pronilsystems_are_topological}
    Let $(X, \mu, T)$ and $(Y, \nu, S)$ be ergodic $s$-step pro-nilsystems, equipped with uniquely ergodic topological $s$-step pro-nilsystem models as provided by \cref{lem:topological_model_inverse_limit}.  
    If 
    \[
    \pi \colon (X, \mu, T) \to (Y, \nu, S) 
    \] 
    is a measure-theoretic factor map, then $\pi$ agrees almost everywhere with a topological factor map 
    \[
    \pi^\tp \colon (X, T) \to (Y, S).
    \]
    Moreover, if $\pi$ is a measure-theoretic isomorphism, then $\pi^\tp$ is a topological isomorphism.
\end{proposition}

In particular, it follows that if two inverse limits give rise to the same ergodic measure-theoretic pro-nilsystem, then the associated topological models are isomorphic. 
This result was proved by \citet{host2010nilsequences}, using the theory of dynamical dual functions. 
For completeness, we give a direct proof instead which avoids the use of any Host--Kra theory. 

\begin{proof}
    Since $(X, \mu, T)$ and $(Y, \nu, S)$ are equipped with the topological models from \cref{lem:topological_model_inverse_limit}, we may write 
    \[
    (X, T) = \varprojlim (X_n, T_n) 
    \quad \text{ and } \quad 
    (Y, S) = \varprojlim (Y_n, S_n),
    \]
    as topological inverse limits of uniquely ergodic $s$-step nilsystems, and assume that $\mu$ and $\nu$ are the unique invariant measures on $X$ and $Y$ respectively. 
    Let $\alpha_n \colon X_{n+1} \to X_n$ and $\beta_n \colon Y_{n+1} \to Y_n$ be the connecting factor maps, and let $p_n \colon X \to X_n$ and $q_n \colon Y \to Y_n$ be the factor maps associated to the inverse limits. 
    Furthermore, let $\mu_n$ and $\nu_n$ denote the Haar measures on $X_n$ and $Y_n$. 

    Let  
    \[
    \lambda := (\Id_X, \pi)_* \mu 
    \]
    be the graph joining of $\pi$, which is an ergodic joining of $(X, \mu, T)$ and $(Y, \nu, S)$. 
    For each $n \in \N$, define 
    \[
    \lambda_n := (p_n, q_n)_* \lambda.  
    \]
    Then $\lambda_n$ is an ergodic joining of the ergodic nilsystems $(X_n, \mu_n, T_n)$ and $(Y_n, \nu_n, S_n)$. 
    By \cref{prop:ergodic_joinings_nilsystems_are_nilsystems}, there is a $T_n \times S_n$-invariant subnilmanifold 
    \[
    Z_n \subseteq X_n \times Y_n 
    \]
    such that $\lambda_n$ is the Haar measure of $Z_n$. 
    In particular, $(Z_n, \lambda_n, T_n \times S_n)$ is an ergodic $s$-step nilsystem, so $(Z_n, T_n \times S_n)$ is uniquely ergodic.  

    We show that the systems $Z_n$ form an inverse system. 
    Consider the maps
    \[
    \alpha_n \times \beta_n \colon X_{n+1} \times Y_{n+1} \to X_n \times Y_n.
    \]
    Since $(\alpha_n \times \beta_n)_* \lambda_{n+1} = \lambda_n$, and since $\alpha_n \times \beta_n$ is continuous, it maps $\supp(\lambda_{n+1}) = Z_{n+1}$ onto $\supp(\lambda_n) = Z_n$.  
    Hence, after restricting, we obtain topological factor maps  
    \[
    \alpha_n \times \beta_n \colon (Z_{n+1}, T_{n+1} \times S_{n+1}) \to (Z_n, T_n \times S_n). 
    \]

    Define 
    \[
    Z := \{(x,y) \in X \times Y \colon (p_n(x), q_n(y)) \in Z_n \text{ for all } n \in \N\} \subseteq X \times Y.  
    \]
    Then $Z$ is a $T \times S$-invariant closed subset of $X \times Y$. 
    The factor maps $p_n \times q_n \colon Z \to Z_n$ are compatible with the connecting factor maps $\alpha_n \times \beta_n$ and separate points on $Z$, so that   
    \[
    (Z, T \times S) \cong \varprojlim (Z_n, T_n \times S_n).  
    \]
    Therefore $(Z, T \times S)$ is a uniquely ergodic topological $s$-step pro-nilsystem. 
    In particular, it is distal. 
    Moreover, $\lambda((p_n \times q_n)^{-1}(Z_n)) = \lambda_n(Z_n) = 1$ for each $n$, so $\lambda(Z) = 1$. 
    Hence, when viewed as a probability measure on $Z$, $\lambda$ is the unique invariant measure on $(Z, T \times S)$.  

    Let 
    \[
    r_1 \colon Z \to X, \quad r_2 \colon Z \to Y 
    \]
    be the coordinate projections. 
    These are topological factor maps. 
    Moreover, since $\lambda$ is the graph joining of $\pi$, the map $r_1$ is a measure-theoretic isomorphism from $(Z, \lambda, T \times S)$ to $(X, \mu, T)$. 
    Because $(Z, T \times S)$ is distal and uniquely ergodic, \cref{lem:isomorphism_upgrade_crit} implies that $r_1$ is a topological isomorphism. 
    Define 
    \[
    \pi^\tp := r_2 \circ r_1^{-1} \colon X \to Y.  
    \]
    Then $\pi^\tp$ is a topological factor map. 
    Moreover, for $\mu$-almost every $x \in X$, $r_1^{-1}(x) = (x, \pi(x))$, and thus 
    \[
    \pi^\tp(x) = r_2(x, \pi(x)) = \pi(x) 
    \]
    for $\mu$-almost every $x \in X$. 

    Finally, suppose that $\pi$ is a measure-theoretic isomorphism. 
    Then the second coordinate projection $r_2 \colon Z \to Y$ is also a measure-theoretic isomorphism. 
    Applying \cref{lem:isomorphism_upgrade_crit} again, we conclude that $r_2$ is a topological isomorphism. 
    Hence $\pi^\tp = r_2 \circ r_1^{-1}$ is a topological isomorphism. 
\end{proof}

\chapter{Local rigidity of self-joinings of nilsystems}
\label{chap:rigidity_result}

The purpose of this chapter is to prove a local rigidity property for self-joinings of nilsystems. 
Roughly speaking, we show that an ergodic self-joining of a nilsystem which is sufficiently close to the diagonal joining must be the graph joining of an automorphism. 
This result will be the key input in the proof of the factor-closure of ergodic pro-nilsystems in \cref{chap:factor_closure}.

Let $(X, \mu, T)$ be an ergodic nilsystem. 
Since $X$ is a compact metric space, the Riesz representation theorem identifies the space $\pr(X \times X)$ of Borel probability measures on $X \times X$ with a subset of the dual space of the Banach space $C(X \times X)$ of continuous functions on $X \times X$.  
We equip $\pr(X \times X)$ with the weak-* topology under this identification. 
This topology on $\pr(X \times X)$ is compact and metrizable. 
Recall from \cref{sec:joinings} that $J_e(X, X)$ denotes the set of ergodic self-joinings of $(X, \mu, T)$. 
We equip $J_e(X, X) \subseteq \pr(X \times X)$ with the subspace topology, so that $J_e(X, X)$ is metrizable as well. 

The following proposition is the main result of this chapter. 

\begin{proposition}[Local rigidity of self-joinings] 
    \label[proposition]{prop:rigidity_self_joinings}
    Let $(X, \mu, T)$ be an ergodic nilsystem, and let $\mu_\Delta := (\Id_X, \Id_X)_* \mu$ be the diagonal joining. 
    Then there exists a neighborhood $U$ of $\mu_\Delta$ in $J_e(X, X)$ such that every $\lambda \in U$ is the graph joining of an automorphism of $(X, \mu, T)$. 
\end{proposition}

Equivalently, if $d$ is a compatible metric on $\pr(X \times X)$, then there exists $\delta > 0$ such that every $\lambda \in J_e(X, X)$ with $d(\lambda, \mu_\Delta) < \delta$ is the graph joining of an automorphism. 

The proof goes as follows. 
Suppose that $\lambda \in J_e(X, X)$ is an ergodic self-joining close to the diagonal joining.  
By \cref{prop:ergodic_joinings_nilsystems_are_nilsystems}, $\lambda$ is the Haar measure of an invariant subnilmanifold $Y$ of the product nilmanifold $X \times X$. 
This turns the problem into a geometric one about the structure of subnilmanifolds. 
The main relevant result here will be that nilmanifolds cannot contain arbitrarily small non-trivial subnilmanifolds, together with an averaged form of this statement. 
Since $\lambda$ is close to the diagonal joining, its support $Y$ must be close to the diagonal in an averaged sense. 
Consequently, the coordinate fibers of $Y$ are small subnilmanifolds of $X \times X$, forcing them to be trivial and showing that $Y$ is a graph joining. 

\cref{sec:no_small_subnilmanifolds} is devoted to the results about small subnilmanifolds. 
In \cref{sec:proof_local_rigidity}, we use these results to prove \cref{prop:rigidity_self_joinings}. 

\section{No small subnilmanifolds}
\label{sec:no_small_subnilmanifolds}

A well-known fact about Lie groups is that they satisfy the no small subgroups property. 
This means that every Lie group admits an identity neighborhood containing no non-trivial subgroup. 
In fact, this property precisely characterizes Lie groups among locally compact groups, which is a result that played a crucial role in the resolution of Hilbert's fifth problem; see \citep{tao2014hilbert}.

In the setting of nilmanifolds, subnilmanifolds are the objects parallel to subgroups.
The next lemma can therefore be viewed as an analogue of the no small subgroups property for nilmanifolds. 
Given a nilmanifold and a compatible metric, it produces a positive lower bound on the diameter of all non-trivial subnilmanifolds.
Throughout this chapter, a subnilmanifold is called trivial if it consists of a single point, and non-trivial otherwise.
Recall our standing assumption from \cref{sec:nilmanifolds_standing_assumption} that $G/G^o$ is finitely generated whenever $X = G/\Gamma$ is a nilmanifold.

\begin{lemma}[No small subnilmanifolds] 
    \label[lemma]{lem:no_small_subnilmanifolds}
    Let $X = G/\Gamma$ be a nilmanifold, and let $d_X$ be a compatible metric on $X$. 
    Then there exists $\epsilon > 0$ such that any subnilmanifold $Y \subseteq X$ with 
    \[
    \diam_X(Y) < \epsilon
    \] 
    is trivial.
\end{lemma}

For the proof of \cref{prop:rigidity_self_joinings}, we need the following stronger averaged form of the no small subnilmanifolds lemma. 
We prove this averaged form directly, and deduce \cref{lem:no_small_subnilmanifolds} as a consequence. 

\begin{lemma}[No small subnilmanifolds, averaged form] 
    \label[lemma]{lem:no_small_subnilmanifolds_averaged}
    Let $X = G/\Gamma$ be an $s$-step nilmanifold, and let $d_X$ be a compatible metric on $X$.
    Then there exists $\eta > 0$ such that for every non-trivial subnilmanifold $Y \subseteq X$, 
    \[
    \int_{Y \times Y} d_X(x,y) \dd\nu(x) \dd\nu(y) \geq \eta,  
    \]
    where $\nu$ denotes the Haar measure on $Y$.
\end{lemma}

Let us first see that \cref{lem:no_small_subnilmanifolds_averaged} indeed implies \cref{lem:no_small_subnilmanifolds}.
For any subnilmanifold $Y \subseteq X$ with Haar measure $\nu$, we can bound    
\[
\int_{Y \times Y} d_X(x,y) \dd\nu(x) \dd\nu(y) \leq \diam_X(Y).
\]
Hence, if $Y$ is non-trivial, then $\diam_X(Y) \geq \eta$ by \cref{lem:no_small_subnilmanifolds_averaged}. 
Equivalently, every subnilmanifold of diameter strictly less than $\eta$ must be trivial. 

The proof of \cref{lem:no_small_subnilmanifolds_averaged} relies on the following consequence of the no small subgroups property of Lie groups, applied to a compact abelian Lie group acting freely on a compact metric space. 

\begin{lemma} 
    \label[lemma]{lem:technical_lemma_cpt_abelian_Lie_action}
    Let $K$ be a compact abelian Lie group acting freely and continuously on a compact metric space $X$, with the action denoted by $K \times X \to X, (u, x) \mapsto u \cdot x$. 
    Then there exists $\eta > 0$ such that for every $x \in X$, and every non-trivial closed subgroup $L \leq K$, 
    \[
    \int_{L \times L} d_X(u \cdot x, v \cdot x) \dd m_L(u) \dd m_L(v) \geq \eta,  
    \]
    where $m_L$ denotes the Haar probability measure of $L$.
\end{lemma}

\begin{proof}
    As a Lie group, $K$ has the no small subgroups property, so there exists an identity neighborhood $U \subseteq K$ containing no non-trivial subgroup. 
    Fix an identity neighborhood $V \subseteq K$ such that $V V^{-1} \subseteq U$. 
    Since the action is free, the map 
    \[
    K \times X \to \R, \quad (u, x) \mapsto d_X(u \cdot x, x) 
    \]
    is strictly positive away from $u = e_K$.
    Because this map is continuous and the set $(K \setminus V) \times X$ is compact, there thus exists $\eta > 0$ such that 
    \[
    d_X (u \cdot x, x) \geq 2 \eta 
    \]
    for all $u \in K \setminus V$ and $x \in X$. 

    Let $L \leq K$ be a non-trivial closed subgroup. 
    We claim that 
    \[
    m_L(L \cap V) \leq 1/2. 
    \]
    Indeed, suppose that $m_L(L \cap V) > 1/2$. 
    Then for every $u \in L$, the sets $L \cap V$ and $u (L \cap V)$ both have measure greater than $1/2$, so they must intersect. 
    Therefore $u \in (L \cap V)(L \cap V)^{-1}$, and so   
    \[
    (L \cap V) (L \cap V)^{-1} = L. 
    \]
    It follows that
    \[
    L \subseteq V V^{-1} \subseteq U, 
    \]
    contradicting the choice of $U$.
    This proves the claim, so $m_L(L \setminus V) \geq 1/2$. 

    Now fix $x \in X$ and $v \in L$. 
    By making the substitution $u \mapsto uv$ using the invariance of $m_L$, we obtain 
    \[
    \int_L d_X(u \cdot x, v \cdot x) \dd m_L(u) 
    = \int_L d_X(u \cdot (v \cdot x), v \cdot x) \dd m_L(u)
    \geq 2 \eta \ m_L(L \setminus V) \geq \eta. 
    \]
    Integrating over $v \in L$, we get the claimed inequality. 
\end{proof}

We are now ready to prove the averaged no small subnilmanifolds lemma. 

\begin{proof}[Proof of \cref{lem:no_small_subnilmanifolds_averaged}]
    Let $p \colon G \to X = G/\Gamma$ be the quotient map.
    We argue by induction on the step $s$ of the nilmanifold $X$. 
    
    First suppose that $s = 1$. 
    Then $X = G/\Gamma$ is a compact abelian Lie group. 
    Consider the action of $X$ on itself by translations, $X \times X \to X, (u, x) \mapsto u x$. 
    Applying \cref{lem:technical_lemma_cpt_abelian_Lie_action} with this action, we obtain $\eta > 0$ such that for every $a \in X$, and every non-trivial closed subgroup $L \leq X$, 
    \[
    \int_{L \times L} d_X(u a, v a) \dd m_L(u) \dd m_L(v) \geq \eta. 
    \]
    Let $Y \subseteq X$ be a non-trivial subnilmanifold, and write $Y = H \cdot a = p(H) a$ for some closed subgroup $H \leq G$, and a point $a \in X$. 
    Since $Y$ is closed,  
    \[
    L := p(H) = Y a^{-1}  
    \]
    is a closed subgroup of $X$, and $L$ is non-trivial because $Y$ is non-trivial. 
    Moreover, the pushforward of $m_L$ under the map $u \mapsto ua$ is invariant under all translations by elements of $L$, and hence equals the Haar measure $\nu$ of $Y$. 
    Therefore, 
    \[
    \int_{Y \times Y} d_X (x,y) \dd \nu(x) \dd \nu(y) 
    = \int_{L \times L} d_X(u a, v a) \dd m_L(u) \dd m_L(v) \geq \eta. 
    \]
    This proves the $1$-step case.

    Now suppose that $s \geq 2$ and that the statement has been proved for $(s-1)$-step nilmanifolds. 
    Let 
    \[
    \pi := \pi_{s-1} \colon X \to Z_{s-1} =: Z
    \]
    be the last projection in the tower of $X$, and let $K := K_s$ be the last structure group; see \cref{sec:nilmanifold_tower}.
    Then $Z$ is an $(s-1)$-step nilmanifold, $K$ is a compact abelian Lie group acting freely and continuously on $X$, the fibers of $\pi$ are exactly the $K$-orbits, and the $K$-action commutes with all translations by elements of $G$ on $X$. 

    Fix a compatible metric $d_Z$ on $Z$.
    Without loss of generality, assume that $d_Z \leq 1$. 
    By the induction hypothesis, there exists $\eta_1 > 0$ such that for every non-trivial subnilmanifold $W \subseteq Z$, 
    \begin{equation}
        \label{eq:def_eta_1} 
        \int_{W \times W} d_Z(z, w) \dd\omega(z) \dd\omega(w) \geq \eta_1
    \end{equation}
    where $\omega$ is the Haar measure on $W$. 
    By compactness, $\pi$ is uniformly continuous, so there exists $\delta > 0$ such that 
    \begin{equation}
        \label{eq:def_delta} 
        d_X(x,y) < \delta \quad \implies \quad d_Z(\pi(x), \pi(y)) < \eta_1/2 \qquad \text{for all } x,y \in X. 
    \end{equation}

    Applying \cref{lem:technical_lemma_cpt_abelian_Lie_action} to the action of $K$ on $X$, choose $\eta_2 > 0$ such that for every non-trivial closed subgroup $L \leq K$, and every $a \in X$,  
    \begin{equation}
        \label{eq:def_eta_2} 
        \int_{L \times L} d_X(u \cdot a, v \cdot a) \dd m_L(u) \dd m_L(v) \geq \eta_2,
    \end{equation}
    where $m_L$ is the Haar measure of $L$. 
    Let 
    \[
    \eta := \min\{\delta \eta_1 / 2, \eta_2\}. 
    \]

    Let $Y \subseteq X$ be a non-trivial subnilmanifold. 
    Then  
    \[
    Y = H \cdot a  
    \]
    for some closed subgroup $H \leq G$ and a point $a \in X$.
    By \cref{lem:algebraic_image_subnilmanifold}, $\pi(Y)$ is a subnilmanifold of $Z$, and $\pi_* \nu$ is its Haar measure.  
    We distinguish two cases: either $\pi(Y)$ is non-trivial, or $\pi(Y)$ is a singleton. 

    Suppose that $\pi(Y)$ is non-trivial. 
    Since $\pi_* \nu$ is the Haar measure on $\pi(Y)$, our choice of $\eta_1$ in \eqref{eq:def_eta_1} implies that 
    \begin{equation}
        \label{eq:eta_1_bound} 
        \begin{aligned}
        \eta_1 &\leq \int_{\pi(Y) \times \pi(Y)} d_Z(z, w) \dd (\pi_*\nu)(z) \dd (\pi_*\nu)(w) \\
        &= \int_{Y \times Y} d_Z(\pi(x), \pi(y)) \dd\nu(x)\dd\nu(y). 
        \end{aligned}
    \end{equation}
    We claim that for every $x, y \in Y$, 
    \[
    d_Z(\pi(x), \pi(y)) \leq \eta_1/2 + \frac{1}{\delta} d_X(x,y). 
    \]
    Indeed, if $d_X(x,y) < \delta$, this follows from \eqref{eq:def_delta}, and if $d_X(x,y) \geq \delta$, then 
    \[
    d_Z(\pi(x), \pi(y)) \leq 1 \leq \frac{1}{\delta} d_X(x,y).
    \]
    Plugging this inequality into \eqref{eq:eta_1_bound} gives
    \[
    \eta_1 \leq \eta_1/2 + \frac{1}{\delta} \int_{Y \times Y} d_X(x,y) \dd\nu(x) \dd\nu(y). 
    \]
    After rearranging this inequality, we see that  
    \[
    \int_{Y \times Y} d_X(x,y) \dd\nu(x) \dd\nu(y) \geq \delta \eta_1 / 2 \geq \eta. 
    \]

    It remains to consider the case where $\pi(Y)$ is a singleton. 
    Then $Y$ is contained in a single fiber of $\pi$, equivalently in a single $K$-orbit.  
    Since $a \in Y$, we get that $Y \subseteq K \cdot a$. 
    Define 
    \[
    L := \{ u \in K : u \cdot a \in Y \},  
    \]
    so that $Y = L \cdot a$. 
    The set $L$ is closed in $K$, because $Y$ is closed and the $K$-action is continuous. 
    We claim that $L$ is a subgroup of $K$. 
    Let $u, v \in L$. 
    Then there exist $h_1, h_2 \in H$ such that 
    \[
    h_1 \cdot a = u \cdot a, \quad h_2 \cdot a = v \cdot a.  
    \]
    Since the $K$-action commutes with all translations by elements of $G$, we can rewrite the second identity to get $v^{-1} \cdot a = h_2^{-1} \cdot a$, and hence 
    \[
    u v^{-1} \cdot a 
    = u \cdot (h_2^{-1} \cdot a)  
    = h_2^{-1} \cdot (u \cdot a) 
    = h_2^{-1} h_1 \cdot a \in Y.
    \]
    Therefore, $u v^{-1} \in L$ so that $L$ is a subgroup of $K$. 
    It is non-trivial, because $Y$ is non-trivial. 

    Next, we claim that $\nu$ is the image of the Haar probability measure $m_L$ on $L$ under the map 
    \[
    \phi \colon L \to Y, \quad u \mapsto u \cdot a.  
    \]
    Let $h \in H$. 
    Since $h \cdot a \in Y = L \cdot a$, there exists $v \in L$ such that $h \cdot a = v \cdot a$. 
    Then for each $u \in L$, 
    \[
    h \cdot \phi(u) = h \cdot (u \cdot a) = u \cdot (h \cdot a) = uv \cdot a = \phi(v u).
    \] 
    By the invariance of $m_L$, it follows that $\phi_* m_L$ is invariant under all translations by elements of $H$. 
    Hence $\phi_* m_L = \nu$ by uniqueness of the Haar measure on $Y$. 
    Therefore, the choice of $\eta_2$ in \eqref{eq:def_eta_2} implies that
    \[
    \int_{Y \times Y} d_X(x,y) \dd\nu(x) \dd\nu(y)
    = \int_{L \times L} d_X(u \cdot a, v \cdot a) \dd m_L(u) \dd m_L(v) 
    \geq \eta_2 \geq \eta.
    \]
    This completes the induction step. 
\end{proof}

\section{Local rigidity of self-joinings}
\label{sec:proof_local_rigidity}

In this section, we further study \cref{prop:ergodic_joinings_nilsystems_are_nilsystems}.  
We first show that for rotations on compact abelian groups, a stronger global rigidity statement holds, and we consider an example showing that this global rigidity fails for nilsystems of higher step.  
Finally, we prove \cref{prop:rigidity_self_joinings} and deduce a first consequence.  

\subsection{The case of group rotations}

Recall that \cref{prop:rigidity_self_joinings} says that any ergodic self-joining of a nilsystem sufficiently close to the diagonal joining is the graph joining of an automorphism. 
For rotations on compact abelian groups, a stronger global version of this statement is true: 
every ergodic self-joining is the graph joining of an automorphism. 
Throughout this subsection, abelian groups are written additively. 

\begin{proposition}[Global rigidity for group rotations]
    Let $(K, m_K, R)$ be an ergodic rotation on a metrizable compact abelian group, where 
    \[
    R \colon K \to K, \quad u \mapsto a + u
    \]
    for some $a \in K$; see \cref{eg:lie_group_rotations}.  
    Then every ergodic self-joining of $(K, m_K, R)$ is the graph joining of an automorphism.  
\end{proposition}

\begin{proof}
    Let $\lambda$ be an ergodic self-joining of $(K, m_K, R)$.  
    Then $\lambda$ is an ergodic invariant measure of the product rotation $R \times R$ on $K \times K$. 
    By standard results, see for example \citep[Chapter 4, Proposition 5]{host2018nilpotent}, $\lambda$ is the image under some rotation of the Haar measure of the closed subgroup 
    \[
    H := \overline{\{(n a, n a) : n \in \Z\}} \leq K \times K.  
    \]
    Since $R$ is ergodic, the set $\{n a: n \in \Z\}$ is dense in $K$; see \citep[Theorem 4.14]{einsiedler2010ergodic}. 
    Hence 
    \[
    H = \{(u, u) : u \in K\}.  
    \]
    The diagonal embedding of $K$ into $K \times K$ is a topological group isomorphism onto $H$. 
    Therefore the Haar measure on $H$ is $m_H = (\Id_K, \Id_K)_* m_K$. 
    Pick $(u,v) \in K \times K$ such that $\lambda$ is the image of $m_H$ under the rotation by $(u,v)$, and let $w := v - u$. 
    Using the invariance of $m_K$, we obtain 
    \[
    \lambda = (R_u, R_v)_* m_K = (\Id_K, R_w)_* m_K,
    \]
    where $R_u$, $R_v$, and $R_w$ denote the rotations by $u$, $v$, and $w$ on $K$ respectively. 
    Hence $\lambda$ is the graph joining of the rotation $R_w$. 
    Since $K$ is abelian, $R_w$ is an automorphism of $(K, m_K, R)$. 
\end{proof}

The global rigidity statement for group rotations does not extend to nilsystems. 
A counterexample already arises among $2$-step nilsystems, as the following example demonstrates. 

\begin{example}
    As in \cref{eg:heisenberg_nilsystem}, let $G = \R^3$ with multiplication law
    \[
    (x,y,z) \cdot (x',y',z') = (x + x', y + y', z + z' + x y'),   
    \] 
    let $\Gamma = \Z^3 \leq G$, and let $X = G/\Gamma$. 
    Choose $\tau = (a, b, 0) \in G$ such that $1$, $a$, and $b$ are linearly independent over $\Q$. 
    If $\mu$ denotes the Haar measure on $X$ and $T$ is the translation by $\tau$, then $(X, \mu, T)$ is the ergodic Heisenberg nilsystem. 
    Let 
    \[
    C := \{(0, 0, t) : t \in \R\} \leq G 
    \]
    be the center of $G$. 
    Choose $r \in \R$ such that $1, a, b, r a$ are linearly independent over $\Q$, and define $h := (0, r, 0) \in G$. 
    Consider the subset 
    \[
    Y := \{(g\Gamma, h g c \Gamma) : g \in G, c \in C\} \subseteq X \times X.  
    \]
    Intuitively, $Y$ is obtained by taking the graph of the translation by $h$ on $X$, and thickening it in the central direction $C$. 
    We claim that $Y$ is a subnilmanifold of $X \times X$. 
    Let 
    \[
    H := \{(g, h g h^{-1} c) : g \in G, c \in C\} \leq G \times G.  
    \]
    Since $C$ is closed and central, $H$ is a closed subgroup of $G \times G$, and $Y = H \cdot (e_X, h\Gamma)$. 
    Moreover, the map  
    \[
    \phi \colon X \times \T \to Y, \quad (g\Gamma, t \bmod 1) \mapsto (g\Gamma, hg (0, 0, t) \Gamma)  
    \]
    is a well-defined homeomorphism, so $Y$ is compact and hence closed. 
    Therefore $Y$ is a subnilmanifold of $X \times X$. 
    Let $\lambda$ denote its Haar measure. 

    Note that 
    \[
    \tau h = (a, b+ r, r a) = h \tau (0, 0, r a)  
    \]
    from which it follows that 
    \[
    (T \times T) \circ \phi = \phi \circ (T \times R_{r a}), 
    \]
    where $R_{r a}$ is the rotation by $r a$ on $\T = \R/\Z$. 
    By the choice of $r$, the translation $T \times R_{r a}$ is ergodic on the nilmanifold $X \times \T$. 
    Hence $(Y, \lambda, T \times T)$ is ergodic. 

    Since $\lambda$ is $T \times T$-invariant, its coordinate projections are $T$-invariant probability measures on $X$, so they agree with $\mu$ by unique ergodicity of ergodic nilsystems. 
    Therefore, $\lambda$ is an ergodic self-joining of $(X, \mu, T)$. 

    It remains to verify that $\lambda$ is not the graph joining of an automorphism. 
    Indeed, every measure-theoretic automorphism of $(X, \mu, T)$ agrees almost everywhere with a topological automorphism; see \cref{cor:factors_between_nilsystem_are_topological}.
    Hence, if $\lambda$ were the graph joining of an automorphism, then $Y = \supp(\lambda)$ would be the graph of a continuous map from $X$ to $X$.
    In particular, $Y$ would be homeomorphic to $X$. 
    Since $Y$ is homeomorphic to $X \times \T$, this is impossible. 
\end{example}

\subsection{Proof of local rigidity of self-joinings}

We now turn to the proof of \cref{prop:rigidity_self_joinings}, for which we need the following lemma. 
Background on the disintegration of measures can be found in \cite[Section 8.6]{eisner2025journey}.

\begin{lemma}
    \label[lemma]{lem:fibers_joining_are_subnilmanifolds}
    Let $(X = G/\Gamma, \mu, T)$ be an ergodic nilsystem, let $\lambda \in J_e(X, X)$ be an ergodic self-joining, and let $Y := \supp(\lambda)$. 
    For each $x \in X$, define the fiber
    \[
    Y_x := Y \cap (\{x\} \times X).  
    \]
    Then the sets $Y_x$ are non-empty, pairwise homeomorphic subnilmanifolds of $X \times X$. 
    Moreover, if $\lambda_x$ denotes the Haar measure on $Y_x$ for each $x$, then  
    \[
    \lambda = \int_X \lambda_x \dd\mu(x) 
    \]
    is the disintegration of $\lambda$ with respect to the first coordinate. 
\end{lemma}

\begin{proof} 
    Let $\pi_1 \colon X \times X \to X$ be the first coordinate projection. 
    By \cref{prop:ergodic_joinings_nilsystems_are_nilsystems}, there exist a closed subgroup $H \leq G \times G$ and a point $z_0=(x_0,y_0) \in X \times X$ such that 
    \[
    Y = H \cdot z_0
    \] 
    is a subnilmanifold of $X \times X$ and $\lambda$ is its Haar measure. 
    Since $\lambda$ is a joining, $(\pi_1)_*\lambda = \mu$. 
    In particular, $\pi_1(Y) \subseteq X$ has full measure, and since it is closed, $\pi_1(Y) = X$.   
    Hence the restriction of $\pi_1$ to $Y$ is surjective, and every fiber $Y_x$ is non-empty. 

    Write $x_0 = g_0\Gamma$, and define  
    \[
    H_0 := H \cap ((g_0\Gamma g_0^{-1}) \times G).
    \] 
    Then $Y_{x_0}=H_0 \cdot z_0$, so $Y_{x_0}$ is a subnilmanifold of $X \times X$. 
    Given $x \in X$, choose $h = (h_1,h_2) \in H$ such that $h_1 \cdot x_0=x$. 
    This is always possible since $\pi_1(Y) = X$. 
    Then 
    \[
    Y_x = h \cdot Y_{x_0},
    \] 
    so the fibers $Y_x$ are pairwise homeomorphic subnilmanifolds.  

    It remains to prove the claimed disintegration identity.  
    For $x \in X$, let $\lambda_x$ be the Haar probability measure on $Y_x$. 
    If $h = (h_1,h_2) \in H$, then as above $h$ maps $Y_x$ onto $Y_{h_1 \cdot x}$. 
    It follows that 
    \[
    h_* \lambda_x = \lambda_{h_1 \cdot x}, 
    \]
    since the pushforward is a translation-invariant probability measure on $Y_{h_1 \cdot x}$.

    We first show that the map $x \mapsto \lambda_x$ is measurable. 
    Indeed, consider the continuous surjection 
    \[
    H \to X, \quad h \mapsto \pi_1(h \cdot z_0). 
    \] 
    By standard results, this map admits a measurable section $\sigma \colon X \to H$; see for example \citep[Lemma 7]{baggett1980functional}. 
    Then $Y_x = \sigma(x) \cdot Y_{x_0}$, so by the above observation,  
    \[
    \lambda_x = \sigma(x)_* \lambda_{x_0}. 
    \]
    Therefore $x \mapsto \int_{Y_x} f \dd\lambda_x = \int_{Y_{x_0}} f(\sigma(x) \cdot y) \dd\lambda_{x_0}(y)$ is measurable for every $f \in C(Y)$, and so $x \mapsto \lambda_x$ is measurable. 

    Define the probability measure $\widetilde{\lambda}$ on $Y$ by setting  
    \[ 
    \int_Y f \dd\widetilde{\lambda} := \int_X \int_{Y_x} f \dd\lambda_x \dd\mu(x) \quad \text{for } f \in C(Y). 
    \] 
    We claim that $\widetilde{\lambda}$ is invariant under all translations by elements of $H$. 
    Let $h = (h_1,h_2) \in H$. Using the fact that $h_* \lambda_x = \lambda_{h_1 \cdot x}$ and the invariance of $\mu$ under the translation $x \mapsto h_1 \cdot x$, we compute 
    \begin{align*} 
        \int_Y f(h\cdot z) \dd\widetilde{\lambda}(z) &= \int_X \int_{Y_x} f(h\cdot z) \dd\lambda_x(z) \dd\mu(x) \\ 
        &= \int_X \int_{Y_{h_1\cdot x}} f(z) \dd\lambda_{h_1\cdot x}(z) \dd\mu(x) \\ 
        &= \int_X \int_{Y_x} f(z) \dd\lambda_x(z) \dd\mu(x)  
        = \int_Y f \dd\widetilde{\lambda} 
    \end{align*}
    for every $f \in C(Y)$. 
    Hence $\widetilde{\lambda}$ is an $H$-invariant probability measure on $Y = H \cdot z_0$, so $\widetilde{\lambda} = \lambda$ by uniqueness of the Haar measure. 
    Therefore,
    \[
    \lambda = \int_X \lambda_x \dd\mu(x), 
    \] 
    concluding the proof.  
\end{proof}

\begin{proof}[Proof of \cref{prop:rigidity_self_joinings}.]
    Let $d_X$ be a compatible metric on $X$, and define the product metric 
    \begin{equation}
        \label{eq:def_product_metric} 
        d_{X \times X}((x,y), (x',y')) := d_X(x,x') + d_X(y,y'). 
    \end{equation}
    Let $\eta > 0$ be the constant obtained from applying \cref{lem:no_small_subnilmanifolds_averaged} to the nilmanifold $X \times X$ with metric $d_{X \times X}$. 

    Define 
    \[
    U := \left\{ \lambda \in J_e(X, X) : \int_{X \times X} d_X(x,y) \dd\lambda(x,y) < \eta/2 \right\}.
    \]
    Since $d_X$ is continuous on $X \times X$, $U$ is open in $J_e(X, X)$ by definition of the weak-* topology. 
    Moreover, 
    \[
    \int_{X \times X} d_X (x,y) \dd \mu_\Delta(x,y) 
    = \int_X d_X(x,x) \dd\mu(x) = 0, 
    \]
    so $U$ is an open neighborhood of $\mu_\Delta$ in $J_e(X, X)$. 
    It remains to prove that every $\lambda \in U$ is the graph joining of an automorphism of $(X, \mu, T)$. 

    Let $\lambda \in U$.
    By \cref{prop:ergodic_joinings_nilsystems_are_nilsystems}, $\lambda$ is the Haar measure of a $T \times T$-invariant subnilmanifold 
    \[
    Y = H \cdot (x_0, y_0) 
    \]
    of $X \times X$, where $H \leq G \times G$ is a closed subgroup and $(x_0, y_0) \in X \times X$. 
    Let 
    \[
    \pi_1 \colon X \times X \to X 
    \]
    be the first coordinate projection. 
    The restriction $\restrict{\pi_1}{Y}$ is surjective. 
    Indeed, since $\mu = (\pi_1)_*\lambda$, the closed set $\pi_1(Y)$ has full measure in $X$, and since $\mu$ has full support it follows that $\pi_1(Y) = X$.  

    Next, we show that $\restrict{\pi_1}{Y}$ is injective. 
    For $x \in X$, define the fiber 
    \[
    Y_x :=  Y \cap (\{x\} \times X).
    \]
    By \cref{lem:fibers_joining_are_subnilmanifolds}, the fibers $Y_x$ are pairwise homeomorphic subnilmanifolds of $X \times X$.  
    Suppose that $\restrict{\pi_1}{Y}$ is not injective. 
    Then $Y_x$ is non-trivial for some, and hence for every $x \in X$.  
    Thus, if $\lambda_x$ denotes the Haar measure of $Y_x$, then our choice of $\eta$ implies that 
    \begin{equation}
        \label{eq:eta_bound} 
        \int_{Y_x \times Y_x} d_{X \times X}(z, z') \dd\lambda_x(z) \dd\lambda_x(z') \geq \eta,  
    \end{equation}
    for every $x \in X$. 
    On the other hand, for $z = (x,y)$ and $z' = (x,y') \in Y_x$, we have by definition \eqref{eq:def_product_metric} that 
    \[
    d_{X \times X} (z,z') = d_X(y,y') \leq d_X(x,y) + d_X(x,y').  
    \]
    Plugging this into \eqref{eq:eta_bound}, we get for all $x \in X$ that  
    \[
    \eta \leq \int_{Y_x \times Y_x} (d_X(x,y)  + d_X(x,y')) \dd\lambda_x(x,y) \dd\lambda_x(x,y') 
    = 2 \int_{Y_x} d_X(x,y) \dd\lambda_x(x, y). 
    \]
    Integrating over $x$, and using the disintegration from \cref{lem:fibers_joining_are_subnilmanifolds}, it follows that 
    \[
    \eta/2  
    \leq \int_X \int_{Y_x} d_X(x,y) \dd\lambda_x(x,y) \dd\mu(x)
    = \int_Y d_X(x,y) \dd\lambda(x,y). 
    \]
    This contradicts the choice of $\lambda \in U$, so $\restrict{\pi_1}{Y}$ must be injective.  

    We have established that $\restrict{\pi_1}{Y} \colon Y \to X$ is a continuous bijection. 
    By compactness of $Y$, it must therefore be a homeomorphism. 
    By symmetry, we may apply the same argument to the second coordinate in order to conclude that the restriction of the second coordinate projection $\restrict{\pi_2}{Y} \colon Y \to X$ is a homeomorphism as well. 
    Define 
    \[
    \phi := \restrict{\pi_2}{Y} \circ (\restrict{\pi_1}{Y})^{-1} \colon X \to X. 
    \]
    Then $\phi$ is a homeomorphism, and since $\restrict{\pi_1}{Y}$ and $\restrict{\pi_2}{Y}$ are continuous factor maps, we have that 
    \[
    \phi \circ T 
    = \restrict{\pi_2}{Y} \circ (\restrict{\pi_1}{Y})^{-1} \circ T
    = \restrict{\pi_2}{Y} \circ (T \times T) \circ (\restrict{\pi_1}{Y})^{-1} 
    = T \circ \phi. 
    \]
    Since $(\restrict{\pi_1}{Y})_* \lambda = \mu$, it follows that $\lambda = ((\restrict{\pi_1}{Y})^{-1})_* \mu$. 
    Therefore 
    \[
    \phi_* \mu 
    = (\restrict{\pi_2}{Y})_* ((\restrict{\pi_1}{Y})^{-1})_* \mu = (\restrict{\pi_2}{Y})_* \lambda = \mu,   
    \]
    so $\phi$ is an automorphism of $(X, \mu, T)$. 
    Finally,
    \[
    \lambda = ((\restrict{\pi_1}{Y})^{-1})_* \mu 
    = (\Id_X, \restrict{\pi_2}{Y} \circ (\restrict{\pi_1}{Y})^{-1})_* \mu = (\Id_X, \phi)_* \mu, 
    \]
    so $\lambda$ is precisely the graph joining of $\phi$. 
\end{proof}

\subsection{Local rigidity for factor maps}

We conclude this chapter with a consequence of \cref{prop:rigidity_self_joinings}. 
It says that factor maps onto an ergodic nilsystem are locally rigid modulo automorphisms.
We will not use this result later. 

\begin{corollary}
    \label[corollary]{cor:rigidty_factor_maps}
    Let $(Y, \nu, S)$ be an ergodic nilsystem.  
    Then there exists $\delta > 0$ such that for every ergodic system $(X, \mu, T)$, if $\pi, \phi \colon X \to Y$ are factor maps satisfying
    \[
    d((\pi, \phi)_* \mu, \nu_\Delta) < \delta,
    \] 
    then there exists an automorphism $\psi$ of $(Y, \nu, S)$ such that $\phi = \psi \circ \pi$ $\mu$-almost everywhere. 
\end{corollary}

Here $d$ denotes a fixed compatible metric on $\Pr(Y \times Y)$, and $\nu_\Delta$ is the diagonal self-joining of $\nu$. 

\begin{proof}
    Applying \cref{prop:rigidity_self_joinings} to the ergodic nilsystem $(Y, \nu, S)$, we obtain $\delta > 0$ such that every $\lambda \in J_e(Y, Y)$ with $d(\lambda, \nu_\Delta) < \delta$ is the graph joining of an automorphism of $(Y, \nu, S)$. 

    Now let $(X, \mu, T)$ be an ergodic system, and let $\pi, \phi \colon X \to Y$ be factor maps with $d((\pi, \phi)_*\mu, \nu_\Delta) < \delta$. 
    Let $\lambda := (\pi, \phi)_* \mu$. 
    Since $(X, \mu, T)$ is ergodic, $\lambda \in J_e(Y, Y)$.
    Hence $\lambda$ is the graph joining of an automorphism $\psi$.  
    It follows for $\mu$-almost every $x \in X$ that $(\pi(x), \phi(x))$ lies on the graph of $\psi$, equivalently that $\phi(x) = \psi(\pi(x))$.   
\end{proof}

\chapter{Factor-closure of pro-nilsystems}
\label{chap:factor_closure}

In this chapter we prove the central theorem of this thesis as announced in the introduction: the class of ergodic pro-nilsystems is closed under taking factors.  
We state it here in the form and notation used throughout the proof. 

\begin{theorem}[Factor-closure of ergodic pro-nilsystems]
    \label{thrm:main}
    Let $s \geq 1$, and let $(X, \mu, T)$ be an ergodic $s$-step pro-nilsystem. 
    Suppose that 
    \[
    \pi \colon (X, \mu, T) \to (Y, \nu, S) 
    \]
    is a measure-theoretic factor map. 
    Then $(Y, \nu, S)$ is an $s$-step pro-nilsystem. 
\end{theorem}

Since factors of ergodic systems are ergodic, it follows that $(Y, \nu, S)$ is automatically an ergodic $s$-step pro-nilsystem. 
Thus, the theorem indeed says that the class of ergodic $s$-step pro-nilsystems is closed under taking factors. 

This theorem is an immediate consequence of the Host--Kra structure theorem, which was discussed in the introduction. 
Indeed, the structure theorem identifies the ergodic $s$-step pro-nilsystems precisely as the ergodic systems of order $s$, and this latter class is known to be closed under taking factors; see \citep[Chapter 9, Proposition 17]{host2018nilpotent}.
The goal of the present chapter is to give a proof which does not use the Host--Kra structure theorem. 
Instead, we prove the theorem directly using only the theory of nilsystems developed in \cref{chap:niltheory}, and the local rigidity result from \cref{chap:rigidity_result}, whose proof also relied only on intrinsic nilsystem arguments.  
For single nilsystems, the corresponding statement is classical: every factor of an ergodic $s$-step nilsystem is again an ergodic $s$-step nilsystem; see \cref{thrm:factor_ergodic_nilsystem_is_nilsystem}. 

The chapter is organized as follows. 
\cref{sec:proof_strategy} contains an outline of the argument. 
\cref{sec:proof_special_case_main_thrm} proves the special case where the factor map is an extension by a compact abelian group. 
\cref{sec:pronilsystem_tower} constructs the tower of factors of a pro-nilsystem, and \cref{sec:proof_main_thrm} combines these two ingredients to prove \cref{thrm:main}. 
Finally, in \cref{sec:topological_factor_closure} we derive the analogous factor-closure result for transitive topological pro-nilsystems. 

\section{Proof strategy}
\label{sec:proof_strategy}

The proof of \cref{thrm:main} has two main steps. 
First, we reduce the theorem to the case where the factor map is an extension by a compact abelian group. 
Then, we prove this special case using the local rigidity result from \cref{chap:rigidity_result}. 

For the first step, we begin by passing the tower of factors of a nilsystem from \cref{sec:tower_of_factors_nilsystem} to the inverse limit. 
More precisely, in \cref{sec:pronilsystem_tower} we show that every ergodic $s$-step pro-nilsystem $X$ admits a tower of factors 
\[
X = Z_s \longrightarrow Z_{s-1} \longrightarrow \cdots \longrightarrow Z_1 \longrightarrow Z_0 = \{*\}, 
\]
where each map $Z_i \to Z_{i-1}$ is an extension by a compact abelian group. 
Given a factor map $\pi \colon X \to Y$, we let $W_i$ be the smallest factor of $X$ containing both $Y$ and $Z_i$ as factors. 
Since the factors $Z_i$ are increasing, so are the factors $W_i$, hence they form a tower 
\[
X = W_s \longrightarrow W_{s-1} \longrightarrow \cdots \longrightarrow W_1 \longrightarrow W_0 = Y. 
\]
It therefore suffices to show that each $W_i$ is an $s$-step pro-nilsystem. 
We do this by downward induction. 
For $i = s$, $W_s = X$ is an $s$-step pro-nilsystem by assumption. 
By \cref{lem:induced_group_ext}, each map $W_i \to W_{i-1}$ is still an extension by a compact abelian group. 
Thus the induction step is reduced to the special case where the factor map is a compact abelian group extension, which is recorded in \cref{thrm:main_special_case}. 

To prove this special case, suppose that $\pi \colon (X, \mu, T) \to (Y, \nu, S)$ is an extension by the compact abelian group $K$. 
Choose an inverse limit 
\[
(X, \mu, T) = \varprojlim (X_n, \mu_n, T_n), 
\]
where each $X_n$ is an ergodic $s$-step nilsystem. 
Let $p_n \colon X \to X_n$ denote the associated factor maps. 
We would like to use the finite-stage nilsystems $X_n$ to construct an inverse limit for $Y$. 
The main obstruction is that the given inverse limit for $X$ need not be compatible with the group extension structure. 
More concretely, the vertical rotations $V_u$ for $u \in K$ do not in general descend to automorphisms of the systems $X_n$. 
Equivalently, the factors $X_n$ need not be invariant under all vertical rotations. 

This is the point where the local rigidity result comes in. 
For $u \in K$, define the probability measure 
\[
\lambda_u^{(n)} := (p_n, p_n \circ V_u)_* \mu
\]
on $X_n \times X_n$. 
This is an ergodic self-joining of $X_n$, and $V_u$ descends to an automorphism of $X_n$ precisely when $\lambda_u^{(n)}$ is the graph joining of an automorphism. 
The map $u \mapsto \lambda_u^{(n)}$ is continuous and maps $0$ to the diagonal joining. 
Therefore, \cref{prop:rigidity_self_joinings} shows that $\lambda_u^{(n)}$ is the graph joining of an automorphism for all $u$ in some open subgroup $H_n \leq K$. 
By compactness, the quotient $K/H_n$ is finite. 
Thus we may choose a finite set $R_n$ of representatives for this quotient.
After enlarging $X_n$ to the finite self-joining generated by the factors $V_r^{-1}(\mcx_n)$ for $r \in R_n$, we obtain a  new increasing sequence of $s$-step nilsystem factors $\wt X_n$ which are invariant under all vertical rotations. 

Finally, let $W_n$ be the largest common factor of $Y$ and $\wt X_n$. 
These are ergodic $s$-step nilsystems, and an averaging argument over $K$, using the invariance of the systems $\wt X_n$ under the vertical rotations, shows that $Y = \varprojlim W_n$. 
Hence $Y$ is an $s$-step pro-nilsystem. 

\section{The compact abelian group extension case}
\label{sec:proof_special_case_main_thrm}

We now begin with the proof of \cref{thrm:main} in the special case where the factor map is an extension by a compact abelian group. 

\begin{theorem}[Compact abelian group extension case]
    \label{thrm:main_special_case} 
    Let $s \geq 1$, and let $(X, \mu, T)$ be an ergodic $s$-step pro-nilsystem. 
    Suppose that 
    \[
    \pi \colon (X, \mu, T) \to (Y, \nu, S)
    \]
    is a factor map which is an extension by a compact abelian group $K$. 
    Then $(Y, \nu, S)$ is an $s$-step pro-nilsystem. 
\end{theorem}

\begin{proof}
    By definition of a compact abelian group extension, see \cref{sec:group_extensions_measurable}, we may assume that 
    \[
    (X, \mu, T) = (Y \times K, \nu \times m_K, T_\rho), 
    \] 
    where $m_K$ is the Haar measure on $K$, $\rho \colon Y \to K$ is a measurable map, $T_\rho(y, g) = (Sy, \rho(y) + g)$ for $(y, g) \in Y \times K$, and $\pi$ is the coordinate projection onto $Y$.   
    Recall that for $u \in K$, the vertical rotation 
    \[
    V_u \colon X \to X, \quad V_u(y, g) = (y, u + g)  
    \]
    is an automorphism of $(X, \mu, T)$. 
    These vertical rotations satisfy $V_u \circ V_v = V_{u + v}$ and $V_u^{-1} = V_{-u}$ almost everywhere, for all $u, v \in K$. 

    Fix an inverse limit  
    \[
    (X, \mu, T) = \varprojlim (X_n, \mu_n, T_n)
    \]
    where each $(X_n, \mu_n, T_n)$ is an ergodic $s$-step nilsystem, and let $p_n \colon X \to X_n$ denote the associated factor maps. 

    For each $n \in \N$ and each $u \in K$, define a probability measure on $X_n \times X_n$ by 
    \begin{equation}
        \label{eq:def_lambda_u^n}    
        \lambda_u^{(n)} := (p_n, p_n \circ V_u)_* \mu.  
    \end{equation}
    Since $p_n$ is a factor map and $V_u$ is an automorphism, $\lambda_u^{(n)}$ is a self-joining of $(X_n, \mu_n, T_n)$. 
    Moreover, the map $(p_n, p_n \circ V_u)$ realizes the system $(X_n \times X_n, \lambda_u^{(n)}, T_n \times T_n)$ as a factor of the ergodic system $(X, \mu, T)$, and hence it is ergodic. 
    Therefore, 
    \[
    \lambda_u^{(n)} \in J_e(X_n, X_n). 
    \]
    For each $n \in \N$, define
    \[
    H_n := \{
        u \in K : \lambda_u^{(n)} \text{ is the graph joining of an automorphism of } X_n
    \}. 
    \] 
    We claim that $H_n$ is a subgroup of $K$. 
    First, $\lambda_0^{(n)} = (p_n, p_n)_* \mu$ is the diagonal joining, which is the graph of the identity map, so that $0 \in H_n$.   
    Next, let $u, v \in H_n$, and choose automorphisms $\phi_u$ and $\phi_v$ of $(X_n, \mu_n, T_n)$ such that 
    \[
    p_n \circ V_u = \phi_u \circ p_n 
    \quad \text{and} \quad 
    p_n \circ V_v = \phi_v \circ p_n \quad \mu\aee
    \]
    Then
    \[
    p_n \circ V_{u + v} = p_n \circ V_u \circ V_v = \phi_u \circ p_n \circ V_v = \phi_u \circ \phi_v \circ p_n \quad \mu\aee
    \]
    so that $\lambda_{u+v}^{(n)}$ is the graph joining of the automorphism $\phi_u \circ \phi_v$. 
    Thus $u + v \in H_n$. 
    Similarly, 
    \[
    \phi_u^{-1} \circ p_n = \phi_u^{-1} \circ p_n \circ V_u \circ V_{-u} = \phi_u^{-1} \circ \phi_u \circ p_n \circ V_{-u} = p_n \circ V_{-u} \quad \mu\aee, 
    \]
    so $\lambda_{-u}^{(n)}$ is the graph joining of the inverse automorphism $\phi_u^{-1}$, and thus $-u \in H_n$. 
    This proves the claim. 

    We next prove that $H_n$ is open in $K$. 
    Recall from \cref{chap:rigidity_result} that $J_e(X_n, X_n)$ is equipped with the weak-* topology under the identification with a subspace of the dual of $C(X_n \times X_n)$. 
    We first show that the map 
    \[
    K \to J_e(X_n, X_n), \quad u \mapsto \lambda_u^{(n)} 
    \]
    is continuous. 
    Let $u_i \to u$ in $K$, and let $f, g \in C(X_n)$. 
    Writing $F: = f \circ p_n$ and $G := g \circ p_n$, we want to compute the limit of
    \begin{equation}
        \label{eq:weak_*_cont_check} 
        \int_{X_n \times X_n} f(x) g(y) \dd\lambda_{u_i}^{(n)}(x, y) 
        = \int_X F \cdot (G \circ V_{u_i}) \dd\mu.  
    \end{equation}
    Recall from \cref{sec:group_extensions_measurable} that the Koopman representation of $K$ on $L^2(X, \mu)$ associated to the vertical rotations is strongly continuous. 
    Hence $G \circ V_{u_i} \to G \circ V_u$ in $L^2(X, \mu)$.  
    Therefore, 
    \[
    \int_X F \cdot (G \circ V_{u_i}) \dd\mu \to \int_X F \cdot (G \circ V_u) \dd\mu.  
    \]
    Plugging this into \eqref{eq:weak_*_cont_check}, we see that  
    \[
    \int_{X_n \times X_n} f(x) g(y) \dd\lambda_{u_i}^{(n)}(x, y) 
    \to \int_{X_n \times X_n} f(x) g(y) \dd\lambda_u^{(n)}(x,y). 
    \]
    Since the linear span of the functions $(x,y) \mapsto f(x) g(y)$ with $f, g \in C(X_n)$ is dense in $C(X_n \times X_n)$ by Stone--Weierstrass, this shows that $\lambda_{u_i}^{(n)} \to \lambda_u^{(n)}$ in the weak-* topology. 
    By \cref{prop:rigidity_self_joinings}, there exists an open neighborhood 
    \[
    U_n \subseteq J_e(X_n, X_n) 
    \]
    of the diagonal joining such that every element of $U_n$ is the graph joining of an automorphism of $X_n$.   
    Since $\lambda_0^{(n)}$ is the diagonal joining, the preimage of $U_n$ under the continuous map $u \mapsto \lambda_u^{(n)}$ is an open neighborhood of $0$ contained in $H_n$. 
    Therefore $H_n$ is an open subgroup. 

    Since $H_n$ is an open subgroup of the compact group $K$, the quotient $K/H_n$ is finite. 
    For each $n \in \N$, fix a finite set $R_n \subseteq K$ of representatives for this quotient, so that   
    \[
    R_n + H_n = K.  
    \]
    Replacing $R_n$ by $R_1 \cup \cdots \cup R_n \cup \{0\}$ if necessary, we may assume that $0 \in R_n$ for all $n$ and that $R_n \subseteq R_m$ whenever $n \leq m$. 

    For each $n \in \N$ and each $u \in H_n$, choose an automorphism $V_u^{(n)}$ of $(X_n, \mu_n, T_n)$ such that $\lambda_u^{(n)}$ is the graph joining of $V_u^{(n)}$. 
    By definition \eqref{eq:def_lambda_u^n} of $\lambda_u^{(n)}$, this means that 
    \[
    V_u^{(n)} \circ p_n = p_n \circ V_u \quad \mu\aee 
    \]
    In particular, the sub-$\sigma$-algebra $p_n^{-1}(\mcx_n)$ is invariant under the vertical rotations $V_u$ for $u \in H_n$. 

    We now enlarge the factor $X_n$ to obtain a factor $\wt X_n$ which is invariant under all vertical rotations. 
    For each $n \in \N$, define 
    \[
    \wt X_n := X_n^{R_n} = \prod_{r \in R_n} X_n,  
    \qquad
    q_n \colon X \to \wt X_n, \quad q_n(x) := (p_n(V_r x))_{r \in R_n}. 
    \]
    Let $\wt \mu_n := (q_n)_* \mu$, and let $\wt T_n$ be the transformation on $\wt X_n$ obtained by applying $T_n$ componentwise.
    Since $V_r$ commutes with $T$ for each $r$, and since $p_n$ intertwines $T$ and $T_n$, the map 
    \[
    q_n \colon (X, \mu, T) \to (\wt X_n, \wt \mu_n, \wt T_n) 
    \]
    is a factor map. 
    In particular, the system $(\wt X_n, \wt \mu_n, \wt T_n)$ is ergodic. 
    Each coordinate projection of $\wt \mu_n$ is equal to $(p_n \circ V_r)_* \mu = \mu_n$, so $\wt \mu_n$ is an ergodic $|R_n|$-fold self-joining of the nilsystem $(X_n, \mu_n, T_n)$. 
    Inductively applying \cref{prop:ergodic_joinings_nilsystems_are_nilsystems} then shows that $(\wt X_n, \wt \mu_n, \wt T_n)$ is isomorphic to an ergodic $s$-step nilsystem. 

    The sub-$\sigma$-algebra corresponding to the factor $\wt X_n$ is 
    \[
    q_n^{-1} (\wt \mcx_n)
    = \bigvee_{r \in R_n} V_r^{-1}(p_n^{-1}(\mcx_n)).
    \]
    We claim that this is invariant under all vertical rotations. 
    Let $u \in K$. 
    For each $r \in R_n$, choose $s(r) \in R_n$ and $h(r) \in H_n$ such that $r + u = s(r) + h(r)$. 
    Then 
    \[
    V_u^{-1}(V_r^{-1}(p_n^{-1}(\mcx_n))) 
    = V_{r + u}^{-1}(p_n^{-1}(\mcx_n))
    = V_{s(r)}^{-1}(V_{h(r)}^{-1}(p_n^{-1}(\mcx_n)))
    =_\mu V_{s(r)}^{-1}(p_n^{-1}(\mcx_n)),
    \]
    where in the last equality, we used the invariance of $p_n^{-1}(\mcx_n)$ under translations by elements of $H_n$. 
    Therefore,
    \[
    V_u^{-1}(q_n^{-1}(\wt \mcx_n)) 
    =_\mu \bigvee_{r \in R_n} V_{s(r)}^{-1}(p_n^{-1}(\mcx_n))
    \subseteq_\mu q_n^{-1}(\wt \mcx_n).
    \]
    Since $V_u$ is invertible, the converse inclusion also holds, so 
    \[
    V_u^{-1}(q_n^{-1}(\wt \mcx_n)) =_\mu q_n^{-1}(\wt \mcx_n) \quad \text{for all } u \in K. 
    \]
    Since $0 \in R_n$, the factor $\wt X_n$ lies above $X_n$. 
    Moreover, since both the factors $X_n$ and the sets $R_n$ are increasing in $n$, the sub-$\sigma$-algebras $q_n^{-1}(\wt \mcx_n)$ are increasing in $n$. 

    For each $n \in \N$, define 
    \[
    \mcw_n := \overline{\pi^{-1}(\mcy)}^\mu \cap \overline{q_n^{-1}(\wt \mcx_n)}^\mu, 
    \]
    where the bar denotes $\mu$-completion. 
    Let $W_n$ be the corresponding factor of $X$. 
    This is the largest common factor of $Y$ and $\wt X_n$. 
    Since $W_n$ is a factor of the ergodic $s$-step nilsystem $\wt X_n$, it is isomorphic to an $s$-step nilsystem by \cref{thrm:factor_ergodic_nilsystem_is_nilsystem}. 
    Moreover, the $\sigma$-algebras $\mcw_n$ are increasing in $n$, because the sequence $q_n^{-1}(\wt \mcx_n)$ is increasing.

    We now show that the increasing family of sub-$\sigma$-algebras $(\mcw_n)_{n \in \N}$ generates $\pi^{-1}(\mcy)$ modulo $\mu$. 
    Let $f \in L^2(X, \pi^{-1}(\mcy), \mu)$, and let $\epsilon > 0$.
    Since $X = \varprojlim X_i$ and $\wt \mcx_n$ lies above $X_n$, there exist $n_0 \in \N$ and $g \in L^2(X, q_{n_0}^{-1}(\wt \mcx_{n_0}), \mu)$ such that $\| f - g \|_{L^2(\mu)} < \epsilon$; see \cref{sec:inverse_limits}.  
    Set 
    \[
    h := \E_\mu (g \mid \pi^{-1}(\mcy)).  
    \]
    By the averaging formula from \cref{sec:group_extensions_measurable}, 
    \[
    h(x) = \int_K g(V_u x) \dd m_K(u) \quad \text{for } \mu\aee \ x \in X. 
    \]
    By strong continuity of the Koopman representation, the map $u \mapsto g \circ V_u$ is continuous from $K$ to $L^2(X,\mu)$. 
    An application of Fubini's theorem shows that $h$ coincides with the $L^2(X,\mu)$-valued Bochner integral $\int_K (g \circ V_u) \dd m_K(u)$. 
    Since $q_{n_0}^{-1}(\wt \mcx_{n_0})$ is invariant under all vertical rotations, each $g \circ V_u$ lies in the closed subspace $L^2(X, q_{n_0}^{-1}(\wt \mcx_{n_0}), \mu)$.
    Thus $h \in L^2(X, q_{n_0}^{-1}(\wt \mcx_{n_0}), \mu)$. 
    On the other hand, $h$ is $\pi^{-1}(\mcy)$-measurable by construction. 
    Therefore, $h \in L^2(X, \mcw_{n_0}, \mu)$.  
    Moreover, 
    \[
    \| f - h \|_{L^2(\mu)} 
    = \| \E_\mu( f - g \mid \pi^{-1}(\mcy)) \|_{L^2(\mu)} 
    \leq \| f - g \|_{L^2(\mu)} 
    < \epsilon.
    \]
    Since $f \in L^2(X, \pi^{-1}(\mcy), \mu)$ and $\epsilon > 0$ were arbitrary, it follows that $\pi^{-1}(\mathcal{Y}) \subseteq_\mu \bigvee_{n \in \N} \mcw_n$. 
    The reverse inclusion is immediate from the definition of $\mcw_n$. 
    Hence 
    \[
    \pi^{-1}(\mathcal{Y}) =_\mu \bigvee_{n \in \N} \mcw_n.  
    \]
    By the inverse limit criterion for increasing families of factors in \cref{lem:inverse_limit_criterion}, $(Y, \nu, S)$ is the inverse limit of the ergodic $s$-step nilsystems $W_n$.   
    Hence $(Y, \nu, S)$ is an $s$-step pro-nilsystem.  
\end{proof}

\section{The tower of factors of a pro-nilsystem}
\label{sec:pronilsystem_tower}

Recall from \cref{sec:tower_of_factors_nilsystem} that a nilsystem $(X, \mu, T)$ admits a natural tower of factors 
\[
X = Z_s \longrightarrow Z_{s-1} \longrightarrow \cdots \longrightarrow Z_1 \longrightarrow Z_0 = \{*\}, 
\]
where for $i = 1, \dots, s$, the map $Z_i \to Z_{i-1}$ is an extension by the structure group $K_i$. 
In this section, we pass this construction to the inverse limit and obtain an analogous tower for an arbitrary ergodic $s$-step pro-nilsystem. 

The key idea is that a factor map between ergodic nilsystems in reduced form induces a homomorphism between their respective towers. 

\begin{lemma}
    \label[lemma]{lem:tower_morphism}
    Let $(X = G/\Gamma, \mu, T)$ and $(X' = G'/\Gamma', \mu', T')$ be ergodic $s$-step nilsystems in reduced form, and let 
    \[
    X = Z_s \xrightarrow{\pi_{s-1}} Z_{s-1} \xrightarrow{\pi_{s-2}} \cdots \xrightarrow{\pi_0} Z_0 = \{*\} 
    \]
    and 
    \[
    X' = Z_s' \xrightarrow{\pi_{s-1}'} Z_{s-1}' \xrightarrow{\pi_{s-2}'} \cdots \xrightarrow{\pi_0'} Z_0' = \{*\} 
    \]
    be their respective towers of factors. 
    Suppose that $\alpha \colon (X, \mu, T) \to (X', \mu', T')$ is a factor map. 
    Then there exist topological factor maps 
    \[
    \alpha_i \colon Z_i \to Z_i' \quad \text{for } i = 0, \dots, s 
    \]
    with $\alpha_s = \alpha$ almost everywhere, making the diagram 
    \[\begin{tikzcd}
        {X = Z_s} & {Z_{s-1}} & \cdots & {Z_1} & {Z_0 = \{*\}} \\
        {X' = Z_s'} & {Z_{s-1}'} & \cdots & {Z_1'} & {Z_0' = \{*\}}
        \arrow["{\pi_{s-1}}", from=1-1, to=1-2]
        \arrow["{\alpha_s}", from=1-1, to=2-1]
        \arrow["{\pi_{s-2}}", from=1-2, to=1-3]
        \arrow["{\alpha_{s-1}}", from=1-2, to=2-2]
        \arrow["{\pi_1}", from=1-3, to=1-4]
        \arrow["{\pi_0}", from=1-4, to=1-5]
        \arrow["{\alpha_1}", from=1-4, to=2-4]
        \arrow["{\alpha_0}", from=1-5, to=2-5]
        \arrow["{\pi_{s-1}'}"', from=2-1, to=2-2]
        \arrow["{\pi_{s-2}'}"', from=2-2, to=2-3]
        \arrow["{\pi_1'}"', from=2-3, to=2-4]
        \arrow["{\pi_0'}"', from=2-4, to=2-5]
    \end{tikzcd}\]
    commute. 
    Moreover, if $K_i$ and $K_i'$ denote the structure groups of $X$ and $X'$ respectively, then for $i = 1, \dots, s$, there is a continuous surjective homomorphism $\xi_i \colon K_i \to K_i'$ such that 
    \[
    \alpha_i (u \cdot z) = \xi_i(u) \cdot \alpha_i(z) 
    \]
    for all $z \in Z_i$ and $u \in K_i$. 
\end{lemma}

\begin{proof}
    Suppose that $T$ and $T'$ are the translations by $\tau \in G$ and $\tau' \in G'$ respectively. 
    By the reduced form assumption and \cref{prop:factor_maps_ergodic_nilsystems_are_algebraic}, we may assume after modifying $\alpha$ on a null set that 
    \[
    \alpha(g\Gamma) = a \phi(g) \Gamma' \quad \text{for } g\Gamma \in X,  
    \]
    where $\phi \colon G \to G'$ is a continuous surjective group homomorphism with $\phi(\Gamma) \subseteq \Gamma'$, and $a \in G'$ satisfying $\phi(\tau) = a^{-1} \tau' a$. 

    Since $\phi$ is surjective, $\phi(G_{i+1}) = G'_{i+1}$ for $i = 0, \dots, s$. 
    Using the identifications $Z_i \cong G/(G_{i+1} \Gamma)$ and $Z_i' \cong G'/(G_{i+1}' \Gamma')$ from \cref{sec:tower_of_factors_nilsystem}, it follows that $\alpha$ descends to the maps
    \[
    \alpha_i \colon Z_i \to Z_i', 
    \quad g G_{i+1}\Gamma \mapsto a \phi(g) G_{i+1}'\Gamma'. 
    \]
    Each $\alpha_i$ is continuous, surjective and intertwines the induced translations by $\tau$ and $\tau'$ on $Z_i$ and $Z_i'$. 
    Hence it is a topological factor map, and by unique ergodicity also a measure-theoretic factor map. 

    We claim that these factor maps make the above diagram commute. 
    Indeed, for $i = 0, \dots, s-1$, and $g G_{i+2} \Gamma \in Z_{i+1}$, 
    \[
    \alpha_i (\pi_i(g G_{i+2} \Gamma)) 
    = \alpha_i(g G_{i+1} \Gamma)
    = a \phi(g) G_{i+1}' \Gamma'
    = \pi_i'(a \phi(g) G_{i+2}' \Gamma')
    = \pi_i'(\alpha_{i+1}(g G_{i+2} \Gamma)), 
    \]
    so $\alpha_i \circ \pi_i = \pi_i' \circ \alpha_{i+1}$.

    For the structure groups, we use the identifications $K_i \cong G_i/(G_i \cap G_{i+1} \Gamma)$ and $K_i' \cong G_i'/(G_i' \cap G_{i+1}' \Gamma')$. 
    Since $\phi(G_i \cap G_{i+1} \Gamma) \subseteq G_i' \cap G_{i+1}'\Gamma'$, the homomorphism $\phi$ descends, for $i = 1, \dots, s$, to a continuous surjective group homomorphism
    \[
    \xi_i \colon K_i \to K_i', \quad 
    g (G_i \cap G_{i+1} \Gamma) \mapsto \phi(g) (G_i' \cap G_{i+1}' \Gamma'). 
    \]
    For $u = g(G_i \cap G_{i+1} \Gamma) \in K_i$ and $z = h G_{i+1} \Gamma \in Z_i$, we compute
    \[
    \alpha_i (u \cdot z)
    = \alpha_i(g h G_{i+1} \Gamma)
    = a \phi(g) \phi(h) G_{i+1}' \Gamma'
    = \phi(g) a \phi(h) G_{i+1}' \Gamma'
    = \xi_i(u) \cdot \alpha_i(z). 
    \]
    Here we used that $\phi(g) \in \phi(G_i) = G_i'$ is central modulo $G_{i+1}'$. 
\end{proof}

We are now ready to construct the tower of factors for a pro-nilsystem. 

\begin{proposition}
    \label[proposition]{prop:tower_of_factors_pronilsystem}
    Let $s \geq 1$, and let $(X, \mu, T)$ be an ergodic $s$-step pro-nilsystem. 
    Then there exists a tower of factors 
    \[
    X = Z_s \xrightarrow{\pi_{s-1}} Z_{s-1} \xrightarrow{\pi_{s-2}} \cdots \xrightarrow{\pi_1} Z_1 \xrightarrow{\pi_0} Z_0 = \{*\},  
    \]
    such that for $i = 1, \dots, s$, $Z_i$ is an $i$-step pro-nilsystem, and the map $Z_i \to Z_{i-1}$ is an extension by a compact abelian group $K_i$.  
\end{proposition}

\begin{proof}
    Choose an inverse limit 
    \[
    (X, \mu, T) = \varprojlim (X_n, \mu_n, T_n) 
    \]
    where each $(X_n, \mu_n, T_n)$ is an ergodic $s$-step nilsystem. 
    By \cref{lem:reduced_form}, we may assume that each $X_n$ is in reduced form. 
    Let $\alpha_n \colon X_{n+1} \to X_n$ denote the connecting factor maps. 

    For each $n \in \N$, let 
    \[
    X_n = Z_{n,s} \xrightarrow{\pi_{n,s-1}} Z_{n,s-1} \xrightarrow{\pi_{n,s-2}} \cdots \xrightarrow{\pi_{n,1}} Z_{n,1} \xrightarrow{\pi_{n,0}} Z_{n,0} = \{*\}
    \]
    be the tower of factors of $X_n$ from \cref{sec:tower_of_factors_nilsystem}, with structure groups $K_{n,i}$ for $i = 1, \dots, s$. 
    Each $(Z_{n,i}, \mu_{n,i}, T_{n,i})$ is an $i$-step nilsystem, and it is ergodic as a factor of the ergodic system $X_n$, hence uniquely ergodic by \cref{cor:equivalence_ergodicity_minimality_etc_nilsystems}.  
    Applying \cref{lem:tower_morphism} to the factor map $\alpha_n \colon X_{n+1} \to X_n$, we obtain for each $i = 0, \dots, s$ a topological factor map 
    \[
    \alpha_{n, i} \colon Z_{n+1, i} \to Z_{n,i}, 
    \]
    with $\alpha_{n,s} = \alpha_n$ almost everywhere, satisfying $\alpha_{n,i} \circ \pi_{n+1, i} = \pi_{n,i} \circ \alpha_{n, i+1}$. 
    Moreover, we obtain for $i = 1, \dots, s$ a continuous surjective homomorphism
    \[
    \xi_{n,i} \colon K_{n+1, i} \to K_{n, i}
    \]
    such that 
    \[
    \alpha_{n,i} (u \cdot z) = \xi_{n,i}(u) \cdot \alpha_{n,i}(z)
    \] 
    for each $u \in K_{n+1, i}$ and each $z \in Z_{n+1, i}$.  

    Fix $i = 0, \dots, s$. 
    Since the connecting factor maps $\alpha_{n,i}$ are topological factor maps, we may take the topological inverse limit 
    \[
    Z_i = \left\{
        (z_n)_{n \in \N} \in \prod_{n \in \N} Z_{n,i} : \alpha_{n,i}(z_{n+1}) = z_n \text{ for all } n \in \N
    \right\}, 
    \] 
    equipped with the transformation $T_i((z_n)_{n \in \N}) = (T_{n,i} z_n)_{n \in \N}$. 
    Since each system $Z_{n,i}$ is uniquely ergodic, \cref{lem:dynamical_prop_inverse_limit} implies that $(Z_i, T_i)$ is uniquely ergodic. 
    Let $\mu_i$ be its unique invariant probability measure. 
    Again by \cref{lem:dynamical_prop_inverse_limit},
    \[
    (Z_i, \mu_i, T_i) = \varprojlim (Z_{n,i}, \mu_{n,i}, T_{n,i}) 
    \]
    in the measure-theoretic sense. 
    In particular, for $i = 1, \dots, s$, $(Z_i, \mu_i, T_i)$ is an ergodic $i$-step pro-nilsystem. 
    Since $Z_{n,0} = \{*\}$ for all $n$, we have $Z_0 = \{*\}$. 
    Moreover, for $i = s$, since $\alpha_{n,s} = \alpha_n$ almost everywhere, we can assume that $Z_s = X$.  

    For $i = 0, \dots, s-1$, define 
    \[
    \pi_i \colon Z_{i+1} \to Z_i, \quad 
    (z_n)_{n \in \N} \mapsto (\pi_{n,i}(z_n))_{n \in \N}. 
    \]
    This map is well-defined since for $n \in \N$ and $z = (z_n)_{n \in \N} \in Z_{i+1}$, 
    \[
    \alpha_{n,i} (\pi_{n+1, i} (z_{n+1}))   
    = \pi_{n,i}(\alpha_{n,i+1}(z_{n+1}))
    = \pi_{n,i}(z_n), 
    \]
    so that $\pi_i(z) \in Z_i$.
    The map $\pi_i$ is continuous and satisfies $\pi_i \circ T_{i+1} = T_i \circ \pi_i$. 
    It is also surjective since $\pi_{n,i}$ is surjective for each $n$; see for example \citep[Theorem 29.13]{willard2012general}.
    We conclude that $\pi_i \colon Z_{i+1} \to Z_i$ is a topological factor map, and by unique ergodicity also a measure-theoretic factor map. 

    So far we have constructed the tower of factors 
    \[
    X = Z_s \xrightarrow{\pi_{s-1}} Z_{s-1} \xrightarrow{\pi_{s-2}} \cdots \xrightarrow{\pi_1} Z_1 \xrightarrow{\pi_0} Z_0 = \{*\},  
    \]
    where for $i = 1, \dots, s$, $Z_i$ is an $i$-step pro-nilsystem.
    It remains to show that the factor maps $\pi_i$ are compact abelian group extensions.

    For $i = 1, \dots, s$ recall that the structure group $K_{n,i}$ is a compact abelian Lie group for each $n$. 
    Define
    \[
    K_i := \varprojlim K_{n, i} = \left\{
        (u_n)_{n \in \N} \in \prod_{n \in \N} K_{n,i} : \xi_{n,i}(u_{n+1}) = u_n \text{ for all } n \in \N 
    \right\}.  
    \]
    This is a closed subgroup of the product, hence a compact abelian group. 
    The group $K_i$ acts componentwise on $Z_i$: 
    for $u = (u_n)_{n \in \N} \in K_i$ and $z = (z_n)_{n \in \N} \in Z_i$, let 
    \[
    u \cdot z := (u_n \cdot z_n)_{n \in \N}. 
    \]
    This is well-defined, since for all $n \in \N$, 
    \[
    \alpha_{n,i}(u_{n+1} \cdot z_{n+1}) 
    = \xi_{n,i}(u_{n+1}) \cdot \alpha_{n,i}(z_{n+1})
    = u_n \cdot z_n.
    \]
    The action is continuous, and it is free and commutes with $T_i$, because for each $n$, the $K_{n,i}$-action on $Z_{n,i}$ is free and commutes with $T_{n,i}$.  

    We claim that the fibers of $\pi_{i-1}$ are exactly the $K_i$-orbits. 
    First, since the fibers of $\pi_{n,i-1}$ are the $K_{n,i}$-orbits, 
    \[
    \pi_{i-1}(u \cdot z) 
    = (\pi_{n,i-1}(u_n \cdot z_n))_{n \in \N}
    = (\pi_{n,i-1}(z_n))_{n \in \N}
    = \pi_{i-1}(z),  
    \]
    so $\pi_{i-1}$ is constant on the $K_i$-orbits. 
    Conversely, let $z = (z_n)_{n \in \N}$ and $z' = (z_n')_{n \in \N} \in Z_i$ with $\pi_{i-1}(z) = \pi_{i-1}(z')$. 
    For each $n$, there exists $u_n \in K_{n,i}$ with $u_n \cdot z_n = z_n'$ and $u_n$ is unique by freeness of the $K_{n,i}$-action. 
    Then, for every $n$, 
    \[
    \xi_{n,i}(u_{n+1}) \cdot z_n 
    = \xi_{n,i}(u_{n+1}) \cdot \alpha_{n,i}(z_{n+1})
    = \alpha_{n,i}(u_{n+1} \cdot z_{n+1})
    = \alpha_{n,i}(z_{n+1}')
    = z_n'
    = u_n \cdot z_n 
    \]
    so $\xi_{n,i}(u_{n+1}) = u_n$ by freeness. 
    Hence $u := (u_n)_{n \in \N} \in K_i$ and $u \cdot z = z'$, that is, $z$ and $z'$ lie in the same $K_i$-orbit.     

    We conclude for $i = 1, \cdots, s$ that the factor map $\pi_{i-1} \colon Z_i \to Z_{i-1}$ is an extension by the compact abelian group $K_i$ in the topological sense. 
    Finally, by \cref{lem:topological_group_extension_is_measurable} and unique ergodicity, it is therefore also an extension by $K_i$ in the measure-theoretic sense. 
\end{proof}

\section{Proof of the factor-closure theorem}
\label{sec:proof_main_thrm}

We now combine \cref{thrm:main_special_case} and \cref{prop:tower_of_factors_pronilsystem} to prove \cref{thrm:main}. 
Our argument generalizes the proof of the factor-closure of ergodic nilsystems, given in \citep[Chapter 13, Theorem 11]{host2018nilpotent}, to the setting of pro-nilsystems. 

\begin{proof}[Proof of \cref{thrm:main}]
    By \cref{prop:tower_of_factors_pronilsystem}, there is a tower of factors 
    \[
    X = Z_s \longrightarrow Z_{s-1} \longrightarrow \cdots \longrightarrow Z_1 \longrightarrow Z_0 = \{*\},  
    \]
    where, for $i = 1, \dots, s$, the map $Z_i \to Z_{i-1}$ is an extension by a compact abelian group $K_i$. 
    For $i = 0, \dots, s$, let $p_i \colon X \to Z_i$ denote the composition of the maps in the tower. 
    Thus $p_s$ is the identity map on $X$, and $p_0$ is the factor map onto the trivial system.  
    Define the invariant sub-$\sigma$-algebra
    \[
    \mcw_i := \pi^{-1}(\mcy) \vee p_i^{-1}(\mcz_i),
    \]
    and let $W_i$ be the corresponding factor of $X$. 
    Thus $W_i$ is the smallest factor of $X$ containing both $Y$ and $Z_i$ as factors. 
    In particular, $\mcw_s =_\mu \mcx$ and $\mcw_0 =_\mu \pi^{-1}(\mcy)$.   
    Hence $W_s$ is isomorphic to $X$ and $W_0$ is isomorphic to $Y$. 

    By the structure of the tower, the factors $Z_i$ of $X$ are increasing, so $\mcw_{i-1} \subseteq_\mu \mcw_i$ for all $i$.  
    Thus the factors $W_i$ form a tower 
    \[
    X \cong W_s \longrightarrow W_{s-1} \longrightarrow 
    \cdots \longrightarrow W_1 \longrightarrow W_0 \cong Y. 
    \]
    We prove by downward induction on $i$ that each $W_i$ is an $s$-step pro-nilsystem. 
    For $i = s$, this is true because $W_s \cong X$ is an $s$-step pro-nilsystem by assumption. 

    Let $1 \leq i \leq s$, and assume that $W_i$ is an $s$-step pro-nilsystem. 
    Note that $W_i$ is ergodic as a factor of the ergodic system $X$. 
    We claim that the factor map 
    \[
    W_i \to W_{i-1} 
    \]
    is an extension by a compact abelian group. 
    Indeed, since $p_{i-1}^{-1}(\mcz_{i-1}) \subseteq_\mu p_i^{-1}(\mcz_i)$, we obtain 
    \[
    \mcw_i = \pi^{-1}(\mcy) \vee p_i^{-1}(\mcz_i)
    =_\mu \mcw_{i-1} \vee p_i^{-1}(\mcz_i). 
    \]
    The inclusions $p_i^{-1}(\mcz_i) \subseteq_\mu \mcw_i$ and $p_{i-1}^{-1}(\mcz_{i-1}) \subseteq_\mu \mcw_{i-1}$ induce the following commutative diagram of factor maps.  
    \[\begin{tikzcd}
        {W_i} & {Z_i} \\
        {W_{i-1}} & {Z_{i-1}}
        \arrow[from=1-1, to=1-2]
        \arrow[from=1-1, to=2-1]
        \arrow[from=1-2, to=2-2]
        \arrow[from=2-1, to=2-2]
    \end{tikzcd}\]
    Since $Z_i \to Z_{i-1}$ is an extension by the compact abelian group $K_i$, \cref{lem:induced_group_ext} applies, showing that $W_i \to W_{i-1}$ is an extension by a closed subgroup of $K_i$.  
    By the induction hypothesis, $W_i$ is an $s$-step pro-nilsystem, so \cref{thrm:main_special_case} applied to the extension $W_i \to W_{i-1}$ shows that $W_{i-1}$ is an $s$-step pro-nilsystem. 

    By downward induction, $W_0$ is an $s$-step pro-nilsystem. 
    Since $W_0$ is isomorphic to $(Y, \nu, S)$, this completes the proof. 
\end{proof}

\section{Factor-closure of topological pro-nilsystems}
\label{sec:topological_factor_closure}

We finish this chapter by deriving the following topological analogue of \cref{thrm:main}. 

\begin{theorem}[Topological factor-closure]
    \label{thrm:topological_version_main}
    Let $s \geq 1$, and let $(X, T)$ be a transitive topological $s$-step pro-nilsystem. 
    Suppose that 
    \[
    \pi \colon (X, T) \to (Y, S) 
    \]
    is a topological factor map. 
    Then $(Y, S)$ is a topological $s$-step pro-nilsystem. 
\end{theorem}

This theorem is a consequence of the topological structure theorem by \citet{host2010nilsequences}. 
In their terminology, the transitive topological $s$-step pro-nilsystems are precisely the systems of order $s$, and the class of systems of order $s$ is closed under taking factors.  
Thus a factor of a transitive topological $s$-step pro-nilsystem is indeed again an $s$-step pro-nilsystem.  

We give a different proof. 
The idea is to combine \cref{thrm:main} with the correspondence between ergodic pro-nilsystems and their canonical topological models discussed in \cref{sec:pronilsystems}. 
The main additional ingredient is the criterion from \cref{lem:isomorphism_upgrade_crit} for upgrading measure-theoretic isomorphisms to topological ones. 

\begin{proof}[Proof of \cref{thrm:topological_version_main}]
    By the discussion in \cref{sec:pronilsystems}, the transitive topological $s$-step pro-nilsystem $(X, T)$ is distal and uniquely ergodic. 
    Let $\mu$ be its unique invariant probability measure, and set $\nu := \pi_* \mu$. 
    By \cref{lem:topological_model_inverse_limit}, the measure-preserving system $(X, \mu, T)$ is an ergodic $s$-step pro-nilsystem.  

    Applying \cref{thrm:main} to the factor map $\pi \colon (X, \mu, T) \to (Y, \nu, S)$, we conclude that $(Y, \nu, S)$ is an ergodic $s$-step pro-nilsystem. 
    By \cref{lem:topological_model_inverse_limit}, we may choose a uniquely ergodic topological model: there exists a uniquely ergodic, hence transitive, topological $s$-step pro-nilsystem $(\wt Y, \wt S)$ with unique invariant measure $\wt \nu$ and a measure-theoretic isomorphism
    \[
    \Phi \colon (Y, \nu, S) \to (\wt Y, \wt \nu, \wt S). 
    \]
    
    By \cref{prop:factors_between_pronilsystems_are_topological}, the composition $\Phi \circ \pi$ agrees almost everywhere with a topological factor map. 
    More precisely, there exists a topological factor map 
    \[
    \wt \pi \colon (X, T) \to (\wt Y, \wt S).  
    \]
    with $\wt \pi = \Phi \circ \pi$ $\mu$-almost everywhere.  

    Consider the closed $S \times \wt S$-invariant subset 
    \[
    Z := (\pi, \wt \pi)(X) \subseteq Y \times \wt Y.  
    \]
    Since $(X, T)$ is distal and uniquely ergodic, the factor system $(Z, S \times \wt S)$ is also distal and uniquely ergodic. 
    Let 
    \[
    \lambda := (\pi, \wt \pi)_* \mu 
    \] 
    be the unique $S \times \wt S$-invariant probability measure on $Z$.
    Let 
    \[
    p \colon Z \to Y, \quad \wt p \colon Z \to \wt Y 
    \]
    be the coordinate projections. 
    Then $p$ and $\wt p$ are both topological factor maps. 
    Moreover, since $\wt \pi = \Phi \circ \pi$ holds $\mu$-almost everywhere, we have that 
    \[
    \lambda = (\pi, \wt \pi)_*\mu = (\Id_Y, \Phi)_* \nu. 
    \]
    Thus $\lambda$ is the graph joining of the measure-theoretic isomorphism $\Phi$.
    It follows that both 
    \[
    p \colon (Z, \lambda, S \times \wt S) \to (Y, \nu, S), \quad \wt p \colon (Z, \lambda, S \times \wt S) \to (\wt Y, \wt \nu, \wt S) 
    \]
    are measure-theoretic isomorphisms.

    By \cref{lem:isomorphism_upgrade_crit}, $p$ and $\wt p$ are topological isomorphisms. 
    Therefore 
    \[
    \wt p \circ p^{-1} \colon (Y, S) \to (\wt Y, \wt S) 
    \]
    is a topological isomorphism.
    We conclude that $(Y, S)$ is a topological $s$-step pro-nilsystem. 
\end{proof}

\chapter{The weak structure theorem}
\label{chap:weak_structure_thrm}

In this chapter, as announced in the introduction, we give a detailed proof of the weak structure theorem using a form of the inverse theorem for the Gowers norms, following the argument described by Tao in \citep{tao2015weak}. 

\begin{theorem}[Weak structure theorem]
    \label{thrm:weak_structure_theorem}
    Let $s \geq 1$. 
    Every ergodic measure-preserving system of order $s$ is a factor of an ergodic $s$-step pro-nilsystem.  
\end{theorem}

We briefly recall the terminology used in this statement, the details of which can be found in \citep[Chapter 9]{host2018nilpotent} and \citep[Chapter 15]{eisner2025journey}. 
Let $(X, \mu, T)$ be a measure-preserving system and let $f \in L^\infty(X, \mu)$. 
The Gowers--Host--Kra seminorms can be defined for $s \geq 1$ by 
\[
\| f \|_{U^s(X)}^{2^s} := 
\lim_{N \to \infty} \frac{1}{N^s} \sum_{h_1, \dots, h_s \in [N]} \int_X \prod_{\omega \in \{0, 1\}^s}  
\mcc^{|\omega|} T^{\omega_1 h_1 + \dots + \omega_s h_s} f \dd\mu, 
\]
where $[N] := \{1, \dots, N\}$, $\mcc$ is the complex conjugate map, and $|\omega| = \omega_1 + \cdots + \omega_s$. 
The system is said to be of order $s$ if the seminorm $\| \cdot \|_{U^{s+1}(X)}$ is a norm, that is, if $\| f \|_{U^{s + 1}(X)} > 0$ for every non-zero $f \in L^\infty(X, \mu)$. 

We now begin by introducing the dynamical dual functions and formulating the order $s$ condition in terms of them.  
We then combine this formulation with the results from the inverse theory for the Gowers norms collected in \cref{chap:appendix} to prove \cref{thrm:weak_structure_theorem}. 
We conclude the chapter by explaining how \cref{thrm:weak_structure_theorem} combines with \cref{thrm:main} to recover the Host--Kra structure theorem. 

\section{Dynamical dual functions}

The dynamical dual functions form the bridge between the order $s$ condition and the inverse theorem for the Gowers norms. 
The content of this section is based on \citep[Chapters 8, 9]{host2018nilpotent}. 

Let $(X, \mu, T)$ be a measure-preserving system, let $s \geq 1$, and let $f \in L^\infty(X, \mu)$. 
For $N \geq 1$, define 
\[
\mcd_{s,N} f := \frac{1}{N^s} 
\sum_{h_1, \dots, h_s \in [N]} \prod_{\omega \in \{0, 1\}^s \setminus \{0\}} \mcc^{|\omega|} T^{\omega_1 h_1 + \dots + \omega_s h_s} f. 
\]
By \citep[Chapter 8, Theorem 28]{host2018nilpotent}, the functions $\mcd_{s,N} f$ converge in $L^2(X, \mu)$ as $N \to \infty$. 
Importantly, the proof of this fact does not depend on the Host--Kra structure theorem. 
We denote the limit by 
\[
\mcd_s f := \lim_{N \to \infty} \mcd_{s,N} f, 
\]
and call $\mcd_s f$ the dynamical dual function associated to $f$.
Since $\| \mcd_{s,N} f \|_{L^\infty(\mu)} \leq \| f \|_{L^\infty(\mu)}^{2^s - 1}$ for each $N \geq 1$, we have that 
\[
\| \mcd_s f \|_{L^\infty(\mu)} \leq \| f \|_{L^\infty(\mu)}^{2^s - 1}.  
\]
The dual functions are related to the Gowers--Host--Kra seminorms by the identity
\[
\int_X f \cdot \mcd_s f \dd\mu = \| f \|_{U^s(X)}^{2^s},  
\]
for every $f \in L^\infty(X, \mu)$. 

The important point for us is that if $(X, \mu, T)$ is of order $s$, then the dynamical dual functions $\mcd_{s+1} f$ for $f \in L^\infty(X, \mu)$ generate the system in the sense of the following proposition; see \citep[Chapter 9, Proposition 17]{host2018nilpotent} for a proof. 

\begin{proposition}
    \label[proposition]{prop:dual_functions_generate}
    Let $s \geq 1$, and let $(X, \mu, T)$ be a measure-preserving system of order $s$. 
    Then 
    \[
   \Span \{\mcd_{s+1} f : f \in L^\infty(X, \mu)\}
    \]
    is dense in $L^2(X, \mu)$. 
\end{proposition}

The finite approximations to the dynamical dual functions have the following anti-uniformity property; see \cref{chap:appendix} for the notions of nilsequences, complexity and the Gowers norms $\| \cdot \|_{U^{s+1}[N]}$ used below. 
Throughout the rest of this chapter, we write $A \lesssim_{\alpha_1, \dots \alpha_k} B$ to mean that $A \leq C B$ for a constant $C = C(\alpha_1, \dots, \alpha_k) > 0$ depending only on $\alpha_1, \dots, \alpha_k$. 
If no subscript is written, the implicit constant is absolute. 

\begin{lemma}
    \label[lemma]{lem:anti_uniformity_dual_functions}
    Let $s \geq 1$, and let $(X, \mu, T)$ be a measure-preserving system. 
    Suppose that $f \in L^\infty(X, \mu)$ satisfies $\| f \|_{L^\infty(\mu)} \leq 1$. 
    Then for $\mu$-almost every $x \in X$, every $N \geq 1$, and every $1$-bounded function $g \colon [N] \to \C$,
    \[
    | \E_{h \in [N]} \mcd_{s+1,N} f (T^h x) \overline{g(h)} | \lesssim_s \| g \|_{U^{s+1}[N]}. 
    \]
\end{lemma}

\begin{proof}
    Since $\| f \|_{L^\infty(\mu)} \leq 1$, there is a set of full measure on which $|f(T^h x)| \leq 1$ for every $h \in \Z$. 
    Fix any $x$ in this set. 
    Let $N \geq 1$ and let $g \colon [N] \to \C$ be $1$-bounded. 
    Expanding the definition of $\mcd_{s+1, N} f$, the expression we want to estimate is 
    \[
    \left| \frac{1}{N^{s + 2}} \sum_{h, h_1, \dots, h_{s+1} \in [N]} \overline{g(h)} \prod_{\omega \in \{0, 1\}^{s+1} \setminus \{0\}} \mcc^{|\omega|} f(T^{h + \omega_1 h_1 + \cdots + \omega_{s+1} h_{s+1}} x)\right|. 
    \]

    Let $\wt N := 2^{s+1} N$, and let $\wt g \colon \Z/\wt N \Z \to \C$ be the extension of $g$ by $0$. 
    Then the above expression is bounded, up to a constant depending only on $s$, by 
    \[
    \left| \frac{1}{\wt N^{s+2}} \sum_{h, h_1, \dots, h_{s+1} \in \Z/\wt N \Z} \overline{\wt g(h)} \prod_{\omega \in \{0, 1\}^{s+1} \setminus \{0\}} F_\omega(h + \omega_1 h_1 + \dots + \omega_{s+1} h_{s+1})\right|, 
    \]
    for some appropriate $1$-bounded functions $F_\omega \colon \Z/ \wt N \Z \to \C$. 
    By the usual Gowers--Cauchy--Schwarz inequality, see for example \citep[Appendix B.12]{green2010linear}, this is bounded by 
    \[
    \| \wt g \|_{U^{s+1}(\Z/\wt N \Z)}.  
    \]
    But by the definition of the Gowers norm on $[N]$, $\| \wt g \|_{U^{s+1}(\Z/\wt N \Z)} \leq \|g \|_{U^{s+1}[N]}$.  
    Hence, we conclude that 
    \[
    | \E_{h \in [N]} \mcd_{s+1,N} f (T^h x) \overline{g(h)} | \lesssim_s \| g \|_{U^{s+1}[N]}. 
    \]
\end{proof}

Combining this lemma with \cref{lem:antiuniform_approx_nilsequence}, we immediately obtain the following corollary. 

\begin{corollary}
    \label[corollary]{cor:dual_approx_nilsequence} 
    Let $s \geq 1$ and $\epsilon > 0$, let $(X, \mu, T)$ be a measure-preserving system, and let $f \in L^\infty(X, \mu)$ with $\| f \|_{L^\infty(\mu)} \leq 1$. 
    Then there exists $M = M(s, \epsilon)$, such that for $\mu$-almost every $x \in X$ and every $N \geq 1$, there exists a $1$-bounded nilsequence $\psi$ of degree $s$ and complexity $M$ satisfying 
    \[
    \E_{h \in [N]} |\mcd_{s+1, N} f (T^h x) - \psi(h)|^2 \leq \epsilon.
    \]
\end{corollary}

\section{Proof of the weak structure theorem}

We are now ready to prove \cref{thrm:weak_structure_theorem}. 

\begin{proof}[Proof of \cref{thrm:weak_structure_theorem}]
    Let $(X, \mu, T)$ be an ergodic measure-preserving system of order $s$.
    Since $s$ remains fixed throughout the proof, for $f \in L^\infty(X, \mu)$ and $N \geq 1$, we write  
    \[
    \mcd_N f := \mcd_{s + 1, N} f , \quad \mcd f  := \mcd_{s+1} f
    \]
    for the dynamical dual functions from the previous section. 

    Since $(X, \mu)$ is a standard probability space by assumption, $L^2(X, \mu)$ is separable. 
    Combining this with \cref{prop:dual_functions_generate}, we may choose a sequence $(f_n)_{n \geq 1}$ in $L^\infty(X,\mu)$ such that 
    \[
    \Span\{\mcd f_n : n \geq 1\}
    \]
    is dense in $L^2(X,\mu)$.
     In particular,
    \[
    \Span\{T^h \mcd f_n : n \geq 1,\ h \in \Z\}
    \]
    is dense in $L^2(X,\mu)$.
    Since rescaling a function only results in a rescaling of the dynamical dual function, we may assume without loss of generality that $\| f_n \|_{L^\infty(\mu)} \leq 1$ for each $n \geq 1$. 
    Therefore, by \cref{prop:polynomial_factor_criterion}, it suffices to construct an ergodic $s$-step pro-nilsystem $(Y,\nu,S)$ and a sequence of functions $\wt f_n \in L^\infty(Y,\nu)$ such that for every $d \geq 1$, all $n_1,\dots,n_d \geq 1$, all $h_1,\dots,h_d \in \Z$, and every $P \in \Q(i)[z_1,\dots,z_d,\overline{z_1},\dots,\overline{z_d}]$, 
    \begin{equation}
    \begin{aligned}
        \label{eq:weak_structure_moment_identity}
        &\int_Y P(S^{h_1}\wt f_{n_1}, \dots, S^{h_d}\wt f_{n_d},
        \overline{S^{h_1}\wt f_{n_1}}, \dots, \overline{S^{h_d}\wt f_{n_d}}) \dd\nu \\
        &\qquad =
        \int_X P(T^{h_1} \mcd f_{n_1}, \dots, T^{h_d} \mcd f_{n_d},
        \overline{T^{h_1} \mcd f_{n_1}}, \dots, \overline{T^{h_d} \mcd f_{n_d}}) \dd\mu.
    \end{aligned}
    \end{equation}

    The idea is to find a point $x \in X$ and a family of nilsequences, which approximate the orbit of each dual function $\mcd f_n$ along this point.
    Choose an increasing sequence $(N_m)_{m \geq 1}$ with $N_m \geq m$ such that 
    \[
    \|\mcd f_n - \mcd_{N_m} f_n\|_{L^2(\mu)}^2 \leq 2^{-100(n+m)} 
    \]
    for all $m \geq n \geq 1$. 
    Then the function 
    \[
    \sum_{m=1}^\infty \sum_{n=1}^m 2^{50(n+m)} |\mcd f_n - \mcd_{N_m} f_n |^2
    \]
    belongs to $L^1(X,\mu)$.
    By the maximal ergodic theorem, see for example \citep[Corollary 7.47]{eisner2025journey}, there exists a set of full measure $\Omega_0 \subseteq X$, such that for every $x \in \Omega_0$ there is a constant $C_x > 0$ satisfying   
    \begin{equation}
        \label{eq:weak_structure_finite_dual_error}
        \E_{h \in [H]} |\mcd f_n (T^h x) - \mcd_{N_m}f_n (T^h x)|^2 \leq C_x 2^{-50(n+m)}
    \end{equation}
    for every $H \geq 1$ and all $m \geq n \geq 1$.

    Next, fix $m \geq n \geq 1$ and $k \geq 1$.
    Let $M_{n,k}$ be the complexity bound obtained by applying \cref{cor:dual_approx_nilsequence} with $s$ and $\epsilon = 2^{-100(n+k)}$.
    Define $B_{n,m,k}$ to be the set of points $y \in X$ for which there does not exist a $1$-bounded nilsequence $\psi$ of degree $s$ and complexity $M_{n,k}$ such that
    \begin{equation}
        \label{eq:weak_structure_good_short_intervals}
        \sup_{1 \leq H \leq 2^{-10k}N_m}
        \E_{h \in [H]} |\mcd_{N_m} f_n(T^h y) - \psi(h)|^2
        \leq 2^{-20(n+k)}.
    \end{equation}
    If $2^{-10k}N_m < 1$, the supremum is taken over the empty set, and we let $B_{n,m,k} = \emptyset$.
    In either case, $B_{n,m,k}$ is measurable. 
    To see this, let $L := \lfloor 2^{-10k}N_m \rfloor$, and let $\mathcal{N} \subseteq \C^L$ be the set consisting of points 
    \[
    (\psi(1), \dots, \psi(L))
    \]
    where $\psi$ ranges over all $1$-bounded nilsequences of degree $s$ and complexity $M_{n,k}$. 
    This set is compact.
    Indeed, by the complexity bound we only need to consider nilsequences on a finite collection of filtered nilmanifolds. 
    By taking the Cartesian product, we can therefore restrict to a single fixed nilmanifold, where we can apply Arzelà--Ascoli and \cref{lem:compactness_taylor_coefficients} to prove that $\mathcal{N}$ is closed. 
    Since it is also bounded by $1$-boundedness of the nilsequences, it is compact.
    Therefore, the set of $z = (z_1, \dots, z_L) \in \C^L$ for which there exists some $w \in \mathcal{N}$ such that 
    \[
    \max_{1 \leq H \leq L} \E_{h \in [H]} |z_h - w_h|^2 \leq 2^{-20(n+k)} 
    \]
    is closed. 
    Its preimage under the measurable map 
    \[
    y \mapsto (\mcd_{N_m} f_n(T y), \dots, \mcd_{N_m} f_n(T^L y)) 
    \]
    is precisely $X \setminus B_{n,m,k}$, so that $B_{n,m,k}$ is measurable. 
        
    We claim that
    \begin{equation}
        \label{eq:weak_structure_bad_set_bound}
        \mu(B_{n,m,k}) \lesssim 2^{-50(n+k)}.
    \end{equation}
    To see this, we may assume that $B_{n,m,k} \neq \emptyset$, so that $N_m \geq 2^{10k} \geq 2^{10}$. 
    By \cref{cor:dual_approx_nilsequence}, for $\mu$-almost every $y \in X$ there exists a $1$-bounded nilsequence $\psi_y$ of degree $s$ and complexity $M_{n,k}$ such that
    \[
    \E_{h \in [N_m]} |\mcd_{N_m} f_n(T^h y) - \psi_y(h)|^2 \leq 2^{-100(n+k)}.
    \]
    For any such $y \in X$, define
    \[
    a_y(h) :=
    \begin{cases}
    |\mcd_{N_m} f_n(T^h y) - \psi_y(h)|^2 & \text{if } h \in [N_m],\\
    0 & \text{if } h \in \Z \setminus [N_m].
    \end{cases}
    \]
    Set $R := \lfloor (1-2^{-10k}) N_m \rfloor$, and note that $R \geq N_m/2$.  
    Now if $r$ is an integer with $1 \leq r \leq R$ such that 
    \[
    \sup_{1 \leq H \leq 2^{-10k}N_m}  \E_{h \in [H]} a_y(h +r) \leq 2^{-20(n + k)}, 
    \]
    then $T^r y \not\in B_{n,m,k}$ since the shifted nilsequence $h \mapsto \psi_y(h + r)$ in this case satisfies \eqref{eq:weak_structure_good_short_intervals}. 
    By the discrete Hardy--Littlewood maximal inequality, the number of integers $1 \leq r \leq R$ for which
    \[
    \sup_{1 \leq H \leq 2^{-10k}N_m} \E_{h \in [H]} a_y(h+r) > 2^{-20(n+k)}
    \]
    is bounded by $2^{-70(n+k)} N_m$.
    Hence, using the $T$-invariance of $\mu$, 
    \begin{align*}
        \mu(B_{n,m,k})
        &= \E_{1 \leq r \leq R} \mu(T^{-r} B_{n,m,k}) \\
        &= \int_X \E_{1 \leq r \leq R} \ind{B_{n,m,k}}(T^r y) \dd \mu(y) \\
        &= \frac{1}{R} \int_X |\{1 \leq r \leq R : T^r y \in B_{n,m,k} \}| \dd \mu(y)
        \lesssim 2^{-50(n + k)}, 
    \end{align*}
    which proves \eqref{eq:weak_structure_bad_set_bound}. 

    Combining \eqref{eq:weak_structure_bad_set_bound} with Fatou's lemma, it follows that 
    \[
    \liminf_{M \to \infty} \frac{1}{M} \sum_{m=1}^M \sum_{n=1}^m \sum_{k=1}^{\infty} 2^{10(n + k)} \ind{B_{n,m,k}} \in L^1(X, \mu). 
    \]
    Therefore, the set $\Omega_1 \subseteq X$ where this function is finite has full measure. 
    If $x \in \Omega_1$, then there exists $C'_x > 0$ and a subsequence $(m_j)_{j \geq 1}$ depending on $x$ such that for all $j \geq 1$, 
    \[
    \sum_{n=1}^{m_j} \sum_{k=1}^\infty 2^{10(n + k)} \ind{B_{n,m_j, k}}(x) \leq C'_x.  
    \]
    Pick $K_x \geq 1$ such that $2^{10 K_x} > C'_x$. 
    Then for every $j \geq 1$, all $m_j \geq n \geq 1$, and each $k \geq K_x$, 
    \[
    x \not\in B_{n,m_j,k}. 
    \]
    Finally, by Birkhoff's pointwise ergodic theorem, and since there are only countably many choices $d,n_1,\dots,n_d \geq 1,h_1,\dots,h_d \in \Z$ and $P \in \Q(i)[z_1, \dots, z_d, \overline{z_1}, \dots, \overline{z_d}]$, we may fix   
    \[
    x \in \Omega_0 \cap \Omega_1
    \]
    which is generic\footnote{We say that $x \in X$ is generic for $f \in L^1(X, \mu)$ if $\lim_{N \to \infty} \E_{h \in [N]} f(T^h x) = \int_X f \dd\mu$.} for all functions
    \[
    P(T^{h_1} \mcd f_{n_1}, \dots, T^{h_d} \mcd f_{n_d},
    \overline{T^{h_1} \mcd f_{n_1}}, \dots, \overline{T^{h_d} \mcd f_{n_d}}). 
    \]
    We may further assume that $x$ lies in the subset of full measure on which all the functions $\mcd f_n$, $\mcd_N f_n$, and their shifts are bounded in magnitude by $1$.

    Fix the subsequence $(m_j)_{j \geq 1}$ and the integer $K := K_x$ associated to this point $x$.
    For each $n \geq 1$, each $k \geq K$, and all sufficiently large $j$, we know that $x \not\in B_{n,m_j,k}$, hence we can fix a nilsequence $\psi_{n,j,k}$ of degree $s$ and complexity $M_{n,k}$ satisfying \eqref{eq:weak_structure_good_short_intervals} with $m = m_j$ and $y = x$.
    Combining this with \eqref{eq:weak_structure_finite_dual_error} and using that $N_{m_j} \to \infty$, we see that for every fixed $H \geq 1$,
    \begin{equation}
        \label{eq:weak_structure_bound_psi_j}
        \limsup_{j \to \infty} \
        \E_{h \in [H]} |\mcd f_n(T^h x) - \psi_{n,j,k}(h)|^2
    \lesssim 2^{-20(n+k)}.
    \end{equation}

    For each $n \geq 1$, $k \geq K$, and $j$ sufficiently large, write 
    \[
    \psi_{n,j,k}(h) = F_{n,j,k}(g_{n,j,k}(h) \Gamma_{n,j,k}), \quad \text{for } h \in \Z
    \]
    where $G_{n,j,k}/\Gamma_{n,j,k}$ is a filtered nilmanifold of degree $s$ and complexity $M_{n,k}$, $g_{n,j,k} \in \Poly(\Z, (G_{n,j,k})_\bullet)$ is a polynomial sequence, and $F_{n,j,k} \colon G_{n,j,k}/\Gamma_{n,j,k} \to \C$ is a $1$-bounded Lipschitz function with Lipschitz norm at most $M_{n,k}$. 

    Fix $n \geq 1$ and $k \geq K$. 
    Since $G_{n,j,k}/\Gamma_{n,j,k}$ are of uniformly bounded complexity $M_{n,k}$ for all $j$, we may assume that these filtered nilmanifolds belong to a fixed finite collection depending only on $n$ and $k$. 
    By the pigeonhole principle, there exists a filtered nilmanifold $G_{n,k}/\Gamma_{n,k}$ of degree $s$ such that, after passing to a subsequence of $(m_j)_{j \geq 1}$, 
    \[
    G_{n,j,k}/\Gamma_{n,j,k} = G_{n,k}/\Gamma_{n,k} \quad \text{for all } j. 
    \]
    Moreover, since the Lipschitz norms of the functions $F_{n,j,k}$ are uniformly bounded by $M_{n,k}$, the Arzelà--Ascoli theorem implies that after passing to a further subsequence of $(m_j)_{j \geq 1}$, there is a continuous function $F_{n,k} \colon G_{n,k}/\Gamma_{n,k} \to \C$ such that 
    \[
    F_{n,j,k} \to F_{n,k}
    \quad \text{uniformly as } j \to \infty.
    \] 
    Finally, by \cref{lem:compactness_taylor_coefficients}, we may replace each $g_{n,j,k}$ by a polynomial sequence that agrees with it modulo $\Gamma_{n,k}$, thus leaving $\psi_{n,j,k}$ unchanged, and pass to yet a further subsequence of $(m_j)_{j \geq 1}$, so that 
    \[
    g_{n,j,k} \to g_{n,k} 
    \quad \text{pointwise as } j \to \infty 
    \]
    for some polynomial sequence $g_{n,k} \in \Poly(\Z, (G_{n,k})_\bullet)$. 

    We can carry out this same construction for all choices of $n \geq 1$ and $k \geq K$.  
    By a diagonalization argument over the countably many pairs $(n, k)$, we obtain a single subsequence of $(m_j)_{j \geq 1}$ such that for all $n \geq 1$ and all $k \geq K$, 
    \[
    \psi_{n,j,k}(h) = F_{n,j,k}(g_{n,j,k}(h) \Gamma_{n,j,k}) \to F_{n,k}(g_{n,k}(h) \Gamma_{n,k}) =: \psi_{n,k}(h)
    \]
    for all $h \in \Z$.
    Then it follows by \eqref{eq:weak_structure_bound_psi_j} that for every $n \geq 1$, every $k \geq K$, and every $H \geq 1$,
    \begin{equation}
        \label{eq:weak_structure_orbit_approx}
        \E_{h \in [H]} |\mcd f_n(T^h x) - \psi_{n,k}(h)|^2
        \lesssim 2^{-20(n+k)}.
    \end{equation}

    By \cref{lem:linearising_nilsequence}, for each $n \geq 1$ and $k \geq K$, there exists an $s$-step nilsystem $(Y_{n,k}, S_{n,k})$, a point $y_{n,k} \in Y_{n,k}$, and a continuous function $\wt F_{n,k} \colon Y_{n,k} \to \C$ such that
    \[
    \psi_{n,k}(h) = \wt F_{n,k}(S_{n,k}^h y_{n,k}) \quad \text{for } h \in \Z. 
    \]
    By replacing $Y_{n,k}$ by the orbit closure of $y_{n,k}$, we may assume by \cref{thrm:orbit_closures_nilsystems_are_nilsystems} that $(Y_{n,k},S_{n,k})$ is uniquely ergodic.  
    Since $\psi_{n,k}$ is $1$-bounded and the orbit of $y_{n,k}$ is dense in this replacement, $\wt F_{n,k}$ is $1$-bounded on $Y_{n,k}$.

    We now combine these nilsystems into an inverse limit.
    For $l \geq 1$, let
    \[
    Z_l := \prod_{1 \leq n \leq l} \prod_{K \leq k \leq K + l} Y_{n,k}, 
    \]
    equipped with the product translation $R_l$, and let 
    \[
    z_l := (y_{n,k})_{1 \leq n \leq l, K \leq k \leq K + l}. 
    \]
    Then $(Z_l, R_l)$ is an $s$-step nilsystem, and after replacing $Z_l$ by the orbit closure of $z_l$, we may assume by \cref{thrm:orbit_closures_nilsystems_are_nilsystems} that it is uniquely ergodic.  
    The coordinate projections 
    \[
    Z_{l+1} \to Z_l
    \]
    are factor maps which respect restricting to the orbits, so the systems $(Z_l, R_l)$ form an inverse system. 
    Take the topological dynamical inverse limit
    \[
    (Y,S) := \varprojlim (Z_l,R_l),
    \]
    and let $\nu$ be the unique $S$-invariant probability measure on $Y$.
    By \cref{lem:topological_model_inverse_limit}, $(Y,\nu,S)$ is an ergodic $s$-step pro-nilsystem.

    For each $n \geq 1$ and $k \geq K$, let $q_{n,k} \colon Y \to Y_{n,k}$ be the coordinate factor map induced by the inverse limit.
    Let $y_0 \in Y$ be the point determined by $q_{n,k}(y_0) = y_{n,k}$ for all $n$ and $k$. 
    Define 
    \[
    \wt f_{n,k} := \wt F_{n,k} \circ q_{n,k}.
    \]
    Then $\wt f_{n,k}$ is continuous and $1$-bounded, and
    \[
    \wt f_{n,k}(S^h y_0) 
    = \wt F_{n,k}(S_{n,k}^h q_{n,k}(y_0))
    = \wt F_{n,k}(S_{n,k}^h y_{n,k})
    = \psi_{n,k}(h)
    \]
    for every $h \in \Z$.
    Combining this with \eqref{eq:weak_structure_orbit_approx}, we get for all $H \geq 1$ that 
    \[
    \E_{h \in [H]} |\mcd f_n(T^h x) - \wt f_{n,k}(S^h y_0)|^2
    \lesssim 2^{-20(n+k)}. 
    \]
    Hence, by the $1$-boundedness of the functions in the absolute value, we obtain for every $r \in \Z$ that  
    \begin{equation}
        \label{eq:weak_structure_model_orbit_approx}
        \limsup_{H \to \infty} \
        \E_{h \in [H]} |\mcd f_n(T^{h + r} x) - \wt f_{n,k}(S^{h + r} y_0)|^2
        \lesssim 2^{-20(n+k)}. 
    \end{equation}

    Since $(Y,S)$ is uniquely ergodic, \eqref{eq:weak_structure_model_orbit_approx} implies that for all $k' \geq k \geq K$,
    \[
    \int_Y |\wt f_{n,k'} - \wt f_{n,k}|^2 \dd\nu
    =
    \lim_{H \to \infty}
    \E_{h \in [H]} |\wt f_{n,k'}(S^h y_0) - \wt f_{n,k}(S^h y_0)|^2
    \lesssim 2^{-20(n+k)}.
    \]
    Thus, for each $n \geq 1$, the sequence $(\wt f_{n,k})_{k \geq K}$ is Cauchy in $L^2(Y,\nu)$.
    Let $\wt f_n$ be its $L^2$-limit.
    Since the functions $\wt f_{n,k}$ are all $1$-bounded, we have that $\|\wt f_n\|_{L^\infty(\nu)} \leq 1$.

    It remains to verify \eqref{eq:weak_structure_moment_identity}.
    Fix $d \geq 1$, $n_1,\dots,n_d \geq 1$, $h_1,\dots,h_d \in \Z$, and $P \in \Q(i)[z_1,\dots,z_d,\overline{z_1},\dots,\overline{z_d}]$. 
    Since $x$ was chosen to be  generic for the corresponding function on $X$,
    \begin{align*}
    \int_X P(T^{h_1} \mcd f_{n_1}, \dots, \overline{T^{h_d} \mcd f_{n_d}}) \dd\mu 
    =
    \lim_{H \to \infty}
    \E_{h \in [H]}
    P(\mcd f_{n_1}(T^{h+h_1}x), \dots, \overline{\mcd f_{n_d}(T^{h+h_d}x)}).
    \end{align*}
    On the other hand, by unique ergodicity of $(Y,S)$, for every $k \geq K$,
    \begin{align*}
    \int_Y P(S^{h_1}\wt f_{n_1,k}, \dots, \overline{S^{h_d}\wt f_{n_d,k}}) \dd\nu 
    =
    \lim_{H \to \infty}
    \E_{h \in [H]}
    P(\wt f_{n_1,k}(S^{h+h_1}y_0), \dots, \overline{\wt f_{n_d,k}(S^{h+h_d}y_0)}).
    \end{align*}
    The polynomial $P$, when viewed as a function of $d$ variables, is Lipschitz on $\overline{\D}^d$. 
    Therefore, the squared difference between the above two quantities is bounded up to a constant depending only on $P$ by 
    \[
    \limsup_{H \to \infty} \ \E_{h \in [H]} \sum_{i=1}^{d} |\mcd f_{n_i}(T^{h + h_i} x) - \wt f_{n_i,k}(S^{h+h_i} y_0)|^2 \lesssim_{d, P} 2^{-20 k},  
    \]
    where the last bound follows from \eqref{eq:weak_structure_model_orbit_approx}.
    Letting $k \to \infty$, and using the convergence $\wt f_{n_i,k } \to \wt f_{n_i}$ in $L^2(Y, \nu)$, we obtain precisely \eqref{eq:weak_structure_moment_identity}. 
    This completes the proof. 
\end{proof}

\section{Recovering the Host--Kra structure theorem}

We end this chapter by explaining how the weak structure theorem combines with the factor-closure theorem from \cref{chap:factor_closure} to recover the Host--Kra structure theorem, originally proved by \citet{host2005nonconventional}. 

\begin{theorem}[Host--Kra structure theorem]
    \label{thrm:host_kra_structure_theorem}
    Let $s \geq 1$. 
    An ergodic measure-preserving system is of order $s$ if and only if it is isomorphic to an $s$-step pro-nilsystem. 
\end{theorem}

The two implications in this theorem are of a rather different nature. 
The first implication says that every ergodic $s$-step pro-nilsystem is of order $s$. 
This direction is still non-trivial, but it is direct in the sense that one starts from a concrete class of systems and verifies the corresponding seminorm property. 
More precisely, one first proves that ergodic $s$-step nilsystems are of order $s$, using an algebraic construction of the cubic structures on nilmanifolds, and then passes to inverse limits; see \citep[Chapter 12, Corollary 19]{host2018nilpotent}. 

The converse implication is the structural direction. 
It says that if an ergodic system is of order $s$, then this abstract condition, which is formulated in terms of the Gowers--Host--Kra seminorms, forces the system to belong to the concrete class of $s$-step pro-nilsystems. 
This is the difficult part of the Host--Kra structure theorem, and it is precisely this direction that we can now deduce.  

\begin{proof}[Proof of  structural direction of \cref{thrm:host_kra_structure_theorem}]
    Let $(X, \mu, T)$ be an ergodic system of order $s$. 
    By \cref{thrm:weak_structure_theorem}, there exists an ergodic $s$-step pro-nilsystem $(Y, \nu, S)$ such that $(X, \mu, T)$ is a factor of $(Y, \nu, S)$. 
    \cref{thrm:main} then implies that $(X, \mu, T)$ is isomorphic to an $s$-step pro-nilsystem. 
\end{proof}

Note that our proof of \cref{thrm:weak_structure_theorem} ultimately rests on the inverse theorem for the Gowers norms on finite intervals, as stated in \cref{thrm:inverse_theorem}. 
The previous argument therefore constitutes a deduction of the structural direction of the Host--Kra structure theorem from the inverse theorem for the Gowers norms. 

This route to prove the structure theorem was originally outlined by \citet{tao2015weak}. 
Tao observed that the Host--Kra structure theorem splits into two statements: the weak structure theorem, realizing every ergodic system of order $s$ as a factor of an ergodic $s$-step pro-nilsystem, and the factor-closure theorem for ergodic $s$-step pro-nilsystems. 
He then described how the first statement follows from the inverse theorem for the Gowers norms, by the argument worked out in the previous section. 
What remained missing in this plan was an independent proof of the factor-closure theorem. 
Our proof of \cref{thrm:main} provides precisely this missing ingredient. 
The results of this thesis thus complete Tao's proposed route from the inverse theorem for the Gowers norms to the Host--Kra structure theorem. 

\appendix
\chapter{Inverse theory for the Gowers norms}
\label[appendix]{chap:appendix}

In this appendix, we collect the input from the inverse theory for the Gowers norms that is needed in the proof of the weak structure theorem. 
The goal is not to give a self-contained account of the full theory, but rather to state the necessary definitions and results that are used in \cref{chap:weak_structure_thrm}. 
The material presented here is based on \citep{green2010arithmetic, green2007quantitative, green2012inverse, green2011inverse}.  

\section{Filtered nilmanifolds}

The inverse theorem for the Gowers norms is formulated in terms of a quantitative refinement of the notion of a nilmanifold introduced in \cref{sec:nilmanifolds_definition}, called a filtered nilmanifold. 

\begin{definition}[Filtered nilmanifold]
    Let $s \geq 1$. 
    A filtered nilmanifold of degree $s$ is a triple 
    \[
    (X = G/\Gamma, G_\bullet, \mathcal{M}),  
    \]
    where $X = G/\Gamma$ is a nilmanifold with $G$ connected and simply connected, $G_\bullet = (G\si)_{i \geq 0}$ is a sequence of connected rational subgroups 
    \[
    G = G^{(0)} = G^{(1)} \geq G^{(2)} \geq \cdots  
    \] 
    such that $[G\si, G^{(j)}] \subseteq G^{(i+j)}$ for all $i, j \geq 0$, $G^{(s+1)} = \{e_G\}$, and $\mathcal{M}$ is a Mal'cev basis adapted to the filtration $G_\bullet$.  
\end{definition}

We refer to \citep[Section 2]{green2007quantitative} for the precise definition of an adapted Mal'cev basis. 
It is only included in the definition to fix the coordinate system used to quantify the complexity of a filtered nilmanifold. 
The precise definition of the complexity of a filtered nilmanifold can be found in \citep[Definition 1.4]{green2010arithmetic}. 
Roughly speaking, a filtered nilmanifold is of complexity $M \geq 1$ if the dimension of $X$ and the degree of $G_\bullet$ are bounded by $M$, and the algebraic data describing $X$ and $G_\bullet$ in the Mal'cev coordinates provided by $\mathcal{M}$ is in a suitable technical sense bounded in size by $M$. 
Here and below, the term complexity $M$ is used in the non-minimal sense, namely, for what Green and Tao call complexity at most $M$. 

The only feature of this definition that we will explicitly need is the following finiteness property: 
for fixed $M \geq 1$, there are only finitely many filtered nilmanifolds of complexity $M$, up to the natural algebraic notion of isomorphism. 
In arguments where the complexity is bounded by $M$, we may therefore regard the filtered nilmanifolds as ranging over a finite collection depending only on $M$. 

We usually omit the filtration $G_\bullet$ and the Mal'cev basis $\mathcal{M}$ from the notation and refer to a filtered nilmanifold simply as $X = G/\Gamma$. 

\section{Polynomial sequences}

If $G$ is a group and $g \colon \Z \to G$ is a map, then for $h' \in \Z$, we denote the discrete derivative of $g$ in the direction of $h'$ by $\partial_{h'} g(h) := g(h + h') g(h)^{-1}$ for $h \in \Z$.  

\begin{definition}[Polynomial sequence]
    Let $X = G/\Gamma$ be a filtered nilmanifold with filtration $G_\bullet = (G\si)_{i \geq 0}$. 
    A polynomial sequence adapted to this filtered nilmanifold is a map $g \colon \Z \to G$ such that 
    \[
    \partial_{h_1} \dots \partial_{h_i} g (h) \in G\si
    \] 
    for all $i \geq 0$ and $h, h_1, \dots, h_i \in \Z$. 
    The set of all such polynomial sequences is denoted by $\Poly(\Z, G_\bullet)$. 
\end{definition}

The following lemma is an analogue of the Taylor expansion for classical polynomials.
A proof can be found in \citep[Lemma A.1]{green2010arithmetic}. 
Here we use the generalized binomial coefficients 
\[
\binom{h}{i} := \frac{h (h-1) \dots (h - i + 1)}{i!}
\]
for all $h \in \Z$ and $i \geq 0$. 

\begin{lemma}[Taylor expansion]
    \label[lemma]{lem:taylor_expansion}
    Let $X = G/\Gamma$ be a filtered nilmanifold of degree $s$ with filtration $G_\bullet = (G\si)_{i \geq 0}$, and let $g \in \Poly(\Z, G_\bullet)$.  
    Then there are unique Taylor coefficients $g_i \in G\si$ for $i = 0, \dots, s$ such that 
    \[
    g(h) = g_0 g_1^{\binom{h}{1}} \cdots g_s^{\binom{h}{s}}  
    \]
    for all $h \in \Z$. 
    Conversely, every expansion of this form defines a polynomial sequence. 
\end{lemma}

We can use the Taylor expansion to prove the following compactness result, which, for a fixed filtered nilmanifold $G/\Gamma$, allows us to extract pointwise limits from sequences in $\Poly(\Z, G_\bullet)$ after adjusting the elements of the sequence only modulo $\Gamma$.  

\begin{lemma}
    \label[lemma]{lem:compactness_taylor_coefficients}
    Let $X = G/\Gamma$ be a filtered nilmanifold of degree $s$ with filtration $G_\bullet = (G^{(i)})_{i \geq 0}$ and let $(g_j)_{j \geq 1}$ be a sequence in $\Poly(\Z, G_\bullet)$. 
    Then there exists a sequence $(\wt g_j)_{j \geq 1}$ in $\Poly(\Z, G_\bullet)$ such that 
    \[
    \wt g_j(h) \Gamma = g_j(h) \Gamma
    \]
    for each $h \in \Z$ and $j \geq 1$, and such that after passing to a subsequence, $(\wt g_j)_{j \geq 1}$ converges pointwise on $\Z$ to some polynomial sequence $g \in \Poly(\Z, G_\bullet)$. 
\end{lemma}

\begin{proof}
    For $i = 0, \dots, s$, let $\Gamma^{(i)} := \Gamma \cap G^{(i)}$. 
    Since $G^{(i)}$ is a rational subgroup of $G$ by definition of a filtered nilmanifold, $\Gamma^{(i)}$ is cocompact in $G^{(i)}$; see \cref{lem:rational_subgroups_equivalent_chatacterizations}.
    Hence there exists a compact subset $K\si \subseteq G\si$ such that 
    \[
    G\si = K\si\Gamma\si.  
    \]

    Fix $j \geq 1$. 
    Applying \cref{lem:taylor_expansion}, we obtain a unique Taylor expansion
    \[
    g_j(h) = a_{j,0} a_{j,1}^{\binom{h}{1}} \cdots a_{j,s}^{\binom{h}{s}} \quad \text{for } h \in \Z,  
    \]
    where $a_{j,i} \in G\si$ for $i = 0, \dots, s$. 
    We will adjust the Taylor coefficients one by one, to lie in the compact sets $K\si$. 

    Suppose that at some stage of this construction, we have a polynomial sequence $u \in \Poly(\Z, G_\bullet)$ such that 
    \[
    u(h) \Gamma = g_j(h) \Gamma 
    \]
    for all $h \in \Z$, and with Taylor expansion 
    \[
    u(h) = b_0 b_1^{\binom{h}{1}} \cdots b_s^{\binom{h}{s}} 
    \]
    where $b_i \in G\si$. 
    Fix $i = 0, \dots, s$, and choose $\gamma_i \in \Gamma\si$ such that $b_i \gamma_i^{-1} \in K\si$. 
    Define 
    \[
    u'(h) := u(h) \gamma_i^{-\binom{h}{i}}.  
    \]
    Since $h \mapsto \gamma_i^{-\binom{h}{i}}$ is a polynomial sequence, and $\Poly(\Z, G_\bullet)$ is a group by \citep[Proposition 6.2]{green2007quantitative}, the map $u'$ again belongs to $\Poly(\Z, G_\bullet)$.    
    Also, since $\gamma_i \in \Gamma$, we have that 
    \[
    u'(h) \Gamma = u(h) \Gamma = g_j(h) \Gamma \quad \text{for } h \in \Z.  
    \]
    Let $b_0', \dots, b_s'$ denote the Taylor coefficients of $u'$. 
    Observe that $u'(h) = u(h)$ for $h = 0, \dots, i - 1$, whereas $u'(i) = u(i) \gamma_i^{-1}$. 
    It follows that $b'_0 = b_0, b'_1 = b_1, \dots, b'_{i-1} = b_{i-1}$, and $b'_i = b_i \gamma_i^{-1} \in K\si$. 
    In particular, the coefficients of index less than $i$ are unchanged. 

    Starting with $u = g_j$ and applying this procedure iteratively for $i = 0, \dots, s$, we obtain a polynomial sequence $\wt g_j$ such that 
    \[
    \wt g_j(h) \Gamma = g_j(h) \Gamma \quad \text{for } h \in \Z,  
    \]
    and whose Taylor expansion 
    \[
    \wt g_j(h) = \wt a_{j,0} \wt a_{j,1}^{\binom{h}{1}} \cdots \wt a_{j,s}^{\binom{h}{s}}
    \]
    satisfies $\wt a_{j,i} \in K\si$ for $i = 0, \dots, s$. 

    The product $K^{(0)} \times \cdots \times K^{(s)}$ is compact.
    Therefore, after passing to a subsequence of $(\wt g_j)_{j \in \N}$, we can assume that 
    \[
    \wt a_{j,i} \to a_i \in G\si 
    \quad \text{ for } i = 0, \dots, s, \ \text{ as } j \to \infty. 
    \]
    Define 
    \[
    g(h) := a_0 a_1^{\binom{h}{1}} \cdots a_s^{\binom{h}{s}}. 
    \]
    By the converse part of \cref{lem:taylor_expansion}, $g \in \Poly(\Z, G_\bullet)$. 
    For each fixed $h \in \Z$, by continuity of the multiplication in $G$, we get that  
    \[
    \wt g_j(h) = \wt a_{j,0} \wt a_{j,1}^{\binom{h}{1}} \cdots \wt a_{j,s}^{\binom{h}{s}}
    \to 
    a_0 a_1^{\binom{h}{1}} \cdots a_s^{\binom{h}{s}} = g(h), 
    \]
    as $j \to \infty$, which proves the pointwise convergence.  
\end{proof}

\section{Nilsequences}

Having introduced polynomial sequences, we now recall the notion of a nilsequence. 
This is the type of structured function appearing in the inverse theorem for the Gowers norms. 

\begin{definition}[Nilsequence]
    A nilsequence of degree $s$ is a map $\psi \colon \Z \to \C$ of the form 
    \[
    \psi(h) = F(g(h) \Gamma),
    \]
    where $X = G/\Gamma$ is a filtered nilmanifold of degree $s$, $g \in \Poly(\Z, G_\bullet)$ is a polynomial sequence, and $F \colon X \to \C$ is a Lipschitz function. 
    The nilsequence $\psi$ is said to be of complexity $M$ if the filtered nilmanifold $X = G/\Gamma$ is of complexity $M$ and the Lipschitz norm of $F$ is at most $M$. 
\end{definition}

Here the Lipschitz norm is the usual inhomogeneous Lipschitz norm, that is, the sum of the supremum norm and the optimal Lipschitz constant. 
This requires a metric on $X = G/\Gamma$ to be fixed. 
For definiteness, we could use the metric constructed in \citep[Definition 2.2]{green2007quantitative}.
Note that if $\psi$ is a nilsequence of degree $s$ and complexity $M$, then so are its conjugate $\overline{\psi}$ and every shift $h \mapsto \psi(h + r)$ with $r \in \Z$, since conjugation preserves the Lipschitz norm and shifting a polynomial sequence yields again a polynomial sequence. 

The following result connects nilsequences back to the nilsystems used throughout this thesis. 
It says that nilsequences of degree $s$ can be realized as linear nilsequences, that is, as a continuous function evaluated along an orbit of an $s$-step nilsystem. 
For technical reasons, we state this result for continuous maps instead of the usual Lipschitz maps.   
A proof can be found in \citep[Chapter 14]{host2018nilpotent}. 

\begin{lemma}
    \label[lemma]{lem:linearising_nilsequence}
    Let $\psi(h) = F(g(h) \Gamma)$ be a sequence, where $X = G/\Gamma$ is a filtered nilmanifold of degree $s$, $g \in \Poly(\Z, G_\bullet)$, and $F \colon X \to \C$ is a continuous function.   
    Then there exists an $s$-step nilsystem $(\wt X = \wt G/ \wt \Gamma, \wt T)$, a point $\wt x \in \wt X$, and a continuous function $\wt F \colon \wt X \to \C$ such that 
    \[
    \psi(h) = \wt F(\wt T^h \wt x), 
    \]
    for all $h \in \Z$. 
\end{lemma}

\section{The Gowers norms and the inverse theorem}
\label[appendix]{sec:gowers_norms}

Let $s \geq 1$, let $G$ be a finite abelian group which we write additively, and let $f \colon G \to \C$ be a function.  
The Gowers uniformity norm is defined by 
\[
\| f \|_{U^s(G)}^{2^s} := \frac{1}{|G|^{s+1}} \sum_{x, h_1, \dots, h_s \in G} \prod_{\omega \in \{0,1\}^s} \mcc^{|\omega|} f(x + \omega_1 h_1 + \dots + \omega_s h_s),
\] 
where $\mcc$ denotes complex conjugation, and $|\omega| = \omega_1 + \dots + \omega_s$ for $\omega = (\omega_1, \dots, \omega_s) \in \{0, 1\}^s$. 

The version of the inverse theorem that we need, however, concerns functions defined on the finite intervals $[N] := \{1, \dots, N\}$, where $N \geq 1$. 
In this case, for $s \geq 1$ and $f \colon [N] \to \C$, the Gowers norm is defined by 
\[
\| f \|_{U^s[N]} := \| \wt f \|_{U^s(\Z/\wt N\Z)}/ \| \ind{[N]} \|_{U^s(\Z/\wt N\Z)},  
\]
where $\wt N \geq 2^s N$ is any integer and $\wt f$ is the extension of $f$ by zero outside of $[N]$. 
This definition does not depend on the specific choice of $\wt N$, so we may simply choose $\wt N = 2^s N$.  

We say that a function $f \colon [N] \to \C$ is $1$-bounded if $|f(h)| \leq 1$ for all $h \in [N]$, and we write 
\[
\E_{h \in [N]} f(h) := \frac{1}{N} \sum_{h \in [N]} f(h).  
\]

We can now state the inverse theorem for the Gowers $U^{s+1}[N]$-norm, proved by \citet{green2012inverse}. 

\begin{theorem}[Inverse theorem for $U^{s+1}\lbrack N \rbrack$-norm]
    \label{thrm:inverse_theorem}
    Let $s \geq 1$, let $N \geq 1$, and let $0 < \delta \leq 1$. 
    Suppose that $f \colon [N] \to \C$ is a $1$-bounded function with $\| f \|_{U^{s+1}[N]} \geq \delta$. 
    Then there exist constants $M = M(s, \delta) \geq 1$ and $c = c(s,\delta) > 0$, and a nilsequence $\psi$ of degree $s$ and complexity $M$ such that 
    \[
    | \E_{h \in [N]} f(h) \overline{\psi(h)} | \geq c(s, \delta). 
    \]
\end{theorem}

We will not use the inverse theorem directly. 
Instead, we will use a well-known consequence of it, called the arithmetic regularity lemma, due to \citet{green2010arithmetic}. 
The following formulation is \citep[Proposition 2.7]{green2010arithmetic}.
By a growth function, we mean a monotone increasing function $\mcf \colon \R^+ \to \R^+$ such that $\mcf(M) \geq M$ for all $M \in \R^+$.  

\begin{theorem}[Arithmetic regularity lemma]
    \label{thrm:arithmetic_regularity}
    Let $f \colon [N] \to [0,1]$, let $s \geq 1$, let $\epsilon > 0$, and let $\mcf \colon \R^+ \to \R^+$ be a growth function.
    Then there exist $M = M(s, \epsilon, \mcf) \geq 1$ and a decomposition 
    \[
    f = \psi + r + u, 
    \] 
    where $\psi, r, u \colon [N] \to [-1, 1]$ are functions such that: 
    \begin{enumerate}[(i)]
    \item $\psi$ is a nilsequence of degree $s$ and complexity $M$,  
    \item $\| r \|_{L^2} \leq \epsilon$, 
    \item $\| u \|_{U^{s+1}[N]} \leq 1/\mcf(M)$, 
    \item $\psi$ and $\psi + r$ take values in $[0,1]$. 
    \end{enumerate} 
\end{theorem}

We also need the following result, which is essentially a converse to the inverse theorem; see \citep[Theorem 1.6.12]{tao2012higher} for a proof. 

\begin{proposition}
    \label[proposition]{prop:necessity_inverse_theorem}
    Let $s \geq 1$, $M \geq 1$, and $\epsilon > 0$.
    Then there exists a constant $c = c(s, M, \epsilon) > 0$ such that if $N \geq 1$ and $f \colon [N] \to \C$ is a $1$-bounded function with 
    \[
    | \E_{h \in [N]} f(h) \overline{\psi(h)} | \geq \epsilon,
    \]
    for some nilsequence $\psi$ of degree $s$ and complexity $M$, then $\| f \|_{U^{s+1}[N]} \geq c$.  
\end{proposition}

We use the arithmetic regularity lemma to prove the following result.
It says that if a function $f \colon [N] \to \C$ is anti-uniform, meaning that its correlation with any uniform function is small, then $f$ is close in $L^2$ to a nilsequence.

\begin{lemma}
    \label[lemma]{lem:antiuniform_approx_nilsequence}
    Let $s \geq 1$, $N \geq 1$, $C > 0$, and let $f \colon [N] \to \C$ be a $1$-bounded function such that 
    \[
    | \E_{h \in [N]} f(h) \overline{g(h)} | \leq C \left\|g\right\|_{U^{s+1}[N]}
    \]
    for every $1$-bounded function $g \colon [N] \to \C$. 
    Then for every $\epsilon > 0$ there exists a $1$-bounded nilsequence $\psi$ of degree $s$ and complexity $M = M(s,\epsilon, C)$, such that 
    \[
    \E_{h \in [N]} |f(h) - \psi(h)|^2 \leq \epsilon.
    \]
\end{lemma}

\begin{proof}
    Denote the $L^2$-norm on $[N]$ by $\| g \|_{L^2} := (\E_{h \in [N]} |g(h)|^2)^{1/2}$, and let $\langle \cdot, \cdot \rangle_{L^2}$ be the associated inner product.  

    Fix $\epsilon > 0$. 
    Choose a growth function $\mcf \colon \R^+ \to \R^+$ such that 
    \[
    \mcf(M) > \frac{1}{c(s, M, \epsilon/100)} + \frac{100 C}{\epsilon} 
    \]
    for each $M \geq 1$, where $c(s, M, \epsilon)$ is the constant from \cref{prop:necessity_inverse_theorem}.  

    By taking the positive and negative parts of the real and imaginary parts of $f$, we can write $f$ as a linear combination of four functions taking values in $[0, 1]$. 
    Applying \cref{thrm:arithmetic_regularity} to these four functions separately, we obtain $M = M(s, \epsilon, C) \geq 1$, and a decomposition 
    \[
    f = \psi_0 + r + u 
    \]
    where $\psi_0, r, u \colon [N] \to \C$ are all bounded in magnitude by $4$, $\psi_0$ is a nilsequence of degree $s$ and complexity $M$, $\| r \|_{L^2} \leq \epsilon/100$, and $\| u \|_{U^{s+1}[N]} \leq 1/\mcf(M)$.  

    Since $f - \psi_0 = r + u$, we have that 
    \[
    \| f - \psi_0 \|_{L^2}^2 
    = |\langle f - \psi_0, r + u \rangle_{L^2}| 
    \leq |\langle f, u \rangle_{L^2}| + |\langle \psi_0, u \rangle_{L^2}| + |\langle f - \psi_0, r \rangle_{L^2}|. 
    \]
    We bound these three terms separately. 
    By the assumption on $f$, applied to the $1$-bounded function $\overline{u}/4$, we obtain
    \[
    |\langle f, u \rangle_{L^2}| \leq C \| u \|_{U^{s+1}[N]} 
    \leq C/\mcf(M) \leq \epsilon/100. 
    \]
    Moreover, since $\| u/4 \|_{U^{s+1}[N]} \leq 1/(4\mcf(M)) < c(s, M, \epsilon/100)$, it follows from \cref{prop:necessity_inverse_theorem} applied to the $1$-bounded function $\overline{u}/4$ that 
    \[
    |\langle \psi_0, u \rangle_{L^2}| \leq 4\epsilon/100.  
    \]
    Finally, since $f$ and $\psi_0$ are bounded in magnitude by $4$, by Cauchy--Schwarz it follows that 
    \[
    |\langle f - \psi_0, r \rangle_{L^2}| 
    \leq 8 \| r \|_{L^2}  \leq 8 \epsilon/100. 
    \] 
    Combining the three bounds, we see that 
    \[
    \| f - \psi_0 \|_{L^2}^2 \leq \epsilon. 
    \]
    It remains to replace $\psi_0$ by a $1$-bounded nilsequence. 
    Let $P \colon \C \to \overline{\D}, z \mapsto z/\max\{1, |z|\}$ be the radial projection onto the closed unit disk, and let  
    \[
    \psi := P \circ \psi_0.
    \] 
    Since $P$ is a contraction, $\psi$ is still a $1$-bounded nilsequence of degree $s$ and complexity $M$. 
    Moreover, since $f$ is $1$-bounded,  
    \[
    \| f - \psi \|_{L^2}^2 
    \leq \| f - \psi_0 \|_{L^2}^2 
    \leq \epsilon. 
    \]
\end{proof}

\chapter{A factor criterion}
\label[appendix]{chap:factor_criterion_appendix}

In this appendix we prove the criterion used in the proof of the weak structure theorem in \cref{chap:weak_structure_thrm} to identify one measure-preserving system as a factor of another. 
Recall from \cref{sec:measure_preserving_systems} our standing assumption that all probability spaces are standard, and that for a measure-preserving system $(X, \mu, T)$, we write $T^h f := f \circ T^h$ for the Koopman action on functions.  

We will need the following standard characterization of factors; see for example \citep[Chapter 12]{eisner2015operator}. 
This works precisely because we work on standard probability spaces. 

\begin{lemma}
    \label[lemma]{lem:factor_Linfty_homomorphism}
    Let $(X, \mu, T)$ and $(Y, \nu, S)$ be measure-preserving systems. 
    Then $(X, \mu, T)$ is a factor of $(Y, \nu, S)$ if and only if there exists a unital $\ast$-homomorphism
    \[
    \phi \colon L^\infty(X, \mu) \to L^\infty(Y, \nu) 
    \] 
    intertwining the Koopman operators $T$ and $S$, and preserving integrals, meaning that $\int_Y  \phi f \dd\nu = \int_X f \dd\mu$ for all $f \in L^\infty(X, \mu)$. 
\end{lemma}

\cref{lem:factor_Linfty_homomorphism} explains the hypothesis in the following criterion. 
The condition \eqref{eq:appendix_moment_identity} is exactly the statement that the assignment $T^h f_n \mapsto S^h \wt f_n$ respects all algebraic relations and all integrals, and therefore gives a plausible way to define a homomorphism as in the lemma. 
We give a different proof, where we realize both families of functions on a common factor of both systems, and use this to produce a factor map directly. 
For $d \geq 1$, we write $\Q(i)[z_1, \dots, z_d, \overline{z_1}, \dots, \overline{z_d}]$ for the polynomial ring over $\Q(i)$ in the $2d$ formal variables $z_1, \dots, z_d, \overline{z_1}, \dots, \overline{z_d}$. 

\begin{proposition}
    \label[proposition]{prop:polynomial_factor_criterion}
    Let $(X, \mu, T)$ and $(Y, \nu, S)$ be measure-preserving systems, and let $(f_n)_{n \geq 1}$ and $(\wt f_n)_{n \geq 1}$ be sequences in $L^\infty(X, \mu)$ and $L^\infty(Y, \nu)$ respectively, such that   
    \[
    \Span\{T^h f_n : n \geq 1, h \in \Z\} 
    \]
    is dense in $L^2(X, \mu)$.
    Suppose that for each $d \geq 1$, all $n_1, \dots, n_d \geq 1$, all $h_1, \dots, h_d \in \Z$, and every $P \in \Q(i)[z_1, \dots, z_d, \overline{z_1}, \dots, \overline{z_d}]$, we have that 
    \begin{equation}
    \label{eq:appendix_moment_identity}
    \begin{aligned}
        &\int_Y P(S^{h_1} \wt f_{n_1}, \dots, S^{h_d} \wt f_{n_d}, \overline{S^{h_1} \wt f_{n_1}}, \dots, \overline{S^{h_d} \wt f_{n_d}}) \dd\nu \\
        &\qquad = \int_X P(T^{h_1} f_{n_1}, \dots, T^{h_d} f_{n_d}, \overline{T^{h_1} f_{n_1}}, \dots, \overline{T^{h_d} f_{n_d}}) \dd\mu. 
    \end{aligned}
    \end{equation}
    Then $(X, \mu, T)$ is a factor of $(Y, \nu, S)$. 
\end{proposition}

\begin{proof} 
    First note that \eqref{eq:appendix_moment_identity} remains valid for polynomials $P$ with coefficients in $\C$. 
    Indeed, every monomial has coefficient $1 \in \Q(i)$, so by linearity of the integral, \eqref{eq:appendix_moment_identity} holds for all polynomials. 
    Next, we observe that the assignment $f_n \mapsto \wt f_n$ preserves the $L^\infty$-norm. 
    Applying \eqref{eq:appendix_moment_identity} with $d = 1$, $h_1 = 0$, and with $P(z,\overline{z}) = z^p \overline{z}^p$ gives $\| \wt f_n \|_{L^{2p}(\nu)} = \| f_n \|_{L^{2p}(\mu)}$. 
    Letting $p \to \infty$, we obtain 
    \[
    \| \wt f_n  \|_{L^\infty(\nu)} = \| f_n \|_{L^\infty(\mu)} =: R_n
    \] 
    for all $n \geq 1$. 
    Fix measurable representatives of $f_n$ and $\wt f_n$ that are bounded everywhere by $R_n$. 

    We now construct a common factor of both systems.  
    Let $\overline{\D}_R := \{z \in \C : |z| \leq R\}$, and consider the product 
    \[
    Z := \prod_{n \geq 1} \prod_{h \in \Z} \overline{\D}_{R_n}. 
    \]
    Then $Z$ is compact and metrizable. 
    We denote an element in $Z$ by $z = (z_{n,h})_{n \geq 1, h \in \Z}$. 
    Define the shift  
    \[
    \sigma \colon Z \to Z, \quad (z_{n,h})_{n \geq 1, h \in \Z} \mapsto (z_{n,h+1})_{n \geq 1, h \in \Z}, 
    \]
    which is a homeomorphism. 
    Because $Z$ is a countable product of second countable spaces, the Borel $\sigma$-algebra on $Z$ is generated by the coordinate projections $\pi_{n,h} \colon Z \to \C, z \mapsto z_{n,h}$. 

    Define the measurable maps 
    \[
    \phi \colon X \to Z, \quad x \mapsto (f_n(T^h x))_{n \geq 1, h \in \Z},   
    \]
    and 
    \[
    \wt \phi \colon Y \to Z, \quad y \mapsto (\wt f_n(S^h y))_{n \geq 1, h \in \Z}.   
    \]
    Then $\phi \circ T = \sigma \circ \phi$, and $\wt \phi \circ S = \sigma \circ \wt \phi$. 
    Consequently, the pushforward measures 
    \[
    \lambda := \phi_* \mu, \quad \wt \lambda := \wt \phi_* \nu 
    \]
    are $\sigma$-invariant Borel probability measures on $Z$. 
    Hence $(Z, \lambda, \sigma)$ and $(Z, \wt \lambda, \sigma)$ are measure-preserving systems and $\phi$ and $\wt \phi$ are factor maps. 

    We claim that $\lambda = \wt \lambda$.  
    Let $\mca \subseteq C(Z)$ be the unital $\ast$-subalgebra generated by the coordinate functions $\pi_{n,h}$. 
    The elements of $\mca$ are exactly the polynomials with complex coefficients in finitely many coordinates and their complex conjugates. 
    For such a polynomial, \eqref{eq:appendix_moment_identity} implies that 
    \begin{align*}
    \int_Z P(\pi_{n_1, h_1}, \dots, \overline{\pi_{n_d, h_d}}) \dd \lambda 
    &= \int_X P(T^{h_1} f_{n_1}, \dots, \overline{T^{h_d} f_{n_d}}) \dd\mu \\
    &= \int_Y P(S^{h_1} \wt f_{n_1}, \dots, \overline{S^{h_d} \wt f_{n_d}}) \dd\nu
    = \int_Z P(\pi_{n_1, h_1}, \dots, \overline{\pi_{n_d, h_d}}) \dd \wt\lambda. 
    \end{align*}
    Therefore, $\int_Z a \dd\lambda = \int_Z a \dd\wt\lambda$ for $a \in \mca$. 
    By the Stone--Weierstrass theorem, the algebra $\mca$ is uniformly dense in $C(Z)$, so it follows that 
    \[
    \int_Z g \dd\lambda = \int_Z g \dd\wt\lambda
    \quad \text{for } g \in C(Z). 
    \]
    We conclude that $\lambda = \wt \lambda$. 

    We next claim that $\phi \colon (X, \mu, T) \to (Z, \lambda, \sigma)$ is an isomorphism. 
    Since the $\sigma$-algebra on $Z$ is generated by the coordinates, the corresponding sub-$\sigma$-algebra of $X$ is generated by functions $\pi_{n,h} \circ \phi = T^h f_n$ for $n \geq 1$ and $h \in \Z$. 
    By assumption, the span of these functions is dense in $L^2(X, \mu)$. 
    It follows that the sub-$\sigma$-algebra corresponding to the factor $\phi$ equals all of $\mcx$ modulo $\mu$, so $\phi$ is an isomorphism by the correspondence between $\sigma$-algebras and factors; see \cref{sec:measure_preserving_systems}.

    Let $\phi^{-1}$ denote the measurable inverse of $\phi$. 
    The composition 
    \[
    \phi^{-1} \circ \wt \phi \colon Y \to X
    \]
    is a factor map, so $(X, \mu, T)$ is a factor of $(Y, \nu, S)$ as claimed. 
\end{proof}

\printbibliography[heading=bibintoc]

\end{document}